\documentclass[aap]{imsart}

\input{AOAP_style}

\begin{document}

\begin{frontmatter}
\title{Invasion Percolation on Power-Law Branching Processes}
\runtitle{Invasion Percolation on Power-Law Branching Processes}

\begin{aug}

\author[B]{\fnms{Rowel}~\snm{G\"undlach}\ead[label=e2]{r.c.gundlach@tue.nl}}
\and
\author[B]{\fnms{Remco}~\snm{van der Hofstad }\ead[label=e3]{r.w.v.d.hofstad@tue.nl}}

\address[B]{Mathematics and Computer Science,
Eindhoven University of Technology\printead[presep={,\ }]{e2,e3}}
\end{aug}

\begin{abstract}
We analyse the cluster discovered by invasion percolation on a branching process with a power-law offspring distribution.
Invasion percolation is a paradigm model of {\em self-organised criticality}, where criticality is approach without tuning any parameter. 
By performing invasion percolation for $n$ steps, and letting $n\to\infty$, we find an infinite subtree, called \textit{the invasion percolation cluster} (IPC). A notable feature of the IPC is its geometry that consists of a unique path to infinity (also called the \textit{backbone}) onto which finite forests are attached.
The main theorem shows the volume scaling limit of the \textit{$k$-cut IPC}, which is the cluster containing the root when the edge between the $k$-th and $(k+1)$-st backbone vertices is cut.

We assume a power-law offspring distribution with exponent $\alpha$ and analyse the IPC for different power-law regimes.
In a finite-variance setting $(\alpha>2)$ the results, are a natural extension of previous works on the branching process tree \cite{Michelen2019} and the regular tree \cite{Angel2008}. 
However, for an infinite-variance setting ($\alpha\in(1,2)$) or even an infinite-mean setting ($\alpha\in(0,1)$), results significantly change. This is illustrated by the volume scaling of the $k$-cut IPC, which scales as 
$k^2$ for $\alpha>2$, but as
$k^{\alpha/(\alpha-1)}$ for $\alpha \in (1,2)$ and exponentially for $\alpha \in (0,1)$. 
\end{abstract}



\end{frontmatter}

\tableofcontents

\section{Introduction}
\label{sec-intro}
In this paper, we investigate {\em invasion percolation} \cite{Wilkinson1983} on a \textit{branching process} (BP) or Bienaym\'e-(Galton-Watson) tree (see \cite[Chapter 8]{Karlin1975} for an overview of branching processes). Invasion percolation is a paradigmatic model of {\em self-organised criticality} \cite{Bak1988}, where critical percolation is approximated in a dynamical way, without tuning any parameter. Invasion percolation on a graph is obtained by adding independent and identically distributed (i.i.d.) edge weights to the graph, and sequentially adding the edge of minimal weight incident to the current graph. It turns out that the limsup of the edge weights of accepted edges equals the critical percolation threshold. Thus, it can be expected that invasion percolation shares many features with {\em critical percolation} \cite[Chapters 9-10]{Grimmett1999}, while at the same time producing an {\em infinite structure}, called the \textit{invasion percolation cluster} (IPC), that is usually absent in critical percolation. Therefore, it might be compared to the so-called {\em incipient infinite cluster} (IIC)\cite{Kesten1986}, the infinite percolation cluster that is at the verge of appearing at the critical point. However, the IIC is defined through a limiting argument, whereas the IPC is grown {\em dynamically}.

Invasion percolation can be systematically described by \textit{Prim's algorithm} \cite{Prim1957} that constructs the IPC on the weighted BP tree $\mathcal{T}_w$ as follows: Start at step 0 with the root $T^{(0)}(\clock)=\clock$, and at step $n+1$, $T^{(n+1)}(\clock)$ is constructed by adding the edge (and the vertex attached to it) from $\mathcal{T}$ incident to $T^{(n)}(\clock)$ with the lowest weight. We assume that edge weights are i.i.d. standard uniformly distributed, so that the lowest-weight edge is unique for all $n$ with probability 1. 
Then, the IPC is obtained by sending $n\to\infty$, i.e.
\begin{equation}
    \label{eq:def_IPC}
  T(\clock) 
    =\lim_{n\to\infty} T^{(n)}(\clock).
\end{equation}
We study the invasion percolation cluster on i.i.d. branching processes where we allow power-law offspring distributions. For simplicity, we assume that every vertex has at least one child, so that the branching process tree survives almost surely, and thus the IPC is an infinite sub-tree. The scaling behaviour of the volume of the IPC  depends sensitively on the number of finite  moments of the offspring distribution. This can already be envisioned when we note that the critical percolation threshold is the inverse of the expected offspring. In particular, the critical percolation threshold equals zero when the mean offspring is infinite.

Consider a branching process tree $\mathcal{T}$ with a power-law offspring distribution $X$, and denote the root of the tree by $\clock$.
Throughout the paper we assume that $\prob(X=0)=0$
and that there exists an $\alpha>0$ and $c_{\sss X}$ for which the tail distribution of $X$ satisfies that, as $x\to\infty$,
    \begin{equation}
    \label{eq:powerlaw}
    \prob(X>x)=
    1-F_{\sss X}(x) = c_{\sss X} x^{-\alpha}(1+o(1)).
    \end{equation}
In terms of $\alpha$, different regimes can be distinguished where the geometry of the IPC turns out to be quite different. The first regime considers $\alpha>2$, and contains the family of distributions with finite first and second moment.
This also includes offspring distributions with a thinner tail than a power-law, where we replace \eqref{eq:powerlaw} by $1-F_{\sss X}(x)\leq c_x x^{-\alpha}$ for all $x\geq 1$ and $\alpha>2$. 
See Section \ref{sec-disc-open} for a discussion.
The second regime, corresponding to $\alpha\in (1,2)$, contains families of distributions with a finite first but an infinite second moment. The last regime, for $\alpha\in(0,1)$, contains families of distribution with an infinite first moment. The latter regime is particularly interesting, since the percolation critical value equals $p_c=0$, so that a BP tree restricted to weights below $p_c$ is empty and thus the IIC is ill-defined\footnote{ See the discussion on the IIC for more details. }, while the IPC clearly exists.

Between different regimes for $\alpha$, the IPC has some important common features. It consists of a unique path to infinity, called the \textit{backbone}, to which conditionally finite forests are attached (see Lemma~\ref{lem:Toneended}). We denote this path as a sequence of vertices $(v_k)_{k\geq 0}$, where $v_0=\clock$. 
We define the \textit{future maximum weight} (FMW) process $(W_k)_{k\geq 0}$ by letting $W_k$ equal the largest weight that is added after the $k$-th backbone vertex (cf. \eqref{Wk-process}). Previous work on the regular tree \cite{Angel2008}, Poisson weighted infinite tree \cite{Addario-Berry2012} and the BP tree with finite variance \cite{Michelen2019} also observed the existence of a unique path to infinity. This specific geometry is key to our analysis.

One can view the IPC on a BP tree as random disorder (the edge weights) in a random environment (the BP tree). In such problems, one can either fix the random environment via an almost sure realisation or average out over it.
These regimes are also referred to as the \textit{quenched} and \emph{annealed} regimes, respectively. Distinguishing between regimes is crucial as there are many results known to be different between quenched and annealed settings, see for example \cite{Arous2005} and the references therein.
Our analysis regards the \textit{annealed} regime, that allows us to condition on specific properties of the BP tree, making specific calculations, such as the extinction probability of a percolated BP tree, easier. This in contrast with for example \cite{Michelen2020}, where the quenched survival probability of Bernoulli percolated trees is discussed.

\paragraphi{Main innovations in this paper}
We analyse the IPC from a local perspective and therefore define the $k$-cut IPC, $\M_k$, which is the cluster containing the root when the edge between the $k$-th and $(k+1)$-st backbone vertex is removed (Definition \ref{def:k-cut_IPC}).
Therefore, also  
\begin{equation}
 T(\clock) = \lim_{k\to\infty} \M_k(\clock),
\end{equation}
and $\M_k(\clock)$ can be seen as an alternative construction of the IPC.
In the following, we outline the innovations made to formalise our scaling results for the volume of $\M_k(\clock)$, denoted by
\begin{equation}
\label{eq:def_M_k_volume}
    |\M_k(\clock)| = M_k.
\end{equation}
First, Theorem~\ref{thm:constrT} shows that there exists an alternative regime-free algorithm to construct the IPC via the $k$-cut IPC. Rather than considering Prim's algorithm, one sequentially adds backbone vertices with their respective attached conditionally finite forests. 

This was also discussed in \cite[Proposition 2.1]{Angel2008} for fixed $k$ and the method can be broken down into deriving the evolution of the future maximum weight sequence according to a Markov chain \eqref{eq:def_thm_wk},
the offspring distribution of backbone vertices \eqref{eq:def_thm_DBk} and the degree distribution of the vertices in the finite attached forests \eqref{eq:def_tilde_X}. 

Differences between power-law regimes only become apparent when we let $k\to\infty$. We start with the FMW sequence $(W_k)_{k\geq 0}$. For $\alpha>2$ it has been established that, under the appropriate inverse linear scaling, $(W_k)_{k\geq 0}$  converges to a  lower envelope process, c.f. \cite[Corollary 6.3]{Michelen2019}.
This result is extended to general $\alpha$. For $\alpha\in(1,2)$, the scaling remains the same, but the limiting stochastic process changes to a process that we call the $\alpha$-enhanced lower envelope process (Definition \ref{def:alphaLEP}). For $\alpha\in(0,1)$, the changes are more drastic, in that $M_k$ grows exponentially and the limiting stochastic process is deterministic. These results are formalised in Theorem~\ref{thm-subthmWk}.

Next, we focus on the degree sequence of the backbone vertices, which are special as they are conditioned on having exactly one path to infinity.  We show that $(D_{v_k})_{k\geq 0}$ converges to a size-biased version of $X$ (Definition \ref{def:size-biased}), which exists for $\alpha>2$ and $\alpha\in(1,2)$, while for $\alpha \in(0,1)$ an additional exponential scaling is necessary.
These results are formalised in Theorem~\ref{thm-subthmDBK}.

With the backbone degrees formalised, we consider the volume of the finite forest attached to a specific backbone vertex.
For $\alpha>2$ and $\alpha\in(1,2)$, the main contribution to $M_k$ is from unusually large forests (those occurring with probability $O(k^{-1})$). These are of size $k^2$ and $k^{\alpha/(\alpha-1)}$ respectively.
For $\alpha\in(0,1)$, an arbitrary forest is of exponential size in $k$. These results are formalised in Theorem~\ref{thm:Cs_scaling}.

Finally, we combine these ingredients to formalise the volume scaling of the $k$-cut IPC. 
For $\alpha>2$ and $\alpha\in(1,2)$, we show that only a finite number of exceptionally large forests are found in the $k$-cut IPC. As a result, one would expect this to be a Poisson process, but due to the additional randomness of the future maximum weights, $(M_{\lceil kt\rceil})_{t\geq 0}$ converges to an inhomogenous Cox process under the appropriate scaling. Therefore, $M_k$ is of the same order as that of the exceptionally large forests. For $\alpha\in(0,1)$,
 $M_{k}$  is driven by forests at depth close to $k$, and we investigate $(M_{k-\ell})_{\ell\geq 0}$ instead, which scales exponentially. These results are formalised in Theorem~\ref{thm-main}. 

\paragraphi{Notation} We use the standard notation of $\xrightarrow{d}$, $\xrightarrow{\prob}$ for convergence in distribution and probability respectively. We denote convergence in the $J_1$-topology as $\xrightarrow{\mathcal{D}}$, also called convergence of c\`adl\`ag paths in the Skorokhod topology $D[\varepsilon,\infty)$, \cite[Chapter 3]{Billingsley1999}. We use the Bachmann-Landau notation $\Theta(\cdot)$, $O(\cdot)$, $o(\cdot)$ for asymptotic behaviour.
We use the term \textit{with high probability} (w.h.p.) to denote an event $\mathcal{E}_k$ that occurs with probability converging to one as $k\to\infty$. We finally write  $[n]=\{1,\ldots,n\}$ for an integer $n$.

\paragraphi{Organisation of this section}
The rest of this section is organised as follows. In Section~\ref{sec:MainResult}, we present our main results on the scaling limit of $(M_k)_{k\geq 0}$. In Section~\ref{sec:Overview}, we give an overview of the proof by presenting the four key ingredients in it, stated as formal results that are also of independent interest.
In Section~\ref{sec-disc-open}, we give an overview of the relevant literature and discuss how our results fit in.

\subsection{Main results}
\label{sec:MainResult}
In this section, we state our main result, which concerns the scaling limit of the volume of the $k$-cut IPC, $(M_{k})_{k\geq 0}$, as defined in \eqref{eq:def_M_k_volume}. We formally define the $k$-cut IPC as follows:
\begin{definition}[$k$-cut IPC]
    \label{def:k-cut_IPC}
    Let $T(\clock)$ denote the IPC on a weighted BP tree $\mathcal{T}$ with arbitrary offspring distribution, starting from root vertex $\clock$. Here, the edge-weights are i.i.d. and standard uniformly distributed. 
    Recall that $v_k$ denotes the $k$-th backbone vertex. Then we define the $k$-cut IPC $\M_k(\clock)$ as the sub-tree of $T(\clock)$ that consists of the vertices \begin{equation}
         \{u\in T(\clock):\ \text{path from $\clock\to u$ does not contain $v_{k+1}$}\},
    \end{equation}
    and the corresponding edges of $T(\clock)$ between vertices in $\M_k(\clock)$.
\end{definition}
In words, $\M_k(\clock)$ denotes the connected component that includes the root, after removing the edge between the $k$-th and the $(k+1)$-st backbone vertex. We choose to focus on $\M_{k}(\clock)$, as it has many convenient structural properties, while still having the same limit as $T^{(k)}(\clock)$.
Since there is a unique path to infinity, $M_k=|\TMOB_{k}(\clock)|<\infty$ for fixed $k$, almost surely, and the main results of this paper are devoted to a detailed characterisation of the scaling limit of $M_{k}$ as a stochastic process. For $\alpha>2$ or $\alpha\in(1,2)$, $(M_{\lceil kt\rceil})_{t\geq 0}$ converges in the $J_1$-topology, under the proper scaling, to an inhomogenous Cox process, a generalisation of the inhomogenous Poisson process. 
 Consider a finite intensity measure $\lambda(\cdot):(0,\infty)^2\to \mathbb{R}$, then, for disjoint sets 
 $A_1,\ldots, A_n \subset (0,\infty)^2$  and integers $(m_j)_{j=1}^n$, we define  the inhomogenous Poisson process by
    \begin{equation}
    \label{eq:ihpp}
        \prob(N(A_i) =m_i, i\in[n] )
        =
        \prod_{i=1}^n 
        \e^{-\lambda(A_i)}
        \frac{ \lambda(A_i)^{m_i}}{m_i!},
    \end{equation}
    where $N(A)$ is the number of points in $A$. 
    See also \cite[Chapter 3]{Last2017}.
The inhomogenous Cox process extends the inhomogenous Poisson process by allowing for a \emph{random intensity measure} $\xi$. In other words, the inhomogenous  Cox process is a \emph{doubly stochastic process}, where randomness comes from two sources: the (random) inhomogenous Poisson process that is also driven by a random intensity measure:

\begin{definition}[Inhomogeneous Cox process]
\label{def:spatialCox}
Consider an a.s. finite random intensity measure $\xi$. Then we define a Cox process $\mathcal{C}$ with intensity measure $\xi$, such that
$\{\mathcal{C} \mid \xi=\lambda\}$ is an inhomogenous Poisson process as defined in \eqref{eq:ihpp} with intensity measure $\lambda$ \cite[Definition 13.5]{Last2017}. 
\end{definition}

For $\alpha\in(0,1)$, we identify the scaling limit of $(M_{k-\ell})_{\ell=0,1,...}$, which converges in the product-topology to a discrete random process.
Our main result is the following theorem:
\begin{theorem}[Volume growth of the $k$-cut IPC]
\label{thm-main}
The stochastic process $(M_{k})_{k\geq 0}$ as defined in \eqref{eq:def_M_k_volume} satisfies the following three scaling limit results in the three different regimes, as $k\rightarrow \infty$:
\begin{description}
\item[Finite-variance offspring.]
Fix $\alpha>2$
Consider the random measure $\lambda_\alpha$ as given in \eqref{eq:def_lambda_PPP} below, and the inhomogenous Cox process with intensity $(\Pi_{\lambda_\alpha}(x,s))_{x\geq 0, s\in(0,t)}$. Then
\begin{equation}
	\label{scaling-limit-fin-var}
	(k^{-2}M_{\lceil kt\rceil})_{t> 0}\convdJ \Bigg(
	\int_{0}^t\int_{0}^\infty x \Pi_{\lambda_\alpha}(\dif x,\dif s)
	\Bigg)_{t> 0}
	;
\end{equation}
	
\item[Finite-mean, infinite-variance offspring.]
Fix $\alpha \in(1,2)$.
Consider the random measure $\lambda_\alpha$ as given in \eqref{eq:def_lambda_PPP} below, and the inhomogenous Cox process with intensity $(\Pi_{\lambda_\alpha}(x,y))_{x\geq 0, s\in(0,t)}$. Then
\begin{equation}
	(k^{-\alpha/(\alpha-1)}M_{\lceil kt\rceil})_{t> 0}\convdJ
	\Bigg(
	\int_{0}^t\int_{0}^\infty x \Pi_{\lambda_\alpha}(\dif x,\dif s)
	\Bigg)_{t> 0};
\end{equation}
\item[Infinite-mean offspring.]
Fix $\alpha\in(0,1)$.
Consider the future maximum weight process $(W_k)_{k\geq 0}$, formalised in \eqref{Wk-process}, and the discrete-time stochastic process $(Z_\ell)_{\ell=0,1,...}$ as defined in \eqref{eq:def:Z_n_alpha_in_(0,1)}. Then
\begin{equation}
	(M_{k-\ell}W_{k}^{\frac{\alpha}{1-\alpha}})_{\ell\geq 0} \xrightarrow{d} 
   (Z_\ell)_{\ell\geq 0}.
\end{equation}
\end{description}
Here, the convergence $\xrightarrow{\mathcal{D}}$ is in the Skorokhod $J_1$-topology for $\alpha>2$ and $\alpha\in(1,2)$, and the convergence $\xrightarrow{d}$ is in distribution (in the product topology), for $\alpha\in(0,1)$.
\end{theorem}

Theorem~\ref{thm-main} gives a highly-detailed scaling limit result for $(M_k)_{k\geq 0}$. This turns out to be technically less challenging than  working with the total mass of vertices in the $k$-neighbourhood of the IPC given by

\begin{equation}
\label{eq:def_Gamma}
    \T_k(\clock) =  \{u\in T(\clock) \mid \text{height of $u$} < k\}.
\end{equation}
However, we conjecture that our result implies that $ |\T_{k}(\clock)|$ has similar scaling properties as $M_{k}$:

\begin{conjecture}[Volume growth of the $k$-neighbourhood of the IPC]
\label{cor-vol-growth}
Under the conditions of Theorem~\ref{thm-main}, $\T_k(\clock)$, as defined in \eqref{eq:def_Gamma}, has the same volume scaling as that of $\M_k(\clock)$. This means there exists a regime-dependent sequence $x^{(\alpha)}_k$ and a non-degenerate limiting random process $(K^{(\alpha)}(t))_{t>0}$, such that, for $k\to\infty$,
\begin{equation}
    \Big(x_{k}^{(\alpha)} \ |\T_{\lceil kt \rceil }(\clock)|\Big)_{t>0} 
    \xrightarrow{\mathcal{D}} \big(K^{(\alpha)}(t)\big)_{t>0},
\end{equation}
where $x^{(\alpha)}_k$ is the scaling of $M_k$ as derived in Theorem~\ref{thm-main}.
\end{conjecture}
We provide some intuition for this conjecture in Section~\ref{sec-disc-open}.

\subsection{Overview of the proof and additional results}
\label{sec:Overview}
The proof of Theorem~\ref{thm-main} is performed in four steps, as outlined in the main innovations.
Below, we formalise our theorems on 
\begin{itemize}
    \item the geometry of the IPC;
    \item the scaling of the future maximum weights;
    \item the scaling of the backbone degrees;
    \item the scaling of the finite forests attached to the backbone.
\end{itemize}
We also introduce the relevant terminology and state our auxiliary results, which are formalised in the subsequent paragraphs. Even though their main function is to prove Theorem~\ref{thm-main}, they can be regarded as results in their own merit.

\paragraphi{Geometry of the $k$-cut IPC}
We start by describing a convenient algorithm that constructs the $k$-cut IPC.
We introduce some key properties for this algorithm with respective proofs. \\
First, the IPC consists of a unique path to infinity consisting of the vertices $(v_k)_{k\geq 0}$, see Lemma~\ref{lem:Toneended}, also referred to as the \textit{backbone}. We call $v_k$ the $k$-th backbone vertex, where $v_0=\clock$, and
we denote the weights along the backbone by $(\beta_k)_{k\geq 1}$. 
Then we define the \textit{future maximum weight} (FMW) process as
	\begin{equation}
	\label{Wk-process}
	W_k =\max_{i>k} \beta_i,
	\end{equation}
that converges to $p_c=\expec[X]^{-1}$ from above a.s., see Lemma~\ref{lem:ConvW_kpc}. Therefore, $W_k>p_c$ for fixed $k$ and a BP tree restricted to weights below $W_k$ has a positive probability of survival. We denote this probability by $\theta(W_k)$ and $\eta(W_k)=1-\theta(W_k)$. 
The $(W_k)_{k\geq 0}$ process is key to our analysis and is frequently conditioned upon below. Therefore, we often abbreviate 
\begin{equation}
\label{eq:def_prob_k}
    \prob_k(\cdot) = \prob(\ \cdot \mid W_k=w_k) \qquad  \expec_k[ \cdot  ] = \expec[\ \cdot  \mid W_k=w_k].
\end{equation}
Lemma~\ref{lem:Toneended} also implies that $\M_k(\phi)$ is almost surely finite for fixed $k$. In the following theorem we argue that $\M_k(\phi)$ can be constructed within a finite number of steps:
\begin{theorem}[Construction algorithm $k$-cut IPC]
\label{thm:constrT}
The $k$-cut IPC consists of a path of $k+1$ backbone vertices, $(v_i)_{i=0}^k$, to which a collection of finite forests are attached.
This can be constructed \textit{algorithmically} based on the following properties: \\
$(W_k)_{k\geq0}$ evolves according to a Markov chain. Define $R(p)$ as 
\begin{equation}
    \label{eq:def_R(p)_intro}
    R(p) = 
    \frac{ \expec[ X(1-p\theta(p))^{X-1}]}{\theta'(p)},
\end{equation}
then the transition kernel for $W_k$ is given by
\begin{equation}
\label{eq:def_thm_wk}
\begin{aligned}
\prob(W_{k+1} = W_k| W_k) &= 1-R(W_k)\theta(W_k),\\
	\prob(W_{k+1} < u \mid W_k) &= R(W_k) \theta(u),  \qquad u\in[0,W_k],
\end{aligned}
\end{equation}
and $\prob(W_0<p) =\theta(p)$.
Then, given $W_k=w_k$, the degree of the $k$-th backbone vertex has distribution
\begin{equation}
\label{eq:def_thm_DBk}
	\prob_k(D_{v_k}=l)=
	\frac{l(1-w_k\theta(w_k))^{l-1}
		\prob(X=l)}{
		\expec_k[ X(1-w_k\theta(w_k))^{X-1}]}.
	\end{equation}
On each of the $D_{v_k}-1$ edges that do not continue the path to infinity, an independent BP tree is attached with offspring distribution $\tilde X^{(k)}$ that, conditioned on $W_k=w_k$, is empty with probability $1-w_k$ and has distribution 
\begin{equation}
\label{eq:def_tilde_X}
    \prob_k(\tilde X^{(k)}  = \ell) = \eta(w_k)^{\ell-1}\sum_{x\geq \ell} \binom{x}{\ell} w_k^\ell (1-w_k)^{x-\ell}\prob(X=x),
\end{equation}
with probability $w_k$.
\end{theorem}

\begin{figure} 
    \centering
   \includegraphics{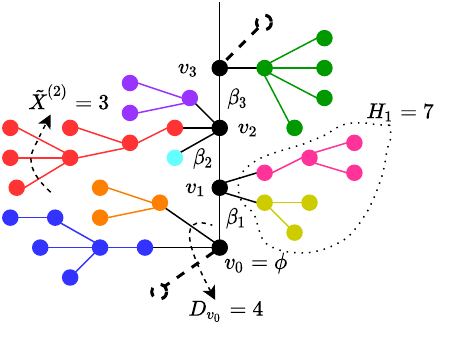}
    \caption{Example of a realisation of the $k$-cut IPC with the relevant definitions of objects discussed in Theorem~\ref{thm:constrT} and \eqref{eq:def_H_k_1}. Different coloured trees represent independent realisations of BP trees with offspring distribution $\tilde X^{(k)}$, as defined in \eqref{eq:def_tilde_X}. $(W_k)_{k \geq 0}$ is not explicitly given in the figure, but can be constructed with \eqref{Wk-process}. For given $k$, in case one of the $D_{v_k}-1$ edges that does not continue the path to infinity is assigned a weight larger than $W_k$, it is not included in the $k$-cut IPC, but it is counted in the degree distribution $D_{v_k}$.}
    \label{fig:kcut}
\end{figure}
In Section~\ref{sec:constructionIPC}, we provide the proof of Theorem~\ref{thm:constrT}. We show \eqref{eq:def_thm_wk} and \eqref{eq:def_thm_DBk} in Section~\ref{sec:Existence_properties_BB} and \eqref{eq:def_tilde_X} in Section~\ref{sec:FiniteForest_properties_BB}. By Theorem~\ref{thm:constrT}, we also provide an insightful construction algorithm for the $k$-cut IPC, in Algorithm \ref{alg:kcutIPC}.
The results are for $k$ fixed. In the following sections we analyse the behaviour for $k\to\infty$.

\paragraphi{Scaling limit of $(W_k)_{k\geq 0}$} 
The future maximum weights are critical to our analysis, as they indicate which weights are  eventually accepted. The behaviour for $W_k$ for finite $k$ is illustrated by \eqref{eq:def_thm_wk}, which we extend here to $k\to\infty$, by formalising the convergence speed of $W_k$ to $p_c$ and the limiting stochastic process of the appropriately scaled $(W_{\lceil kt\rceil})_{k\geq 0}$ process.
We define 
\begin{equation}
    \label{eq:def_halpha}
    \halpha= \min \{ \alpha,2\},
\end{equation}
and introduce the following limiting stochastic process:
\begin{definition}[$\alpha$-Enhanced lower envelope process]
	\label{def:alphaLEP}	
	For $\alpha\in(1,2)$ or $\alpha>2$, we recall $\halpha$ from \eqref{eq:def_halpha}. For a given $\varepsilon>0$, we define the $\alpha$-enhanced lower envelope process ($\alpha$-LEP) as follows: $(L_{\halpha}(t))_{t\geq \varepsilon}$ as a continuous-time left-continuous jump process, with random starting point at time $\varepsilon$ distributed as $\GammaD(1/(\halpha-1),\varepsilon)$.
	For some $t_{0}\geq\varepsilon$, $L_{\halpha}(t_{0})$ remains constant for an $\Exp(L_{\halpha}(t_{0}))$ amount of time after which it jumps down to the random position $\Unif[0,1]^{\halpha-1}L_{\halpha}(t_{0})$. 
\end{definition}

The $\alpha$-LEP is a generalisation of the lower-envelope process, defined by  \cite[Equation (1.19)]{Angel2008}. Indeed for $\halpha=2$, $L_2(t)$ defines the known lower-envelope process.
Moreover, $L_2(t)$ can be defined on a Poisson point process, so that deriving the pointwise distribution follows intuitively. For $\alpha\in(1,2)$, there still is a connection to the Poisson point process (see Remark \ref{rem:alphaLEP_on_SPPP}), but the pointwise distribution does not follow from this observation. Instead we resort to a proof based on a differential equation and find that $(L_{\halpha(t)})_{t\geq\varepsilon}$ is pointwise gamma distributed in Lemma~\ref{lem:RE_alpha<2_gamma}.

\begin{remark}[$\alpha$-LEP on the Poisson point process]
\label{rem:alphaLEP_on_SPPP}
 We can define the $\alpha$-LEP in an alternative way using a Poisson point process $\mathcal{P}$. Consider a realisation of $\mathcal{P}$ on the first quadrant and start $L_{\halpha}(\varepsilon)$ at height $\GammaD(1/(\halpha-1),\varepsilon)$. Then, for $t>\varepsilon$, we define the next lower point by
$ (\tau,x) $, where 
$$\tau = 
\min_{u>t}\{(u,x)\in \mathcal{P}, x<L_{\halpha}(t)\},$$
and $x$ the corresponding height so that $(\tau,x)\in\mathcal{P}$.
 Then for all $t>\varepsilon$,  $L_{\halpha}(t)$ remains constant until time $\tau$
 where it then jumps to $x^{\halpha-1} / L_{\halpha}(t)^{\halpha-2} $. \\
 For $\halpha=2$, this corresponds to jumping to $x$, whereas for $\halpha\in(1,2)$, jumps are smaller.
 We refer to Figure \ref{fig:alpha-LEP} for an illustration of different realisations of sample paths of the $\alpha$-LEP for different values of $\alpha$.
\end{remark}
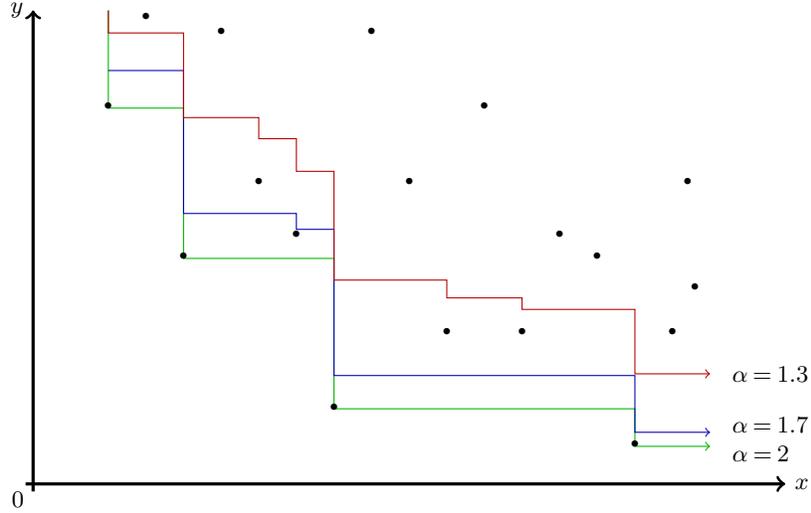
\begin{figure}
    \centering
       
\begin{tikzpicture}
  \draw[->, very thick] (-0.1, 0) -- (10, 0) node[right] {$x$};
  \draw[->,very thick] (0, -0.1) -- (0, 6.3) node[left] {$y$};
  \node at (-0.2,-0.2) {$0$};
  
  \draw[->, black!30!green] (1,6.3) -- (1,5) -- (1,5) -- (2,5) -- (2,5) -- (2,3) -- (2,3) -- (4,3) -- (4,3) -- (4,1) -- (4,1) -- (8,1) -- (8,1) -- (8,0.5) -- (8,0.5) -- (9,0.5) ;
  \node at (9.675,0.4) {$\alpha=2$};
  
  \draw[->, black!30!blue]  ( 1 , 5.5 )--   ( 1.5 , 5.5 ) -- ( 1.5 , 5.5 )--   ( 1.5 , 5.5 ) -- ( 1.5 , 5.5 )  -- ( 2 , 5.5 ) -- ( 2 , 5.5 ) --  ( 2 , 3.598 ) -- ( 2 , 3.598 )  -- ( 2.5 , 3.598 ) -- ( 2.5 , 3.598 )  -- ( 2.5 , 3.598 ) -- ( 2.5 , 3.598 ) --  ( 3 , 3.598 )--  ( 3 , 3.598 ) --  ( 3 , 3.598 ) -- ( 3 , 3.598 ) --  ( 3.5 , 3.598 ) -- ( 3.5 , 3.598 ) --  ( 3.5 , 3.387 ) -- ( 3.5 , 3.387 ) --  ( 4 , 3.387 ) -- ( 4 , 3.387 )  -- ( 4 , 1.442 ) -- ( 4 , 1.442 )  -- ( 4.5 , 1.442 ) -- ( 4.5 , 1.442 ) --  ( 4.5 , 1.442 )--  ( 4.5 , 1.442 ) --  ( 5 , 1.442 ) -- ( 5 , 1.442 ) --  ( 5 , 1.442 ) -- ( 5 , 1.442 )--   ( 5.5 , 1.442 )--  ( 5.5 , 1.442 ) --  ( 5.5 , 1.442 ) -- ( 5.5 , 1.442 )--   ( 6 , 1.442 ) -- ( 6 , 1.442 )  -- ( 6 , 1.442 )--  ( 6 , 1.442 ) --  ( 6.5 , 1.442 ) -- ( 6.5 , 1.442 ) --  ( 6.5 , 1.442 )--  ( 6.5 , 1.442 ) --  ( 7 , 1.442 )--  ( 7 , 1.442 ) --  ( 7 , 1.442 ) -- ( 7 , 1.442 ) --  ( 7.5 , 1.442 )--  ( 7.5 , 1.442 ) --  ( 7.5 , 1.442 ) -- ( 7.5 , 1.442 ) --  ( 8 , 1.442 ) -- ( 8 , 1.442 )--   ( 8 , 0.687 )  --( 8 , 0.687 ) --  ( 8.5 , 0.687 ) -- ( 8.5 , 0.687 ) --  ( 8.5 , 0.687 ) -- ( 8.5 , 0.687 )  -- ( 9 , 0.687 );
    \node at (9.8,0.787) {$\alpha=1.7$};

  \draw[->, black!30!red]
  (1,6.3)--( 1 , 6 )  -- ( 1.5 , 6 )-- ( 1.5 , 6 )   --( 1.5 , 6 ) -- ( 1.5 , 6 ) --  ( 2 , 6 ) -- ( 2 , 6 ) --  ( 2 , 4.874 )--  ( 2 , 4.874 )  -- ( 2.5 , 4.874 ) -- ( 2.5 , 4.874 )--  ( 2.5 , 4.874 ) -- ( 2.5 , 4.874 ) --  ( 3 , 4.874 )--  ( 3 , 4.874 ) --  ( 3 , 4.593 ) -- ( 3 , 4.593 )  -- ( 3.5 , 4.593 ) -- ( 3.5 , 4.593 ) --  ( 3.5 , 4.159 )-- ( 3.5 , 4.159 ) --  ( 4 , 4.159 )  --( 4 , 4.159 ) --  ( 4 , 2.712 )--  ( 4 , 2.712 ) --  ( 4.5 , 2.712 ) -- ( 4.5 , 2.712 )  -- ( 4.5 , 2.712 )  --( 4.5 , 2.712 )  -- ( 5 , 2.712 )  ( 5 , 2.712 ) --  ( 5 , 2.712 ) -- ( 5 , 2.712 )  -- ( 5.5 , 2.712 ) -- ( 5.5 , 2.712 ) --  ( 5.5 , 2.475 ) -- ( 5.5 , 2.475 ) --  ( 6 , 2.475 )--  ( 6 , 2.475 )   --( 6 , 2.475 ) -- ( 6 , 2.475 )  -- ( 6.5 , 2.475 ) -- ( 6.5 , 2.475 ) --  ( 6.5 , 2.322 ) -- ( 6.5 , 2.322 ) --  ( 7 , 2.322 ) -- ( 7 , 2.322 )  -- ( 7 , 2.322 )--  ( 7 , 2.322 )  -- ( 7.5 , 2.322 ) -- ( 7.5 , 2.322 )  -- ( 7.5 , 2.322 ) -- ( 7.5 , 2.322 )--   ( 8 , 2.322 ) -- ( 8 , 2.322 )  -- ( 8 , 1.465 ) -- ( 8 , 1.465 )  -- ( 8.5 , 1.465 ) -- ( 8.5 , 1.465 )  -- ( 8.5 , 1.465 )--  ( 8.5 , 1.465 ) --  ( 9 , 1.465 );
  \node at (9.8,1.465) {$\alpha=1.3$};
  
   \foreach \Point in {(1,5),(1.5,6.2), (2,3), (2.5,6), (3,4), (3.5,3.3), (4,1), (4.5,6), (5,4), (5.5,2), (6,5), (6.5,2), (7,3.3), (7.5,3), (8,0.5), (8.5,2), (8.8,2.6), (8.7,4)}{
    \node at \Point {\textbullet};
  }
  
\end{tikzpicture}
    \caption{Realisation of a Poisson point process, with the corresponding $\alpha$-LEP processes for different values of $\alpha$.  }
    \label{fig:alpha-LEP}
\end{figure}

For $\alpha>2$ and $\alpha\in(1,2)$ and under proper scaling, the future maximum weights converge to the $\alpha$-enhanced lower envelope process in the $J_1$-topology. For $\alpha\in(0,1)$, instead, the future maximum weights converge to a deterministic process in the $J_1$-topology. This is formalised in the following theorem:

\begin{theorem}[Scaling behaviour of the $W_k$ process]
	\label{thm-subthmWk}
The future maximum weight process $(W_{k})_{k\geq 0}$ has the following scaling limits in the three different regimes as $k\to\infty$ and for $\varepsilon>0$:
\begin{description}
	\item[\qquad Finite-variance offspring.]
	Fix $\alpha>2$.
	Let $(L_2(t))_{t\geq \varepsilon}$ denote the classical lower-envelope process, as defined in Definition \ref{def:alphaLEP}. Then
	\begin{equation}
	\label{eq:main_thm_alpha>2}
	\Bigg(\frac{k (W_{\lceil kt \rceil} -p_c)}{p_c}\Bigg)_{t\geq\varepsilon}
	\convdJ \big(L_2(t)\big)_{t\geq \varepsilon};
	\end{equation}
	\item[\qquad Finite-mean, infinite-variance offspring.]
	 Fix $\alpha\in(1,2)$. Let $(L_{\halpha}(t))_{t\geq \varepsilon}$ denote the $\alpha$-enhanced lower envelope process as described in  
	Definition \ref{def:alphaLEP}. Then
	\begin{equation}
		\label{eq:main_thm_1<alpha<2}
	\Bigg(\frac{k (W_{\lceil kt \rceil}-p_c)}{p_c}(\alpha-1) \Bigg)_{t\geq \varepsilon}
    \convdJ (L_{\halpha}(t))_{t\geq \varepsilon};
	\end{equation}
	\item[\qquad Infinite-mean offspring.]
	Fix $\alpha\in(0,1)$. Then 
	\begin{equation}
		\label{eq:main_thm_alpha<1}
	\Bigg(\Big( W_{\lceil kt\rceil}\Big)^{1/k} \Bigg)_{t\geq \varepsilon}
	\convdJ \big(\e^{-t(1-\alpha)^2/\alpha}\big)_{t\geq \varepsilon}.
	\end{equation}
\end{description}
Here, we recall the convergence $\convdJ$ in the $J_1$-topology as introduced in Theorem~\ref{thm-main}.
\end{theorem}

The results in \eqref{eq:main_thm_alpha>2} and \eqref{eq:main_thm_1<alpha<2} are proved in Section~\ref{sec:Scaling_W_k_alpha>1}, where the proof is derived from the aforementioned literature for $\alpha>2$.
The result in \eqref{eq:main_thm_alpha<1} is proved in Section~\ref{sec:Scaling_W_k_alpha_in_0,1}. 

\paragraphi{Degree of backbone vertices} The degree distribution of $v_k$, denoted by $D_{v_k}$, deviates from the distribution of $X$ as backbone vertices are conditioned on having \emph{exactly} one offspring in the unique path to infinity. 
For finite $k$, this distribution is explicit and regime free (see \eqref{eq:def_thm_DBk} for details). We continue to study its asymptotics as $k\to\infty$: 

\begin{theorem}[Scaling behaviour of the degrees along the backbone]
	\label{thm-subthmDBK}
	Let $\theta(p)$ denote the survival probability of the percolated BP tree with probability $p$, offspring distribution $X$ satisfying \eqref{eq:powerlaw} and define $\hat{\theta}(p)=p\theta(p)$ as in \eqref{eq:def_hattheta}.
	Then $D_{v_k}$ and its size-biased distribution $D^\star_{v_k}$ have the following scaling as $k\to\infty$:
	\begin{description}
		\item[\qquad Finite-variance offspring.]
		Fix $\alpha>2$. By Definition \ref{def:size-biased}, let $X^\star$ and $X^{\star\star}$ be the size-biased and doubly size-biased distributions of $X$ respectively. Then
		\begin{equation}
		D_{v_k} \xrightarrow{d} X^\star,
		\qquad \text{ and } 
		\qquad 
		D^\star_{v_k} \xrightarrow{d} X^{\star\star},
		\end{equation}
	 where $\expec[X^\star]<\infty$. 
		\item[\qquad Finite-mean, infinite-variance offspring.]
		Fix $\alpha\in(1,2)$. Let $X^\star$ be the size-biased distribution of $X$. Then
		\begin{equation}
		D_{v_k}\xrightarrow{d}  X^\star,
		\qquad
		\text{and}
		\qquad 
		\hat\theta(w_k) D^\star_{v_k} 
		\xrightarrow{d}
		\GammaD(2-\alpha,1),
		\end{equation}
		where $\expec[ X^\star]=\infty$.
		\item[\qquad Infinite-mean offspring.]
		Fix $\alpha\in(0,1)$. Then
			\begin{equation}
	\hat{\theta}(W_k) D_{v_k} \xrightarrow{d} \texttt{Gamma}(1-\alpha,1).
		\end{equation}
	\end{description}
\end{theorem}
If the edge between $v_k$ and the root of an attached tree $\tilde T^{(k)}$ has a weight larger than $W_k$, then the tree is not included in the IPC. Therefore the \textit{effective degree} of $v_k$, given by $\hat D_{v_k}=\Binom(D_{v_k},w_k)$, is also relevant for the volume scaling.

\paragraphi{Volume of conditionally finite forests}
We denote $H_k$ for the size of the finite forest attached to $v_k$, i.e.,
\begin{equation}
\label{eq:def_H_k_1}
    H_k = M_k - M_{k-1} -1.
\end{equation}
As the $k$-cut IPC discovers $k$ finite forests that consist of $\hat D_{v_k}$ trees each, we are specifically interested in the events where $H_k$ is exceptionally large. For $\alpha>2$ or $\alpha\in(1,2)$, the size of these forests is such that they occur with probability $O(k^{-1})$:

\begin{theorem}[Forest size scaling]
\label{thm:Cs_scaling}
Let $H_k$ be defined in \eqref{eq:def_H_k_1} and recall the FMW sequence $(W_k)_{k\geq 0}$ as defined in \eqref{Wk-process}. Then the following scaling limit results hold:
\begin{description}
    \item[\qquad  Finite-variance offspring.]
    Fix $\alpha>2$ and $h_\alpha(y,a)$ as given in \eqref{eq:def_h_alpha>2}. Then, for any $a>0$,
    \begin{equation}
    \label{eq:TP_scaling_finite_Variance}
   \lim_{k\rightarrow \infty} k\prob(H_k>k^2 x\mid W_k = p_c(1+a/k) )=\int_x^{\infty} h_\alpha(y,a)\dif y;
    \end{equation}
    \item[\qquad Finite-mean, infinite-variance offspring.]
   Fix $\alpha\in(1,2)$ and  $h_\alpha(y,a)$ as given in \eqref{eq:def_h_alpha(1,2)}. Then, for any $a>0$,
    \begin{equation}
     \label{eq:TP_scaling_infinite_Variance}
            \lim_{k\rightarrow \infty} k\prob(H_k>k^{\alpha/(\alpha-1)} x\mid  W_k = p_c(1+a/k))=\int_x^{\infty} h_\alpha(y,a)\dif y;
    \end{equation}
    \item[\qquad  Infinite-mean offspring.]
    Fix $\alpha\in(0,1)$ and  $h_\alpha(y)$ as given in \eqref{eq:def_h_alpha(0,1)}. Then
    \begin{equation}
    \label{eq:TP_scaling_infinite_mean}
        \lim_{k\rightarrow \infty} 
        w_k^{-\alpha/(1-\alpha)}
        \prob(H_k>w_k^{-\alpha/(1-\alpha)} x\mid W_k=w_k )=\int_x^{\infty}  h_\alpha(y)\dif y.
    \end{equation}
\end{description}
\end{theorem}

We formalise
\eqref{eq:TP_scaling_finite_Variance} in Proposition~\ref{prop-size-trees-backbone->2}, 
\eqref{eq:TP_scaling_infinite_Variance} in Proposition~\ref{prop-size-trees-backbone-(1,2)} and
\eqref{eq:TP_scaling_infinite_mean} in Proposition~\ref{prop-size-trees-backbone-(0,1)}.

\paragraphi{Completion of the proof of Theorem~\ref{thm-main}} 
The above ingredients are combined to find the volume growth of the $k$-cut IPC, which can now be expressed by the $k$ backbone vertices plus the sum of the finite forests that are attached to the first $k$ backbone vertices, i.e.,
\begin{equation}
   \label{eq:def_M_k}
    M_k = \sum_{i=0}^k (1+H_i). 
\end{equation}
The scaling behaviour of $M_k$ is given in Theorem \ref{thm-main} and we now provide some intuition on its result.
By Theorem~\ref{thm:Cs_scaling}, we know that there exist forests of size at least $k^2$ for $\alpha>2$ and $k^{\alpha/(\alpha-1)}$ for $\alpha\in (1,2)$ with probability $\Theta(1/k)$. Hence, at depth $k$, we expect to see a binomial number of forests of this exceptional size, which converges to a Poisson number in the limit. Note, however, that this result is conditional on  $(W_k)_{k\geq 0}$, which also converges in distribution as $k\to\infty$.
Hence, in the limit, the Poisson number of large forests is moderated by $(L_{\halpha}(x))_{x>0}$, which explains the doubly stochastic Poisson process, also known as a Cox process (Definition \ref{def:spatialCox}). There is no notion of rare exceptionally-large forests for $\alpha\in(0,1)$, where $M_k$ is dominated by the few forests close to $v_k$. We therefore only consider the forests at a finite distance from $k$, that all contribute a significant proportion. In the limit, this translates to a countable sum of random variables.

\subsection{Literature overview, discussion and open problems}
\label{sec-disc-open}
In this section, we give a literature overview of percolation and invasion percolation.
We illustrate our results with simulations and an example for a specific choice for the distribution of $X$. We conclude with the structure of the rest of the paper. 

\paragraphi{Percolation and the IIC}
On a conceptual level, percolation theory describes the movement of a fluid through a porous medium. Early work as \cite{Broadbent1957,Hammersley}, laid the mathematical foundation by considering percolation on lattices. An infinite number of vertices (or pores) are connected by a set of  edges (or links, bonds, paths or throats), where every edge is either open (or occupied or active) with probability $p$ or closed (or vacant or passive) with probability $1-p$.
The critical probability $p_c$ is the smallest value for $p$ for which an infinite number of vertices is connected by open edges, and was analysed further in \cite{Harris1960,Fisher1961, Kesten1980}. A summary of results on  percolation models from the sixties and seventies can be found in \cite{Book_Kesten}. 
These percolation models are now commonly referred to as \textit{static or Bernoulli percolation} models. For a more up-to-date overview of the field of static percolation, we refer to the books  \cite{Bollobas2012,Grimmett1999}.
Another structure of interest is the infinite cluster in a percolated graph at criticality, i.e. $p=p_c$. This was first formalised in \cite{Kesten1986}, where the author defined this cluster on $\mathbb{Z}^2$ in two limiting procedures. Given a fixed vertex $w$ and $p>p_c$, it turns out that the law conditioned on the event that there exists an open path from $w$ to infinity taking, $p\searrow p_c$, and the law conditioned on the event that the open cluster that contains $w$ leaves the ball of radius $n$ around $w$, taking $n\to\infty$, are the same. The cluster under this law is named the incipient infinite cluster (IIC) and has been studied in more detail and on different graphs in for example \cite{Heydenreich2014, CamesvanBatenburg2015}.

\paragraphi{Invasion percolation}
In the eighties, a new, but closely related, form of percolation arose from a physics application of the capillary displacement of fluids in porous media \cite{Chandler1982} and was shortly after redefined from a mathematical perspective in \cite{Wilkinson1983}. 
In \cite{Wilkinson1983}, the main value of interest is $p'_c$, which denotes the maximum value such that eventually no weights larger than $p'_c$ are accepted and, under weak conditions, equals $p_c$ from static percolation. The alternative interpretation, also noted in \cite{Barabasi1996}, is that this simplifies to performing Prim's algorithm \cite{Prim1957} on an infinite lattice. Invasion percolation has also drawn the attention of theoretical mathematicians, to quantify and exploit the similarities and differences of the static to the dynamical percolation model, as is done in for example \cite{Wilkinson1983, Chayes1985, Chayes1987} on $\mathbb{Z}^d$ and \cite{Nickel1983} on the Cayley tree.
Many practical and theoretical applications also arose from the invasion percolation model, see \cite{Laidlaw1993,Norris2014,Ebrahimi2010,Gabrielli2007,Gabrielli2009,Shao2009,Michelen2019} .  

As it was established that different laws to create an infinite cluster at criticality existed in Bernoulli percolation \cite{Kesten1986}, it was conjectured that the law of the invasion percolation cluster would converge to the same law as the IIC \cite{VanDerHofstad2004}. In the beginning of the new millennium, the relation of the invasion percolation cluster and incipient infinite cluster was investigated in more detail. Whereas in the aforementioned articles the link between the IIC and IPC has been established, more rigorous similarities and conjectures are derived in \cite{Jarai2003} on  $\mathbb{Z}^2$, in \cite{VanDerHofstad2004} on
 $\mathbb{Z}^d$ for $d>6$, in \cite{Addario-Berry2012} for the Poisson weighted infinite tree  and in \cite{Angel2008} for the regular tree. In the latter paper, they showed that the cluster sizes have the \textit{same} scaling, but to \textit{different} limiting objects. This difference is made more concrete by showing that their respective measures are mutually singular and that the IIC stochastically dominates the IPC. These results are shortly after extended to $\mathbb{Z}^2$, where \cite{Damron2009} shows that the measures of the IIC and IPC are also mutually singular, but according to \cite{Sapozhnikov2011} there is no stochastic domination. The IIC has been considered for BP trees as well \cite{Michelen2019_02}.
 However, it is argued in \cite{Michelen2019} that the distribution of the unique infinite paths in the IPC and IIC differ. This implies that on BP trees, there is no general domination of the IIC over the IPC.

\paragraphi{Our contribution and comparison to current literature}
The IPC on the BP tree has been analysed before in \cite{Michelen2019}, for a finite-variance offspring distribution. It was found that the backbone construction from \cite{Angel2008} naturally carries over to this case, but volume scaling has not been discussed. Our work extends \cite[Corollary 6.3]{Michelen2019}, to $\alpha\in(1,2)$ and $\alpha\in(0,1)$ in Theorem~\ref{thm-subthmWk}. Moreover, we derive structural properties of the geometry of the IPC, such as the backbone degrees and the size of the conditionally finite forests attached to the backbone. These results are mostly trivial for the regular tree, but need careful consideration when the degrees are random. Lastly, we extend the volume scaling from \cite[Theorem 1.6]{Angel2008} to BP trees, only here based on the $k$-cut IPC instead of the $k$-neighbourhood.

This work also supplies a basis for the local behaviour of the \textit{minimal spanning tree} (MST) on random graphs. It is well-known that a large class of random graphs show locally tree-like behaviour. For more detail on this, we refer to \cite{Aldous2004} or \cite{VanderHofstad2021}.
In \cite{Addario-Berry2017} the link between the local limit of the MST on the complete graph of $n$ vertices and the IPC was established. The authors first consider Prim's algorithm up to $k<n$ steps. In order to find the local limit of the MST, one can proceed in two ways. First by taking the limit of $n\to\infty$ and then $k\to \infty$, or $k$ and $n$ grow to infinity together with $k<n$. In the former case, the problem simplifies to understanding the behaviour of Prim's algorithm on the respective local limits (which coincides with invasion percolation on the limiting graph), whereas for the latter \cite{Addario-Berry2013solo} showed that this object is a combination of the IPC (first sweep) plus an additional cluster (second sweep).
The MST on the configuration model has been considered in \cite{Addario-Berry2021} when all degree equal to $3$.
It is also known from the literature that the \textit{configuration model} \cite{Bollobas2001} behaves as a BP tree in the local limit \cite[Chapter 4]{VanderHofstad2021}. A first exploration of the local limit of the MST on the configuration model has been performed in \cite{Gundlach2019}, but for the volume scaling, not much is currently known. Therefore, having an understanding of the IPC on BP trees is a significant step towards understanding the MST on the configuration model with power-law degrees.

\paragraphi{Discussion of the results}
We discuss the interpretation of the results, intuition and some take-away messages. 
The $k$-cut IPC is defined such that it at least contains a path of size $k$, on which finite forests are attached. By Theorem~\ref{thm:Cs_scaling}, we expect that for $\alpha>2$ and $\alpha\in(1,2)$, there are a finite (asymptotically Poisson) number of forests of size $k^{\halpha/(\halpha-1)}$ that dominate the volume scaling, where we recall $\halpha= \alpha \wedge 2$. For $\alpha\in(0,1)$, every forest is expected to be of an exponential size in $k$. Therefore, only the forests at a finite distance from $k$ dominate the volume scaling. We illustrate this key message in Figure \ref{fig:IPC_sketch}, where we sketch an informal picture of the $k$-cut IPC for different regimes.

\begin{figure}
    \centering
    \input{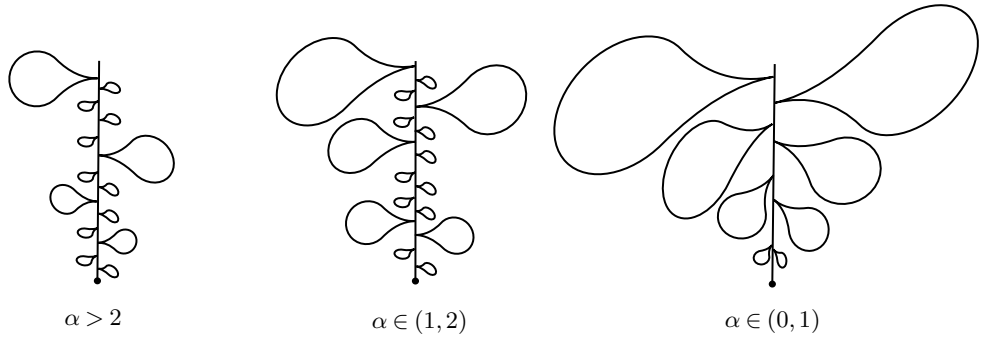}
    \caption{Rough sketch of the $k$-cut IPC for finite $k$, for all different regimes as indicated below each figure. Here, the root of the tree is at the bottom and the size of an attached forest in the figure represents the forest size. }
    \label{fig:IPC_sketch}
\end{figure}

In addition, we provide a simulation study and approximate the IPC by performing Prim's algorithm up to a finite number of steps. After $i$ steps, this cluster is denoted by $T^{(i)}(\clock)$. 
We first present the simulation results of $T^{(i)}(\clock)$. As the backbone is an asymptotic property we identify the vertex farthest away from the root and consider the path between it and the root as the (most likely) backbone. We refer to Figure \ref{fig:sim_ipc} for $T^{(i)}(\clock)$, for different values of $\alpha$.

\begin{figure}
    \centering
    \includegraphics[width=\textwidth, height=3cm]{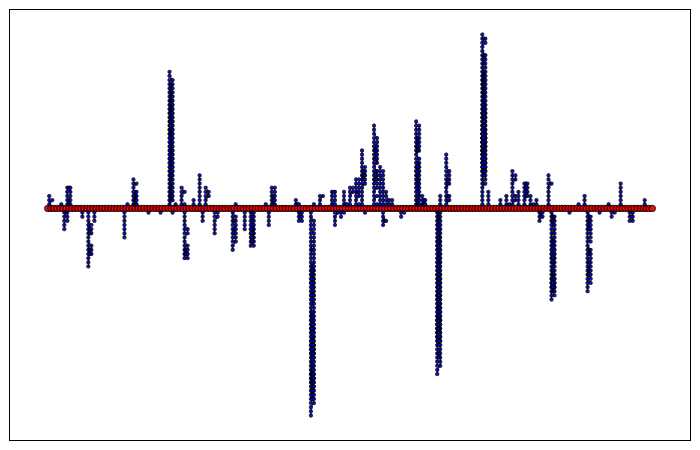}
    \includegraphics[width=0.49\textwidth, height=3cm]{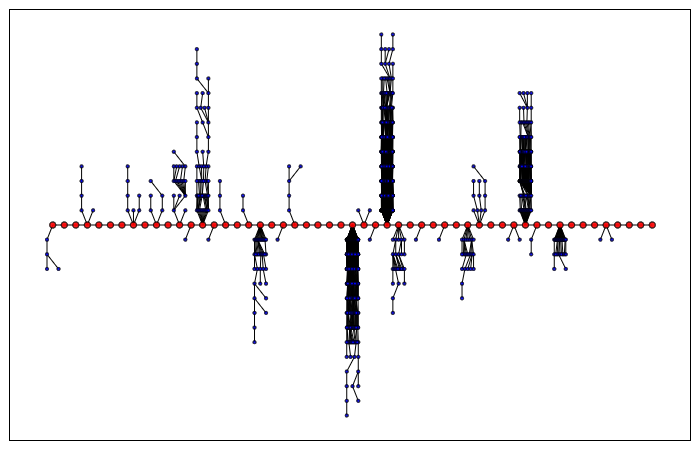}
    \includegraphics[width=0.49\textwidth, height=3cm]{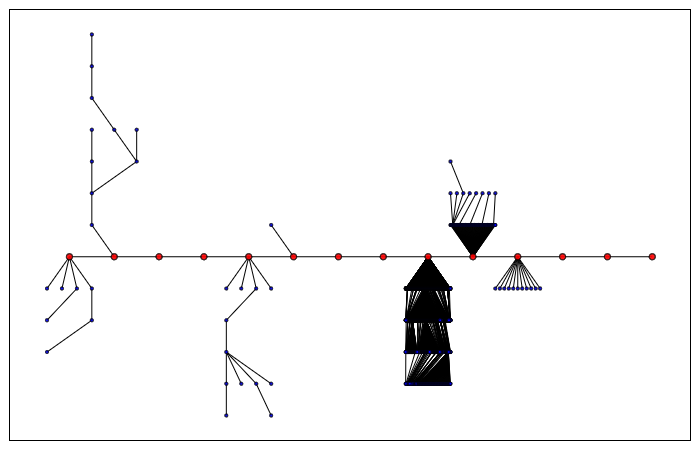}
    \caption{Layout of $T^{(i)}(\clock)$ where the root ($\clock$) is the left most vertex and $i=2000$. The offspring distributions are power laws with $\alpha=2.3$ (top figure), $\alpha=1.3$ (bottom-left figure) and $\alpha=0.7$ (bottom-right figure). In all the IPCs, the longest path from the root is marked red and presumed the most likely backbone.}
    \label{fig:sim_ipc}
\end{figure}

When we consider Figure \ref{fig:sim_ipc}, we see that, for $\alpha>2$ and $\alpha\in(1,2)$, along the backbone, some of the attached forests contribute a significant amount whereas the majority of the attached forests contributes a negligible amount. Figure \ref{fig:sim_ipc} for $\alpha\in(0,1)$ does not yet show asymptotic properties, as only certain large forests are present.
This may be due to the specific choice for $i$ not being large enough. 
While simulating $T^{(i)}(\clock)$ we keep track of the weight we add in each time step. By Lemma~\ref{lem:ConvW_kpc}, we know that the future maximum weight converges to $p_c$. However, the future maximum weight is an asymptotic property and can therefore not be simulated unambiguously. Instead, we plot the weights of each added edge. We refer to Figure \ref{fig:sim_weights} for the results.
\begin{figure}
    \centering
    \includegraphics[width=0.32\textwidth]{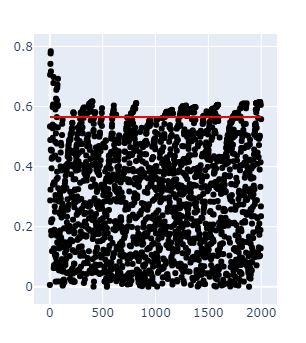}
    \includegraphics[width=0.32\textwidth]{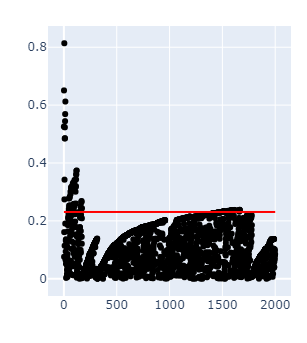}
    \includegraphics[width=0.32\textwidth]{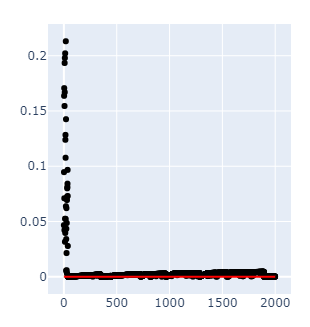}
    \caption{Scatter plot of the weight ($y$-axis) accepted at step $i$ ($x$-axis). The offspring distributions are of power law with $\alpha=2.3$ (left-hand figure), $\alpha=1.3$ (middle figure) and $\alpha=0.7$ (right-hand figure). In all figures $p_c$ is drawn with a red line.}
    \label{fig:sim_weights}
\end{figure}
We observe a wave-like behaviour of the weights, which is in line with the theoretical results. The upwards drift crosses $p_c$ as the tree restricted to weights below $p_c$ is finite a.s., and thus at some point, one is forced to accept larger weights. However, when a large-degree vertex is discovered, there are suddenly many small weights that can be added, causing the collapse of the wave. 
Finally, we note that a simulation study has its limitations. Firstly, the simulation is based on Prim's algorithm, which is only able to return the exact $k$-cut IPC after an infinite number of steps. One can opt for a simulation based on the distributions in Theorem~\ref{thm:constrT}, but this is challenging as well, as one needs the inverse of these distributions.

\paragraphi{Discussion of $\M_{k}(\clock)$ versus the IIC} As illustrated in the literature overview, the comparison of the IPC and the IIC is of great interest. 
Intuitively, the IIC can be seen as a unique path to infinity, on which critical forests are attached. In the IPC, the attached forests are barely sub-critical. In the case of the regular tree, where, for both the IPC and IIC, this path to infinity is uniformly chosen over the vertices, it follows that the IIC stochastically dominates the IPC \cite[Theorem 1.1]{Angel2008}. For the BP tree, this is not clear. For $\alpha>2$ and  $\alpha\in(1,2)$, the notion of the IIC naturally translates into a unique path to infinity on which \textit{critical} forests are attached, where the offspring distribution can have infinite variance. Also in this regime it remains unclear whether a notion of stochastic domination exists. However, and most interestingly, the IIC on a BP tree with power-law offspring distribution with exponent $\alpha\in(0,1)$ is not well-defined. In a sense, one can still take $p>p_c$ and perform percolation on $\mathcal{T}$ conditional on survival, and take a limit for $p\to p_c=0$. Then, likely, no clusters are left. Clearly, one can define the IPC on the BP tree in the power-law regime with exponent $\alpha\in(0,1)$ and by Theorem~\ref{thm-main}, this is not a trivial object. This makes the results for $\alpha\in(0,1)$ particularly interesting.

\paragraphi{Illustrating example for \mtitle{$\alpha\in(0,1)$}}
By, among others, Theorem~\ref{thm-subthmWk}, the $\alpha\in(0,1)$ regime significantly deviates from other power-law regimes. We 
illustrate this with a simple example how the exponential scaling to a deterministic limit arises in for the $(W_k)_{k\geq0}$ process. For $\alpha\in(0,1)$, we take
\begin{equation}
    G_{\sss X}(s) = \expec[ s^X ] = 1-(1-s)^\alpha.
\end{equation}
One can easily verify that this is a proper generating function that corresponds to a distribution with infinite expectation. Using the implicit expression for $\theta(p)$ \cite[Chapter 8.3]{Karlin1975}, we can solve
$ 1- \theta(p) = G_{\sss X}(1-p\theta(p))$, that implies $\theta(p)=p^{\alpha/(1-\alpha)}$. This is sufficient to describe the jump-chain of $W_k$ for fixed $k$. By \eqref{eq:def_thm_wk}, it follows that for any $p>p_c$, $R(p)\theta(p)=1-\alpha$ and for $p_c<u<p$, $R(p)\theta(u) = (1-\alpha) (u/p)^{\alpha/(1-\alpha)}$. Therefore one can write  $W_k=W_0\prod_{i=1}^{k}P_i$, where $(P_i)_{i\geq 0}$ is an i.i.d.\
sequence, with $P_i= 1$ w.p. $\alpha$ and $P_i=\Unif[0,1]^{(1-\alpha)/\alpha}$ w.p. $1-\alpha$. In this case $\log(W_k)$ can be written as an i.i.d. sum of finite expectation random variables. This implies that $\log(W_k)/k$ converges to a degenerate limit by the law of large numbers and by the continuous mapping theorem, so would $W_k^{1/k}$. These results are \textit{exact} for the specific choice of $G_{\sss X}(s)$, but offspring distributions in the same domain of attraction provide similar behaviour in the limit.

\paragraphi{Intuition for Conjecture \ref{cor-vol-growth}}
Conjecture \ref{cor-vol-growth} states that the scaling results for the $k$-cut IPC also translate to the \textit{volume} of the $k$-neighbourhood $\T_k(\clock)$
of $T(\clock)$ as defined in \eqref{eq:def_Gamma}. Naturally, it follows that $|\T_k(\clock)|\leq M_k$, but a lower bound is not evident. For $\alpha>2$ and $\alpha\in(1,2)$, Theorem~\ref{thm-main} argues that the volume of the $k$-cut IPC is driven by exceptionally large forests of size $k^2$ and $k^{\alpha/(\alpha-1)}$ respectively. Moreover, it is likely that such a forest is already found at depth $qk$, where $q\in(0,1)$.
If we take $k$ large enough, then these trees in the attached forests are close to critical in this regime, by \eqref{eq:Conditional_final_clusters_mean_a>1}.
Then, by 
\cite{Kortchemski2017}, the critical trees conditioned on having a specific mass $m$, have width $(\mathcal{W})$ and depth $(\mathcal{D})$ roughly given by 
\begin{equation}
    (\mathcal{W},\mathcal{D}) = 
     (O(m^{1/\halpha}), O(m^{(\halpha-1)/\halpha})),
\end{equation}
where we recall that $\halpha=\alpha\wedge 2$.
Suppose that we consider an exceptionally large BP tree of size $k^{\halpha/(\halpha-1)}$. Then the height of such a tree is $O(k)$, and the width is  $O(k^{1/(\halpha-1)})$. This means that at depth $qk$, with positive probability, we find a forest of size $O(k^{1/(\halpha-1)})$ that is explored up to depth $(1-q)k$ in $\T_{k}(\clock)$. 
This tree is expected to encounter  $(1-q)k\cdot O(k^{1/(\halpha-1)})= O(k^{\halpha/(\halpha-1)})$ vertices. Therefore, with positive probability there also exists   a tree of size $O(k^{\halpha/(\halpha-1)})$ in $\T_k(\clock)$ for $\alpha>2$ and $\alpha\in(1,2)$ respectively. This provides an intuition on the lower bound of $|\T_k(\clock)|$ for these regimes. \\
The $\alpha\in(0,1)$ regime is different, as here, according to \eqref{eq:Conditional_final_clusters_mean_a<1}, attached trees do not approach criticallity but remain finite a.s. The dominant contribution of the volume scaling arises from the effective degrees of the backbone vertices and scale as $	O(\theta(W_k))=O(w_k^{\alpha/(1-\alpha)})$ according to Lemma \ref{lem:convergence_hatDBK_alpha(0,1)}. This means that the tree that consists only of the backbone plus the effective degrees of the $(k-1)$-th backbone is already of size $O(w_k^{\alpha/(1-\alpha)})$, which provides a rigorous and matching lower bound for $\T_k(\clock)$ for $\alpha\in(0,1)$. 

\paragraphi{Intuition for other offspring distributions}
We consider distributions $X$ as defined in \eqref{eq:powerlaw} that are is restricted to distributions with a power-law tail with exponent $\alpha$ for $\alpha>2$, $\alpha\in(1,2)$ and $\alpha\in(0,1)$. We do not consider slowly varying fluctuations (svf), denoted by $\ell(x)$. Therefore, many of our results are tractable, making them insightful and interpertable as we illustrated before. However, allowing for svf-s will not affect the results for $\alpha>2$. Indeed, in this paper we observe that in this regime the value of $\alpha$ does not influence the scaling and only the fact that $\expec[X^2]<\infty$ is of relevance. By using Potter's bounds for $\alpha>2$ \cite[Section 1.4]{Kulik2020}, for large $x$ and $\varepsilon>0$, one can always bound $x^{-\alpha-\varepsilon}<x^{-\alpha}\ell(x)<x^{-\alpha+\varepsilon}$ and take $\varepsilon$ small enough to find similar scaling results. If $\alpha\in(1,2) $ or $\alpha\in(0,1)$, the scaling exponents are a function of $\alpha$ (c.f. Theorems \ref{thm-main} and \ref{thm-subthmWk} amongst others), and therefore Potter's bounds cannot be used directly anymore. We therefore expect that the svf-s return in the scaling results, but do not drastically impact the outcomes of this paper.

However, when one consider the boundary cases $\alpha\in\{1,2\}$, these svf-s play a key role in the scaling results. Indeed, the specific choice 
for $\ell(x)$ can cause $X$ to either have a finite or infinite second moment for $\alpha =2$ or a finite or infinite first moment for $\alpha=1$. Consider for example the random variable $X_\beta$ with distribution function $F(x) = 1-x^{-1}\log(x)^{-\beta}$. It is straightforward to verify that $1-F(x)=x^{-1}\ell(x)$ and $\expec[X_\beta] = \infty$ if $\beta\leq 1$ and $\expec[X_\beta]<\infty$ if $\beta>1$. This elementary example already opens the door for future research on the scaling results of $\alpha\in\{1,2\}$ for $\beta<1$ or $\beta>1$, and is beyond the scope of this paper.

\paragraphi{Organisation of the paper}
In Section~\ref{sec:constructionIPC}, we outline some regime-free properties of the IPC, which are then used to describe an algorithm to construct the $k$-cut IPC for finite $k$, as described in Theorem~\ref{thm:constrT}. In 
Section~\ref{sec:Scaling_W_k} we analyse the asymptotic properties of future maximum weights, as described in Theorem~\ref{thm-subthmWk}. In Section~\ref{sec:DBK} we analyse the degree distribution of backbone vertices, as described in Theorem~\ref{thm-subthmDBK}. Section~\ref{sec:TP_scale} considers how the attached finite forests on the backbone scale, as described in Theorem~\ref{thm:Cs_scaling}. These four properties are then combined in Section~\ref{sec:Combine} to prove the main result in Theorem~\ref{thm-main}.

\section{Construction of the  \mtitle{$k$}-cut IPC}
\label{sec:constructionIPC}
We have defined the IPC as the cluster discovered by Prim's algorithm after $n$ steps, where we take $n\to\infty$, as in \eqref{eq:def_IPC}. In this section we instead consider the $k$-cut IPC (Definition \ref{def:k-cut_IPC}) for finite $k$. 
In order to do so, we need some key properties that are outlined in Theorem~\ref{thm:constrT}. These properties are used for a construction algorithm that generates the $k$-cut IPC in finite time. 

In Section~\ref{sec:Existence_properties_BB} we first show that the IPC satisfies the \textit{one-ended} property (Lemma~\ref{lem:Toneended}), which means that there exists a unique path to infinity (or backbone) within the IPC. We discuss the future maximum weights $(W_k)_{k\geq 0}$ as defined in \eqref{Wk-process} and formalise its evolution as a Markov chain, thereby proving \eqref{eq:def_thm_wk}  of Theorem~\ref{thm:constrT}. 
Next we see that, conditioned on $(W_k)_{k\geq 0}$, we can derive the degree distribution of backbone vertices, thereby proving \eqref{eq:def_thm_DBk} of Theorem~\ref{thm:constrT}.

Then, in Section~\ref{sec:FiniteForest_properties_BB} we  consider the finite attached forests to the backbone. These are again BP trees, but conditioned on extinction and we derive their offspring distribution, thereby proving \eqref{eq:def_tilde_X} of Theorem~\ref{thm:constrT}. 
Finally, in Section \ref{sec:Constr_alg_subsec}, these results are combined to define a finite-time construction algorithm for the $k$-cut IPC, thereby concluding the proof of Theorem~\ref{thm:constrT}. 

We denote $\mathcal{T}(p)$ for the subtree of $\mathcal{T}$ obtained by taking the connected component that includes the root, when all edges with weight larger than $p\in(0,1)$ are omitted, i.e. percolation on $\T$ with percolation value $p$. We introduce the survival probability of this subtree by 
\begin{equation}
\label{eq:deftheta}
\theta(p) = \prob( |\mathcal{T}(p)|=\infty),
\end{equation}
and the extinction probability by $\eta(p)=1-\theta(p)$.
When $p$ is taken small enough, it is likely that $\mathcal{T}(p)$ becomes a finite tree. We denote the (percolation) critical value $p_c$ by
    \begin{equation}
    \label{eq:defp_c}
    p_c = \text{max}\{p: \theta(p)=0\}=\text{argmin}\{p: \theta(p)>0\}.
    \end{equation}    
As $\T(p)$ is a branching process with offspring distribution $\Binom( X, p)$, it follows that $p_c=1/\expec[X]$. 
Based on classical branching theory, we know that the survival probability $\theta(p)$ can be expressed in a fixed point equation (see for example \cite[Chapter 8.3]{Karlin1975}) as
\begin{equation}
\label{eq:fixed_point_theta_and_binom}
\begin{aligned}
    1-\theta(p) &= \expec[ (1-\theta(p))^{\Binom(X,p)}]
    =\expec[ (1-p\theta(p))^X].
\end{aligned}    
\end{equation}
In the following, we
also consider the survival probability of a BP tree with an additional edge attached to the root. These are also called \textit{attached} or \textit{planted} trees. We write the survival probability of an attached tree by 
\begin{equation}
    \label{eq:def_hattheta}
    \hat{\theta}(p)=p\theta(p).
\end{equation}

\subsection{Existence and properties of the backbone}
\label{sec:Existence_properties_BB}
To get a better view of the geometry of the IPC, we show that a fully realised IPC is \textit{one-ended}
\cite{Halin1964}, meaning that for every pair of paths to infinity, there exists a third (non-backtracking) path, that has an infinite intersection with both paths. 
For a tree, it means that there is a \textit{unique} infinite path.
The one-ended property and other properties in this section are regime-free, and have mostly been discussed in the literature.   

\begin{lemma}[One-ended IPC]
\label{lem:Toneended}
Consider $X$ as a power-law distribution as defined in \eqref{eq:powerlaw}. Then the IPC on a BP tree is one ended.  
\end{lemma}
\begin{proof}
    See e.g. \cite[Corollary 2.3]{Michelen2019}, for a proof on BP trees. Inspection of the proof shows that it does not depend on the moments of $X$.
\end{proof}

A key notion induced by the one-ended property is the existence of a unique path to infinity. We next analyse the future maximum weights along this path as introduced in \eqref{Wk-process}. The literature suggests that this evolves according to a Markov chain in \cite[Proposition 3.1]{Angel2008} for the Cayley tree, and adapted in \cite[Lemma 6.1]{Michelen2019} to finite-variance BP trees with $\alpha>2$. The extension to $\alpha\in(1,2)$ and $\alpha\in(0,1)$ is shown next:

\begin{proof}[Proof of Theorem~\ref{thm:constrT}, \eqref{eq:def_thm_wk}]
\cite[Lemma 6.1]{Michelen2019}. Inspection of the proof shows that it does not depend on the moments of $X$.
\end{proof}

We continue by showing that the Markov chain $(W_k)_{k\geq 0}$ converges almost surely to $p_c$. This means that, for every $\vep>0$, eventually no edge with weight larger than $p_c+\vep$ is accepted in the IPC:

\begin{lemma}[Convergence characteristics of $(W_k)_{k\geq 0}$]
	\label{lem:ConvW_kpc}
  $(W_k)_{k\geq 0}$ is a non-increasing process such that
 \begin{equation}
     W_k\xrightarrow{a.s.} p_c.
 \end{equation}
\end{lemma}
\begin{proof}
The claim is equivalent to the existence of an a.s. finite $K$, such that there exists a path from $v_K$ to infinity with weights smaller than $p_c+\varepsilon$, for all $\varepsilon>0$. This claim has been proven in \cite[Proposition 2.1]{Michelen2019}.
Inspection of the proof shows that it does not depend on the moments of $X$.
\end{proof}

We lastly formalise the degrees along the backbone, which are fully determined by $(W_k)_{k\geq 0}$. Intuitively, the IPC follows high-degree vertices, as more degrees imply more opportunities for a path to infinity with relatively low weights:

\begin{proof}[Proof of Theorem~\ref{thm:constrT},
\eqref{eq:def_thm_DBk}]
	Fix $W_k=w_k$.
	In order for a backbone vertex to have degree $l$, it needs to have one forest that survives on weights below $w_k$ and $l-1$ that die out on weights below $w_k$.
	Therefore,
	\begin{equation}
	\prob_k(D_{v_k}=l)=
	\frac{  \prob(W_k\in dw_k\mid X=l)\prob(X=l)}{ \prob(W_k\in dw_k)}.
	\end{equation}
	The first expression in the numerator denotes exact survival on $w_k$ given the degree. Exact survival on $w$ is approximated by $\hat \theta'(w)dw$. For the expression in the denominator, it is necessary to condition on the degree distribution first, as without this event, it is not clear how many trees must die out. With this, the probability simplifies to
	\begin{equation}
	\begin{aligned}
	\prob_k(D_{v_k}=l)&=
	\dfrac{\binom{l}{1}(1-\hat{\theta}(w_k))^{l-1}\hat{\theta}'(w_k)dw_k
		\prob(X=l)}{
		\sum_{y\geq 1}\binom{y}{1}(1-\hat{\theta}(w_k))^{y-1}\hat{\theta}'(w_k)dw_k\prob(X=y)}\\
	&=\frac{l(1-\hat{\theta}(w_k))^{l-1}
		\prob(X=l)}{
		\sum_{y\geq 1}y(1-\hat{\theta}(w_k))^{y-1}\prob(X=y)}.
	\end{aligned}
	\end{equation}
\end{proof}

\subsection{Properties of the finite attached forests}
\label{sec:FiniteForest_properties_BB}
Now that the distribution of the degree of the $k$-th backbone vertex is established, we recall that exactly one continues the path to infinity. This implies that $D_{v_k}-1$ additional finite trees are attached to the $k$-th backbone vertex.
These \textit{attached} trees are restricted on $W_k$ and conditionally finite and therefore, conditionally on $W_k$, are again independent sub-critical BP trees.
We also note that these trees are attached to the backbone via an edge, so that the entire attached sub-tree can be unexplored, in case this edge weight is larger than $W_k$. Given that the attached tree is explored, we find that, for fixed $k$, the offspring distribution is given by 
\begin{equation}
\label{eq:def_tildeX2}
 \tilde X^{\sss(k)}:=\Binom(X,w_k)\mid \text{ Extinction } \T(w_k).
\end{equation}
Next, we explicitly derive the offspring distribution of such an attached sub-tree and discuss some consequences this result brings:
\begin{proof}[Proof of Theorem~\ref{thm:constrT}, \eqref{eq:def_tilde_X}] 
By conditioning on $X$, we write the conditional probability as
    \begin{equation}
        \begin{aligned}
        \prob_k(\tilde X^{\sss(k)}=\ell)
        &= \frac{  \prob_k(\Binom(X,w_k)=\ell,\ \text{Extinction } \T(w_k))}{\prob_k(\text{ Extinction } \T(w_k))}
        \\&=
        \sum_{x\geq\ell }
        \frac{   \prob_k(\Binom(X,w_k)=\ell, \text{ Extinction } \T(w_k)\mid  X=x)
        \prob(X=x)
        }{\eta(w_k)}
        \\&=
        \frac{\sum_{x\geq \ell} \binom{x}{\ell} (\eta(w_k)w_k)^\ell (1-w_k)^{x-\ell}\prob(X=x) }{\eta(w_k)}.
        \end{aligned}
    \end{equation}
\end{proof}
For fixed $k$, $ \expec_k[\tilde X^{\sss(k)}]< 1$, but as $W_k\to p_c$, we expect that the attached trees become critical. However, for $\alpha\in(0,1)$, it is not directly clear what happens as $p_c=0$.
We find 
\begin{equation}
\begin{aligned}
    \expec_k[\tilde X^{\sss(k)}]
   & =
    \frac{\sum_{y\geq0}\sum_{x\geq y} y\binom{x}{y} (\eta(w_k)w_k)^y (1-w_k)^{x-y}\prob(X=x) }{\eta(w_k)}
    \\&=
    \frac{1}{\eta(w_k)}
    \sum_{x\geq 0}w_k\eta(w_k) (1-w_k+w_k\eta(w_k))^{x-1}x\prob(X=x)
    \\&=w_k\expec_k[X(1-w_k+w_k\eta(w_k))^{X-1} ].
    \end{aligned}
\end{equation}
Here, the second equality follows from Newton's binomium. 
Taking a limit for $k\to\infty$ we find, for $\alpha>2$ or $\alpha\in(1,2)$,
\begin{equation}
\label{eq:Conditional_final_clusters_mean_a>1}
    \lim_{w_k\searrow p_c}
    \expec_k[ \tilde X^{\sss(k)}] = p_c\expec[X]=1,
\end{equation}
and, for $\alpha\in(0,1)$, for which $p_c=0$, it follows by \eqref{eq:Bingham_Taylor_theta'_alpha(0,1)} below, that
\begin{equation}
\label{eq:Conditional_final_clusters_mean_a<1}
     \lim_{w_k\searrow 0}
    \expec_k[ \tilde X^{\sss(k)}] =\lim_{w_k\searrow 0} w_k \alpha\frac{1}{w_k} +o(w_k) =\alpha<1,
\end{equation}
The result of \eqref{eq:Conditional_final_clusters_mean_a>1} is as we would expect: when $W_k$ nears $p_c$ the attached sub-trees approach criticality. The more interesting result follows from \eqref{eq:Conditional_final_clusters_mean_a<1}, which implies that for $\alpha\in(0,1)$ the attached sub-trees remain sub-critical, even for $k\to\infty$.

\subsection{Construction algorithm}
\label{sec:Constr_alg_subsec}
We finally formulate the algorithmic construction in Theorem~\ref{thm:constrT} by \eqref{eq:def_thm_wk}-\eqref{eq:def_tilde_X}.
For finite $k$, the $k$-cut IPC can be reconstructed according to the following algorithm:
\begin{alg}[Construction of the $k$-cut IPC]
\label{alg:kcutIPC}
\normalfont
	$T(\clock)$ starts with the root and is from there constructed in an iterative way. In every step we add 
	the next vertex $v_k$ onto the backbone, the outgoing edge $\{v_k,v_{k+1}\}$ and the attached forest to $v_k$.
	Let $\tilde{T}^{\sss(k)}_{i}$ denote the $i$-th finite tree attached to backbone vertex $k$, that has offspring distribution as given in \eqref{eq:def_tilde_X}.
	Then we iterate the following steps:
	
	\begin{itemize}
		\item[(C0)] Consider $k\in\mathbb{N}$, and let $W_k=w_k$ be given;
		\item[(C1)] Given $w_k$, find the degree of $v_k$, denoted by $D_{v_k}$, as described in \eqref{eq:def_thm_DBk};
		\item[(C2)] Pick the next vertex backbone vertex $v_{k+1}$ uniformly over all the $D_{v_k}$ offspring of $v_k$;
		\item[(C3)] For $i=1,..., D_{v_k}-1$: given $w_k$,
		with probability $w_k$
		attach an
		i.i.d. $\tilde{T}^{\sss(k)}_{i}$ tree on $v_k$, as described in \eqref{eq:def_tilde_X},  and 
		with probability $1-w_k$
		add nothing;
		\item[(C4)] Determine $W_{k+1}$ as described in \eqref{eq:def_thm_wk};
		\item[(C5)] Repeat, with $k+1$ instead of $k$ and $W_{k+1}=w_{k+1}$ instead of $W_{k}$.
	\end{itemize}
	\end{alg}
	
\begin{proof}[Proof of Theorem~\ref{thm:constrT}: Algorithm \ref{alg:kcutIPC} creates the IPC]
	Suppose we have added the $k$-th vertex of the backbone and $W_k=w_k>p_c$ is given. 
	As seen from 
	$v_k$, $D_{v_k}$ i.i.d.\ BP trees emerge. As the path to infinity is one-ended
    (Lemma~\ref{lem:Toneended}), and due to the fact that all possible paths from $v_k$ are i.i.d., $v_{k+1}$ is a uniform child of $v_k$, which argues (C2).
	
	Again, by the one-ended property of $T(\clock)$, all other $D_{v_k}-1$ trees must satisfy two constraints: they are finite when they are restricted to weights below $w_k>p_c$. Such a tree starts with a  single edge, incident to $v_k$ and the root of a $\tilde{T}_{i}^{\sss(k)}$ tree, which satisfies both constraints. The edge between $v_k$ and the root of  $\tilde{T}_{i}^{\sss(k)}$ can have a weight larger than $w_k$,  in which case the  $\tilde{T}_{i}^{\sss(k)}$ is not explored in the $k$-cut IPC, and nothing is added.   This argues (C3).\\
	Finally, it is possible that the edge between $v_k$ and $v_{k+1}$ has exactly the same weight as $W_k$, in which case $W_{k+1}$ can be determined by \eqref{eq:def_thm_wk}, and the process continues at $v_{k+1}$. Otherwise, the process continues with $W_{k+1}=w_k$ at $v_{k+1}$. This argues (C4).
\end{proof}

\section{The scaling limit of future maximum backbone weights \mtitle{$(W_k)_{k\geq 0}$}}
 \label{sec:Scaling_W_k}
In this section we identify the scaling limits of $(W_k)_{k\geq0}$ as stated in Theorem~\ref{thm-subthmWk}. Theorem~\ref{thm:constrT} establishes that $(W_k)_{k\geq 0}$ is a Markov process, with an explicit transition kernel. In this section we show that $(W_k)_{k\geq 0}$ behaves asymptotically as a simple Markov jump process, whose law depends on $\alpha$. 
The scaling limit is already known for $\alpha>2$ \cite[Section 6]{Michelen2019},
and is here extended to $\alpha\in(1,2)$ and $\alpha\in(0,1)$.

This section is organised as follows.
In Section~\ref{sec-scaling-surv-prob}, we analyse the scaling of the percolation survival probability when $p$ is close to $p_c$.
In Section~\ref{sec:Scaling_W_k_alpha>1} we extend the known scaling results for the $(W_k)_{k\geq 0}$ process to $\alpha\in(1,2)$. Finally, Section~\ref{sec:Scaling_W_k_alpha_in_0,1} extends these result to $\alpha\in(0,1)$.

\subsection{Scaling of the survival probability}
\label{sec-scaling-surv-prob}

In the following, we see that the scaling of $(W_k)_{k\geq 0}$ differs for $\alpha>2$, $\alpha\in(1,2)$ and $\alpha\in(0,1)$. An indication of this is the different scaling behaviour of $\theta(p)$ for $p$ close to $p_c$ in the respective regimes. This is formalised in the following lemma:

\begin{lemma}[Asymptotics of the survival probability]
\label{lem:theta_scaling_>1}
\label{lem:theta_scaling_<1}
\label{lem:deriv_theta}
Consider $X$ a power-law offspring distribution in the domain of attraction of a stable law with exponent $\alpha$ as defined in \eqref{eq:powerlaw}. Let $\theta(p)$ denote the survival probability of this branching process, restricted to weights lower than $p>p_c$, as described in \eqref{eq:deftheta}. 
Define 
\begin{equation}
\label{eq:def_nu}
    \nu_\alpha = 
    \begin{cases}
     1 &  \text{ if } \alpha>2,\\
     \frac{1}{\alpha-1} & \text{ if } \alpha\in(1,2), \\
     \frac{\alpha}{1-\alpha} 
     & \text{ if }
     \alpha\in(0,1).
    \end{cases}
\end{equation}
Then, as $p\searrow p_c$,
  \begin{equation}
  \label{eq:theta_Scaling_ptopc}
      \theta(p) = C_\theta(p-p_c)^{\nu_\alpha}(1+o(1)), 
  \end{equation}
where
\begin{equation}
\label{eq:def_C_theta}
   C_\theta=  \begin{cases}
     \frac{2\expec[X]^3}{\expec[X(X-1)]} & \text{ if $\alpha>2$},\\
    \Big( \frac{\expec[X]^{\alpha+1}}{c_{\sss X} (-\Gamma(1-\alpha))}\Big)^{\frac{1}{\alpha-1}}& \text{ if $\alpha\in(1,2)$},\\
    (\Gamma(1-\alpha)c_{\sss X})^{\frac{1}{1-\alpha}}& \text{ if $\alpha\in(0,1)$}.\\
   \end{cases}
   \end{equation}
Additionally, 
\begin{equation}
\label{eq:implicit_theta'}
    \theta'(p) = \frac{ \theta(p)\expec[X(1-p\theta(p))^{X-1}]}{1-p\expec[X(1-p\theta(p))^{X-1}]}.
\end{equation}
   
\end{lemma}

\begin{proof}
     We refer to \cite[Proposition 2.2]{Michelen2020} for $\alpha>2$. For $\alpha\in(1,2)$, we use the asymptotic expansion of the generating function that states for $p\searrow p_c$
	\begin{equation}
\begin{aligned}	\label{eq:Bingham_taylor_alphain(1,2)}
	    \expec[(1-p\theta(p))^X] =&
	   \Big(1-\expec[X](-\log(1-p\theta(p)))\\&\quad+ c_{\sss X}(-\Gamma(1-\alpha))
	    (-\log(1-p\theta(p)))^\alpha\Big)(1+o(1)),
     \end{aligned}
	\end{equation}
	see \cite[Theorem 8.1.6]{Bingham1987}, for the special case $n=1$. Using the Taylor expansion of $-\log(1-p\theta(p))$ for $p\searrow p_c$ in combination with \eqref{eq:fixed_point_theta_and_binom}, we find \begin{equation}
	    \theta(p) =\Big(\expec[X]p\theta(p) - c_{\sss X}(-\Gamma(1-\alpha))(p\theta(p))^\alpha\Big)(1+o(1)),
	\end{equation}
	where one can use the continuation of the gamma function, using  $\Gamma(x)=1/x\Gamma(x+1)$, for $\alpha\in(1,2)$ \cite[Chapter 1]{Bonnar2017}, to find
   \begin{equation}
       -\Gamma(1-\alpha)=
       \frac{1}{\alpha-1}\Gamma(2-\alpha).
   \end{equation}
	We can isolate $\theta(p)$ to find
	\begin{equation}
	   \theta(p) =\Bigg( \frac{\expec[X]^{\alpha+1}}{c_{\sss X}(-\Gamma(1-\alpha))} (p-p_c)\Bigg)^{\frac{1}{\alpha-1}}(1+o(1)).
	\end{equation}
	Similarly, we can derive the scaling result for $\alpha\in(0,1)$ via \cite[Corollary 8.1.7]{Bingham1987} that, for $p\searrow 0$,
	\begin{equation}
	\label{eq:Bingham_Taylor_theta_alpha(0,10}
	    \theta(p) = 1-\expec[(1-p\theta(p))^X] 
	   = c_{\sss X}\Gamma(1-\alpha)(-\log(1-p\theta(p)))^{\alpha}(1+o(1)).
	\end{equation}
    By again using the expansion for the logarithm for $p\searrow 0$ we find
	\begin{equation}
	    \theta(p) = c_{\sss X}\Gamma(1-\alpha)(p\theta(p))^{\alpha}(1+o(1)).
	\end{equation}
	Isolating $\theta(p)$ gives that
	for $p\searrow 0$, $\theta(p)= C_\theta p^{\frac{\alpha}{1-\alpha}}(1+o(1))$, with $C_\theta = (\Gamma(1-\alpha)c_{\sss X})^{\frac{1}{1-\alpha}}$.
	
Finally, we derive $\theta'(p)$ for $p\searrow p_c$. Its existence follows from the implicit function theorem. We define
\begin{equation}
    F(p,\theta(p)) = 1-\theta(p)-\expec[(1-p\theta(p))^X],
\end{equation}
then for all $p>p_c$ the derivative of $F$ with respect to $\theta(p)$ exists and is non-zero as 

\begin{equation}
    \frac{\partial}{\partial\theta(p)} F(p,\theta(p)) = -1+p\expec[X(1-p\theta(p))^{X-1}],
\end{equation}
which behaves as $-(\alpha-1)\expec[X](p-p_c)<0$ for $\alpha\in(1,2)$, by \eqref{eq:Rpthetap_alpha(1,2)} and  as $\alpha-1<0$ for $\alpha\in(0,1)$,  by \eqref{eq:Rpthetap_alpha(0,1)}. Therefore, from the implicit function theorem \cite[Theorem 1.3.1]{Krantz2003}, it follows that the derivative of $\theta$ exists.
We can differentiate on both sides and isolate $\theta'(p)$ to find the result in \eqref{eq:implicit_theta'}.
\end{proof}

\subsection{Scaling limit of  \mtitle{$(W_k)_{k\geq 0}$} for \mtitle{$\alpha>2$} and \mtitle{$\alpha\in(1,2)$}}
\label{sec:Scaling_W_k_alpha>1}
In this section, we discuss the scaling limit of $(W_k)_{k\geq 0}$ for $\alpha>2$ and $\alpha\in(1,2)$.
The proof structure shares many similarities  with that of \cite[Corollary 6.3]{Michelen2019}, where $\alpha>2$ is discussed. The proof for $\alpha\in(1,2)$ as presented below, also extends to $\alpha>2$ and these proofs are therefore discussed together. 

\begin{proof}[Proof of Theorem~\ref{thm-subthmWk} for $\alpha\in(1,2)$]
The proof is split into four parts. Recall $\halpha= \alpha\wedge 2$ from \eqref{eq:def_halpha} and define the scaled jump process by
\begin{equation}
\label{eq:def_V_k}
V_k=\frac{W_k-p_c}{p_c}(\hat\alpha-1).
\end{equation}
Then the proof structure is as follows:

\begin{enumerate}
    \item \hyperref[par:3.2-part1]{\textbf{Bounds on the jump chain.}} We   approximate $(V_k)_{k\geq 0}$ for $k$ large with a simpler process, $(\tilde V_k(\eta))_{k\geq 0}$, as defined in \eqref{eq:def_V_tilde}.
    In Lemma~\ref{lem:A-Jumpchain_(1,2)} we show in particular that  $(\tilde V_k(\eta))_{k\geq 0}$ gives rise to upper and lower bounds on $(V_k)_{k\geq 0}$, for $k$ large enough. 
    \item \hyperref[par:3.2-part2]{\textbf{Continuous-time extension.}}
    We find a continuous-time equivalent for the upper and lower bounds, 
    by considering $(\tilde V_{N(t)}(\eta))_{t\geq 0}$ with $(N(t))_{t\geq 0}$ a Poisson process, for different values of $\eta$.
    We define this continuous time process by $(L_{\eta,\halpha}(t))_{t\geq 0}$
    that is a generalisation of  $(L_{\halpha}(t))_{t\geq 0}$, as defined in Definition \ref{def:alphaLEP}. 
    Lemma~\ref{lem-lim-reduced-env-process} shows that this Poissonisation also gives rise to 
    bounds on $(k\tilde V_{\lceil kt \rceil}(\eta) )_{t\geq 0}$
    for $k$ large enough, that are close to $(kL_{\eta,\halpha}(kt))_{t\geq 0}$.
    \item \hyperref[par:3.2-part3]{\textbf{Asymptotic point-wise convergence.}}
    The previous part implies that we need to find the scaling behaviour of $(kL_{\eta,\hat\alpha}(k t))_{t\geq 0}$.
    We show that $L_{\eta,\hat\alpha}(t)\overset{d}{=} c(\eta) \GammaD((1+\eta)/(\halpha-1),t)$ for fixed $t$, as is described in Lemma~\ref{lem:RE_alpha<2_gamma} and $c(\eta)$ given in \eqref{eq:ceta}.
    Due to the scaling property of the gamma distribution in \eqref{eq:gamma_scaling}, it also follows that $kL_{\eta,\hat\alpha}(kt)\overset{d}{=} c(\eta) \GammaD((1+\eta)/(\hat\alpha-1),t)$.
    \item \hyperref[par:3.2-part4]{\textbf{Extension to $J_1$-convergence.}} Finally, we extend the one-dimensional convergence results to convergence in the $J_1$-topology by showing convergence of the finite-dimensional distribution and tightness. This concludes the proof.
\end{enumerate}

In the following, we present the details of this proof.  However, first, we make a remark on calibrating the process at time 0 properly:
\begin{remark}[Calibrating the starting positions]
\label{rem:start_pos}
In the proof overview, we speak of multiple processes that are linked via couplings. We make sure that the couplings are proper by aligning the starting position via a.s.\ finite stopping times, 
$n_1$ defined in \eqref{eq:def_n1_stoptime} and $\tau_1$ and $\tau_2$ defined in \eqref{eq:def_tau_stoptime}. In the limit, these stopping times disappear, but they are necessary to formalise the proof. We refer to Figure \ref{fig:doublefigsstarttimes}  for an illustrating figure.
 \end{remark}

\begin{figure}[H]
     \centering
    \begin{subfigure}[t]{.48\linewidth}
         \centering
         \includegraphics[width=\linewidth]{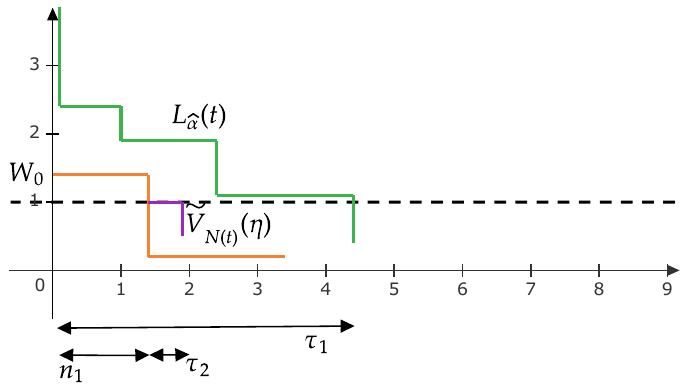}
         \caption{$V_0\leq1$ and therefore $n_0=0$.}
         \label{fig:starttimes_1}
     \end{subfigure}
     \hfil
     \begin{subfigure}[t]{.48\linewidth}
         \centering
         \includegraphics[width=\linewidth]{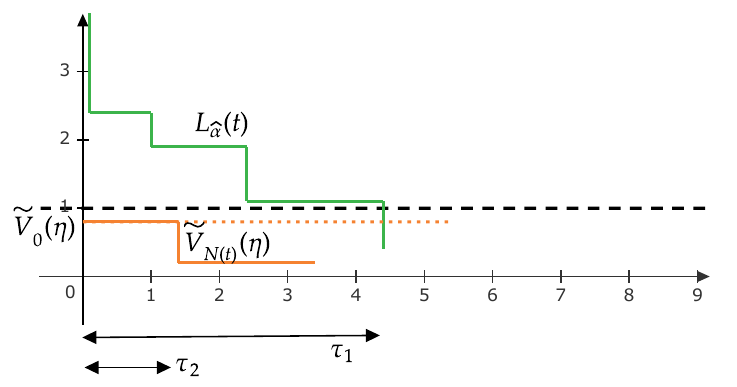}
         \caption{$V_0>1$ and therefore $n_0>0$.}
         \label{fig:starttimes_2}
     \end{subfigure}
        \caption{Realisations of $(\tilde V_{N(t)}(\eta))_{t\geq 0}$ and $(L_{\halpha}(t))_{t>0}$ with the respective stopping times $\tau_1$ and $\tau_2$. We notice that by \eqref{eq:uniform_jump_property}, the distribution of $\tilde V_{N(\tau_2)+n_1}(\eta)$ equals $L_{\halpha}(\tau_1)$, and thus, shifted by those times, as in \eqref{eq:VN(t)=R(t)_2}, the two processes are calibrated properly.}
        \label{fig:doublefigsstarttimes}
\end{figure}
\paragraphi{Part 1 - Bounds on the jump chain}
\label{par:3.2-part1}
We first define an auxiliary non-increasing Markov jump-process that is key in the following analysis.
For $U_k\overset{d}{=} \Unif[0,1]$ i.i.d, we define for $0<\varepsilon\ll1 $ and $\eta\in(-\varepsilon,\varepsilon)$,
\begin{equation}
\label{eq:def_V_tilde}
 \tilde V_{k+1}(\eta) = \begin{cases}
     \tilde V_k(\eta) & \text{w.p. } 1-(1-\eta) \tilde V_k(\eta),\\
    U_k^{(\hat{\alpha}-1)/(1+\eta)} \tilde V_k(\eta) & \text{w.p. } (1-\eta) \tilde V_k(\eta) ,\\
    \end{cases}
\end{equation}
under the condition that  $\tilde V_0(\eta)$ is distributed as $\min\{1/(1-\eta),V_0\}$.

Under the appropriate coupling, we show that for different values of $\eta$, $(\tilde V_{k}(\eta))_{k\geq 0}$, forms upper and lower bounds for $(V_k)_{k\geq 0}$. This is formalised as follows:
\begin{lemma}[Bounds on the jump chain for $\alpha\in(1,2)$]
\label{lem:A-Jumpchain_(1,2)}
Consider $\alpha\in(1,2)$ and the future maximum added weight sequence $(W_k)_{k\geq 0}$.
Define the a.s. finite stopping time $n_1\geq 0$ as 
\begin{equation}
\label{eq:def_n1_stoptime}
   n_1 = \min 
    \{n\geq 0\colon \
     V_n<1
    \}.
\end{equation}
For given $\eta\in[0,\varepsilon)$, there exists a coupling such that with probability 1 there exists some finite $K(\eta)$, such that, for $k>K(\eta)$,
\begin{equation}
\tilde V_{k-n_1}(-\eta) < V_{k} < \tilde V_{k-n_1}(\eta),    
\end{equation}
\end{lemma}
\begin{proof}
Recall the transition kernel for $(W_k)_{k\geq 0}$, from \eqref{eq:def_thm_wk}, which we rewrite to $V_k$ as introduced in \eqref{eq:def_V_k}. Then,
\begin{equation}
\label{eq:jumpchain_vk}
    V_{k+1}=\begin{cases}
     V_k & \text{ w.p. } 1-R(W_k)\theta(W_k),\\
     \frac{\theta^{-1}(U_k\theta(W_k))
     -p_c}{p_c}(\halpha-1)
     & \text{ w.p. } R(W_k)\theta(W_k),\\
    \end{cases}
\end{equation}
where $\theta^{-1}(x)$ is the inverse of $\theta(x)$. Note that $\theta^{-1}(p)$ exists for $p\in(p_c,1)$ since $p\mapsto \theta(p)$ is a bijective function on the domain $[p_c,1]$ that maps to $[0,1]$.
Furthermore, note that the distribution of $V_{k+1}$ at a jump follows from the distribution in \eqref{eq:def_thm_wk}, as for $x<V_k$,
\begin{equation}
\begin{aligned}
    \prob(V_{k+1}<x\mid V_{k+1}<V_{k})
   & =
    \frac{ R(W_k)\theta\big( p_c+\frac{p_cx}{(\halpha-1)}\big)}{R(W_k)\theta(W_k)}\\&=
    \prob\Big(\frac{\theta^{-1}(U_k\theta(W_k))
     -p_c}{p_c}(\halpha-1)<x\Big).
     \end{aligned}
\end{equation}
The jumps of $(V_k)_{k\geq 0}$ are driven by two sequences of i.i.d.\ standard uniformly distributed random variables, denoted by $U^{(j)}_k$ and $U^{(s)}_k$, where the former determines that a jump occurs in the $k$-th time step precisely when $\{U^{(j)}_k < R(W_k)\theta(W_k)\}$, and the latter determines the jump sizes. \\
Let us now find proper bounds for  $R(W_k)\theta(W_k)$. 

For $\alpha>2$, based on the definition of $R(p)$ in \eqref{eq:def_R(p)_intro},
one can use the approximations on $\theta(p)$ and $\theta'(p)$ directly from \cite[Corollary 3.2]{Michelen2019}, to find for $p\searrow p_c$
\begin{equation}
\label{eq:Rpthetap_alpha>2}
R(p)\theta(p)= \expec[X](p-p_c)(1+o(1)).
\end{equation}
For $\alpha\in(1,2)$, we first recall $\theta'(p)$ in \eqref{eq:implicit_theta'}, to find for $p\in(p_c,1)$
\begin{equation}
    R(p)\theta(p) = 
    1-p\expec[X(1-p\theta(p))^{X-1}].
\end{equation}
 We use the approximation from \cite[Theorem 8.1.6]{Bingham1987}, that states, for $\alpha\in(1,2)$ and $p\searrow p_c$,
\begin{equation}
\label{eq:Bingham_taylor_Derive_alpha(1,2)}
    \expec[X(1-p\theta(p))^{X-1}]
   =
   \Big( \expec[X] -\alpha c_{\sss X}(-\Gamma(1-\alpha)) (p\theta(p))^{\alpha-1}\Big)(1+o(1)).
\end{equation}
By Lemma~\ref{lem:theta_scaling_>1}, we 
then find, for $p\searrow p_c$, that
\begin{equation}
\label{eq:Rpthetap_alpha(1,2)}
    \begin{aligned}
    R(p)\theta(p) &=
    \Big(1-p(\expec[X] - \alpha c_{\sss X}(-\Gamma(1-\alpha))p^{\alpha-1}\theta(p)^{\alpha-1})\Big)(1+o(1))\\
    &=
    \Big(-\expec[X](p-p_c) +\alpha(p-p_c)\expec[X](p\expec[X])^{\alpha}\Big)(1+o(1))
    \\&=
    (\alpha-1)\expec[X](p-p_c)(1+o(1)).
    \end{aligned}
\end{equation}
The second statement follows from $\theta(p)= C_\theta(p-p_c)^{1/(\alpha-1)}(1+o(1))$ and $C_\theta$ in \eqref{eq:def_C_theta}.

By Lemma~\ref{lem:ConvW_kpc}, $W_k\searrow p_c$ almost surely and monotonically, meaning that for some $\eta>0$ there exists almost surely a finite random variable $K(\eta)$, such that, for any $k>K(\eta)$, and with probability 1, $|W_k-p_c|<\eta$.
Then, by \eqref{eq:Rpthetap_alpha>2} and \eqref{eq:Rpthetap_alpha(1,2)},
for any $\eta_1$, with probability one there exists a large  $K_1:=K_1(\eta_1)$ such that, for all $k>K_1$,
there exists a coupling such that
\begin{equation}
    (1-\eta_1)
    V_k < R(W_k)\theta(W_k)
    <(1+\eta_1)
    V_k,
\end{equation}
thereby bounding the jump probability of $V_k$ between that of $\tilde V_k(\eta)$ and $\tilde V_k(-\eta)$, as given in \eqref{eq:def_V_tilde} respectively.
We continue with bounds on the jump size distribution of \eqref{eq:jumpchain_vk}.
According to \aprefb{2}{sec:app:boundjumpsize}, there exists a coupling such that for any given $\eta_2>0$, there exists a $K_2:=K_2(\eta_2)$, such that for any $k>K_2$, with probability 1,
\begin{equation}
\label{eq:bounds_jumpsize_alpha(1,2)}
U_k^{(\halpha-1)/(1 -\eta_2)} V_k\leq 
    \frac{\theta^{-1}(U_k\theta(W_k))
     -p_c}{p_c}(\halpha-1)\leq  U_k^{(\halpha-1)/(1+\eta_2)} V_k,
\end{equation}
thereby also bounding the jump size between that of $\tilde V_k(-\eta)$ and $\tilde V_k(\eta)$ respectively. 
It remains to show that the $(V_k)_{k\geq 0 }$ process and the respective upper and lower bound processes are properly aligned at $k=0$. 

In case $V_0<1$ (i.e. $n_1=0$), the processes are properly coupled. Otherwise, we perform a time shift of $n_1$.
As $(V_k)_{k\geq 0}$ converges to 0 a.s. by Lemma~\ref{lem:ConvW_kpc}, $n_1$ is an almost surely finite random variable uniformly in $\eta\in(-\varepsilon,\varepsilon)$ and we find
\begin{equation}
   \tilde V_{k}(-\eta) \leq V_{k+n_1} \leq 
    \tilde V_k(\eta).
\end{equation}
We refer to Remark \ref{rem:start_pos} for an illustration on the effect of $n_1$.
Next, w.l.o.g. fix $0<\eta\ll 1$, then  for all $k>\max\{n_1,K_1(\eta),K_2(\eta) \}$ we find that there exists a coupling such that
\begin{equation}
    \tilde V_{k-n_1}(-\eta) \leq \frac{W_k-p_c}{p_c}(\hat\alpha-1) \leq
    \tilde V_{k-n_1}(\eta).
\end{equation}
Here the coupling is constructed such that it relies on the \textit{same} uniform random variables $(U^{(j)}_k)_{k\geq 1}$ to determine whether a jump down occurs and the \textit{same} uniform random variables $(U^{(s)}_k)_{k\geq 1}$ that determine the jump size.
\end{proof}
In the following, we find continuous-time equivalents for both the upper and lower bound processes.

\paragraphi{Part 2 - Continuous-time extension} 
\label{par:3.2-part2}
Let us first define the following non-increasing continuous-time process:
\begin{definition}[$\alpha$-Enhanced generalised lower envelope processes]
\label{def:GELEP}
Let $\eta \in(-\varepsilon,\varepsilon)$, $\varepsilon>0$ and 
\begin{equation}
    \label{eq:ceta}
c(\eta) = \frac{1}{1-\eta}.
\end{equation}
Then we define the $\alpha$-enhanced generalised lower envelope process ($\alpha$-GLEP) as follows:
For $t=\varepsilon$, let $L_{\eta,\halpha}(t)\overset{d}{=} c(\eta)\GammaD((1+\eta)/(\halpha-1),\varepsilon)$.
Then, the process
$ (L_{\eta,\halpha}(t))_{t\geq \varepsilon}$ evolves according to the following procedure: given $ L_{\eta,\halpha}(t_0)=l$ for some $t_0>\varepsilon$, and $E_{t_0} \overset{d}{=} \Exp( (1-\eta)l )$ then $L_{\eta,\halpha}(t)=l$ for $t\in[t_0,t_0+E_{t_0})$, after which it jumps to $L_{\eta,\halpha}(t+E_{t_0}) = U^{(\hat{\alpha}-1)/(1+\eta)}l$, where $U\overset{d}{=}\Unif[0,1]$.
\end{definition}
Note that the $\alpha$-LEP of Definition \ref{def:alphaLEP}
is a special case of the $\alpha$-GLEP with $\eta=0$ and we write $L_{0,\halpha}(t) = L_{\halpha}(t)$. In the following we show that $L_{\eta,\halpha}(t)$ is the natural continuous-time extension of $\tilde V_k(\eta)$.
\begin{lemma}[Relation of jump-process to the $\alpha$-GLEP]
\label{lem-lim-reduced-env-process}
Fix $t$ and $\xi>0$. Recall $(L_{\eta,\halpha}(t))_{t>0}$ from Definition \ref{def:GELEP}, for $\eta\in(-\varepsilon,\varepsilon)$.
Let $(N(t))_{t\geq 0}$ be a  Poisson process with rate $1-\eta$ and define the following a.s. finite stopping times as 
    \begin{equation}
    \label{eq:def_tau_stoptime}
    \tau_1 = \inf\{t\geq 0\colon\ L_{\eta,\halpha}(t) < \tilde V_0(\eta) \}, \quad \text{ and }\quad
    \tau_2 = \inf\{t\geq 0\colon\ N(t)=1\}.
    \end{equation}
Then, with probability 1, there exists a coupling and $K(\xi)<\infty$ such that, for any $k>K(\xi)$ and $t\geq\varepsilon$,
\begin{equation}
\label{eq:in_lemma:krkt_boundsonKV}
     k L_{\eta,\halpha}(kt(1+\xi) +\tau_1) \leq   k\tilde V_{\lceil kt +\tau_2\rceil+ n_1}(\eta) \leq k L_{\eta,\halpha}(kt(1-\xi) +\tau_1) 
\end{equation}

\end{lemma}

\begin{proof}
We first observe that, for $t_0>\varepsilon$ and conditionally on $\tilde V_{N(t_0)}= L_{\eta,\halpha}(t_0)$, there exists a coupling such that $(\tilde V_{N(t)}(\eta))_{t>t_0}=(L_{\eta,\halpha}(t))_{t>t_0}$, as both are jump processes with the same inter-jump and jump-size distributions. 
It therefore remains to show that there exists a time where the positions of both processes have the same distribution. 
We next show that based on the stopping times $\tau_1$ and $\tau_2$, the two processes are also calibrated properly. 

At $t=\varepsilon>0$, for $\varepsilon$ small enough, $ L_{\eta, \halpha}(\varepsilon)>\tilde V_0(\eta)$, with high probability. As $( L_{\eta, \halpha}(t))_{t\geq \varepsilon}$ converges to 0, there exists some a.s. finite time $\tau_1$, where $ L_{\eta, \halpha}$ jumps from some level $\ell$ above $\tilde V_0(\eta)$, to below $\tilde V_0(\eta)$.
This level is distributed as $\Unif[0,1]^{(\hat{\alpha}-1)/(1+\eta)}\ell$, conditioned on being below $ \tilde V_0(\eta)$. Note that for $b\leq \ell$ and $a>0$,
\begin{equation}
\label{eq:uniform_jump_property}
\{\ell\Unif[0,1]^{a}\mid \ell\Unif[0,1]^{a} < b\} \overset{d}{=}
b\Unif[0,1]^{a}.
\end{equation}
Thus, $L_{\eta, \halpha}(\tau_1)\overset{d}{=} \tilde V_0(\eta) \Unif[0,1]^{(\hat{\alpha}-1)/(1+\eta)}$.
Furthermore, we note that, from the definition of $\tilde V_k(\eta)$
in \eqref{eq:def_V_tilde} and $\tau_2$ in \eqref{eq:def_tau_stoptime}
that $\tilde V_{N(\tau_2)}(\eta)\overset{d}{=}\tilde V_0(\eta)  \Unif[0,1]^{(\hat{\alpha}-1)/(1+\eta)}$. Thus, the starting positions align after the respective stopping times.

This means that under this coupling
\begin{equation}
\label{eq:VN(t)=R(t)_2}
    (\tilde V_{N(t+\tau_2)}(\eta))_{t\geq \varepsilon} \overset{d}{=} ( L_{\eta,\halpha}(t+\tau_1))_{t\geq \varepsilon}.
\end{equation}
For a proof of \eqref{eq:uniform_jump_property} and for an illustrative figure of the effect of the stopping times we refer to Remark \ref{rem:start_pos}. 

Next, by the strong law of large numbers, it follows that for any $\xi>0$, there almost surely exists a finite $K$, such that,
for all $K>k$,
\begin{equation}
\frac{kt}{1+\xi}< N(kt)   < \frac{kt}{1-\xi},
\end{equation}
and by the non-increasing nature of $t\mapsto L_{\eta,\halpha}(t)$, the result follows.
\end{proof}

In the following part, we derive the distribution of $L_{\eta,\halpha}(t)$ for fixed $t$, which is used to simplify the bounds in \eqref{eq:in_lemma:krkt_boundsonKV}.

\paragraphi{Part 3 - Asymptotic point-wise distribution}
\label{par:3.2-part3}
We show that the process $(L_{\eta,\halpha}(t))_{t\geq \varepsilon}$ is pointwise gamma distributed. This is the final property necessary to show pointwise convergence of $kV_{\lceil kt \rceil }$.

\begin{lemma}[The $\alpha$-enhanced generalised lower envelope height is gamma distributed]
    \label{lem:RE_alpha<2_gamma}
Recall $\halpha$ from \eqref{eq:def_halpha}.  Then for $\halpha\in(1,2]$, consider the $\alpha$-GLEP $L_{\eta,\halpha}(t)$ for $t\geq\varepsilon$ fixed, as defined in Definition \ref{def:GELEP}.
Take $\eta\in(-\varepsilon,\varepsilon)$ and recall $c(\eta)$ from \eqref{eq:ceta}.
Then, for a fixed $t\geq\varepsilon$,
\begin{equation}
\label{eq:in_lemma_R(t)simGamma}
    L_{\eta,\halpha}(t) \stackrel{d}{=} c(\eta)\GammaD\Big( 
    \frac{1+\eta}{\hat\alpha-1},t \Big).
\end{equation}    
\end{lemma}
\begin{proof}
Let $t>\varepsilon>0$ be given. 
In the following we write $c(\eta):=c$ for notational convenience.
We first define a time-shifted process that is defined for $t\geq 0$ by
\begin{equation}
\label{eq:def_timeshift_L}
   L^{(\varepsilon)}_{\eta,\halpha}(t) :=L_{\eta,\halpha}(t+\varepsilon).
\end{equation}
By definition of the generalised $\alpha$-enhanced lower envelope process, this time-shifted process starts at time $t=0$ in location
$c\GammaD((1+\eta)/(\halpha-1),\varepsilon)$. Note that the claim in \eqref{eq:in_lemma_R(t)simGamma} follows by showing that
\begin{equation}
\label{eq:timeshift_L_goal}
    L^{(\varepsilon)}_{\eta,\halpha}(t)
   \overset{d}{=} c
    \GammaD\Big( 
    \frac{1+\eta}{\hat\alpha-1},t+\varepsilon \Big).
\end{equation}
This proof is organised as follows. For $x>0$,
we derive a partial differential equation for the tail distribution  $\prob(L^{(\varepsilon)}_{\eta,\halpha}(t)>x)$, by considering the expected change in $[t,t+\delta]$. Letting $\delta\searrow 0$ results in a partial differential equation in $x$ and $t$, that we solve via  Laplace transforms. Finally, we show that the resulting transform equals the transform of a gamma tail distribution with the proper parameters.

Let $L^{(\varepsilon)}_{\eta,\halpha}(t)$ be given as in  \eqref{eq:def_timeshift_L}.
Recall that $L^{(\varepsilon)}_{\eta, \halpha}(t)$ stays in the same location for an exponential amount of time with rate $(1-\eta)$ times its current height times and then jumps to height $U^{(\hat{\alpha}-1)/(1+\eta)}$ times its current height, where $U\overset{d}{=}\Unif[0,1]$ is independent of the other random variables.
     We consider the tail distribution of $L^{(\varepsilon)}_{\eta,\halpha}(t+\delta)$ conditionally on $L^{(\varepsilon)}_{\eta,\halpha}(t)$, for fixed $t>\varepsilon$. As $L^{(\varepsilon)}_{\eta,\halpha}(t)$ is a non-increasing process, we find
    \begin{equation}
        \prob(L^{(\varepsilon)}_{\eta,\halpha}(t+\delta)>x) = 
        \int_x^\infty 
        \prob(L^{(\varepsilon)}_{\eta,\halpha}(t+\delta)>x\mid L^{(\varepsilon)}_{\eta,\halpha}(t)=y)
        f_{L^{(\varepsilon)}_{\eta,\halpha}(t)}(y)\dif y,
    \end{equation}
    where $f_{L^{(\varepsilon)}_{\eta,\halpha}(t)}$ is the density of $L^{(\varepsilon)}_{\eta,\halpha}(t)$.
    Given that the process is at height $y$ at time $t$, 
    the number of jumps in 
  $[t,t+\delta]$ is Poisson distributed with rate $y(1-\eta)$. We can condition on this number to find
    \begin{equation}
    \begin{aligned}
         &\int_x^\infty 
        \prob(L^{(\varepsilon)}_{\eta,\halpha}(t+\delta)>x\mid L^{(\varepsilon)}_{\eta,\halpha}(t)=y)
        f_{L^{(\varepsilon)}_{\eta,\halpha}(t)}(y)\dif y\\
        &\quad =\int_x^\infty \Big(
        \e^{-\delta y(1-\eta)} 
        +\prob(U^{(\hat{\alpha}-1)/(1+\eta)}L^{(\varepsilon)}_{\eta,\halpha}(t)>x\mid L^{(\varepsilon)}_{\eta,\halpha}(t)=y)\\&\qquad \times\delta y(1-\eta)\e^{-\delta y(1-\eta)} +o(\delta ) \Big) f_{L^{(\varepsilon)}_{\eta,\halpha}(t)}(y)\dif y.
        \end{aligned}
    \end{equation}
    Furthermore, the probability of two or more arrivals within $[t,t+\delta]$ is $O(\delta^2)$. 
    Taylor expanding the exponentials, we find
    \begin{equation}
    \begin{aligned}
        \prob(L^{(\varepsilon)}_{\eta,\halpha}&(t+\delta)>x) = 
         \prob(L^{(\varepsilon)}_{\eta,\halpha}(t)>x) -\delta(1-\eta) \int_x^\infty 
        y f_{L^{(\varepsilon)}_{\eta,\halpha}(t)}(y) \dif y
       \\&+\delta(1-\eta)\int_x^\infty 
        \Big(1-(x/y)^{(1+\eta)/(\hat{\alpha}-1)}\Big) y f_{L^{(\varepsilon)}_{\eta,\halpha}(t)}(y)\dif y+o(\delta)\\
        &= \prob(L^{(\varepsilon)}_{\eta,\halpha}(t)>x)
        -\delta(1-\eta) \int_x^\infty 
        \frac{x^{(1+\eta)/(\hat{\alpha}-1)}}{y^{(1+\eta)/(\hat{\alpha}-1)}-1}f_{L^{(\varepsilon)}_{\eta,\halpha}(t)}(y)\dif y
        +o(\delta)
        .
        \end{aligned}
        \end{equation}
    Subtracting $\prob(L^{(\varepsilon)}_{\eta,\halpha}(t)>x)$ on both sides, dividing by $\delta$ and
    taking a limit for $\delta\searrow0$ gives a partial differential equation.
    We write $g^{(\varepsilon)}(x,t)=\prob(L^{(\varepsilon)}_{\eta,\halpha}(t)>x)$.
    To avoid confusion, we note that $f_{L^{(\varepsilon)}_{\eta,\halpha}(t)}(x)= \partial/(\partial x) (1-g^{(\varepsilon)}(x,t)) $, and write it as such. Recall $c=1/(1-\eta)$, then we  find
    \begin{equation}
    \label{eq:Differential_equation_g}
    \begin{aligned}
        c\frac{\partial}{\partial t} g^{(\varepsilon)}(x,t) &= 
        -\int_x^\infty 
        \frac{x^{(1+\eta)/(\hat{\alpha}-1)}}{y^{(1+\eta)/(\hat{\alpha}-1)}-1}
        \frac{\partial}{\partial y } (1-g^{(\varepsilon)}(y,t))
        \dif y.\\
        \end{aligned}
    \end{equation}
Via a simple substitution we find that  
\begin{equation}
\label{eq:diff_eq_sol}
    g^{(\varepsilon)}(x,t) 
    =\int_{x/c}^\infty 
    \frac{(t+\varepsilon)^{(1+\eta)/(\halpha-1)}
    z^{(1+\eta)/(\halpha-1)-1}
    }{\Gamma((1+\eta)/(\halpha-1))}\e^{-z(t+\varepsilon)}\dif z
\end{equation}
is a solution to \eqref{eq:Differential_equation_g} and corresponds to the tail distribution of a $c\GammaD((1+\eta)/(\halpha-1),t+\varepsilon)$ random variable. In Appendix \aprefa{B4}{app:difeq}, we provide details of the proof that \eqref{eq:diff_eq_sol} is also the unique solution to \eqref{eq:Differential_equation_g}.
\end{proof}

In particular, the result of Lemma \ref{lem:RE_alpha<2_gamma} implies that, in the special case $\eta=0$, 
\begin{equation}
\label{eq:GLEP=0}
   L_{0,\halpha}(t):= L_{\halpha}(t)\overset{d}{=} \GammaD(1/(\halpha-1),t).
\end{equation}
To conclude the pointwise convergence, we recall that from Lemmas~\ref{lem:A-Jumpchain_(1,2)} and \ref{lem-lim-reduced-env-process} that, for $k$ large enough,
there exists a coupling such that, uniformly in $t\geq \varepsilon$,
\begin{equation}
\label{eq:lep_gen_bounds}
    k L_{-\eta,\halpha}(kt(1+\xi)+\tau_1-\tau_2)
    \leq
    kV_{\lceil k(t+n_1/k)\rceil}
    \leq 
    k L_{\eta,\halpha}(kt(1-\xi)+\tau_1-\tau_2).
\end{equation}
Moreover, Lemma~\ref{lem:RE_alpha<2_gamma} shows that $L_{\eta,\hat\alpha}(t)$ is gamma distributed for fixed $t$. We use the scaling properties of \eqref{eq:gamma_scaling} to obtain, for $t\geq \varepsilon$,
\begin{equation}
\label{eq:GE-L_scaling}
     k L_{\eta,\halpha}(kt(1+\xi)+\tau_1-\tau_2)
     =
     L_{\eta,\halpha}( t(1+\xi) +(\tau_1-\tau_2)/k).
\end{equation}

Substituting \eqref{eq:GE-L_scaling} in \eqref{eq:lep_gen_bounds} and taking a limit for $k\to\infty$, we find for $t\geq \varepsilon$

\begin{equation}
\label{eq:final_coupling_alpha(1,2)}
     L_{-\eta,\halpha}(t(1+\xi)) \leq 
    \lim_{k\to\infty} kV_{\lceil kt\rceil}\leq  
    L_{\eta,\halpha}(t(1-\xi)).
\end{equation}
Letting $\xi\to 0$ and $\eta\to 0$ proves, according to \eqref{eq:GLEP=0}, the pointwise convergence of $kV_{\lceil kt\rceil}$ to $L_{\halpha}(t)$. 

\paragraphi{Part 4 - Extension to $J_1$-convergence}
\label{par:3.2-part4}
According to \cite[Theorem 13.1]{Billingsley1999}, it is sufficient to show convergence of the finite-dimensional distribution and tightness for  $(kV_{\lceil kt\rceil})_{t\geq 0}$.
As the coupling in \eqref{eq:final_coupling_alpha(1,2)} is uniform in $t$, the upper and lower bound in the coupling hold for any $t$ \textit{simultaneously}. Based on its construction from Lemma~\ref{lem:A-Jumpchain_(1,2)}, we find, for $t_1<\cdots<t_n$, and $\eta\in[0,\varepsilon)$,
\begin{equation}
\label{eq:findimdistr_wk_alphain(1,2)}
 \big(c(-\eta) L_{\eta,\halpha}(t_i(1-\xi))\big) _{i=1}^n\leq 
\big(kV_{\lceil kt_i\rceil}\big) _{i=1}^n\leq
\big(c(\eta)
L_{\eta,\halpha}(t_i(1+\xi) )\big) _{i=1}^n,
\end{equation}
where the inequality is componentwise.
This can be extended to show that the multidimensional Laplace transform is properly bounded, from which the result follows by taking $\xi\to 0$ and $\eta\to0$.
This implies convergence of the finite-dimensional distributions.
Tightness is proved in Appendix \aprefa{B.1}{app:3-1}.
\end{proof}
In the following analysis, we also work with the time-scaled $\alpha$-enhanced lower-envelope process, $(tL_{\halpha}(t))_{t>0}$. This scaling factor ensures stationarity, which is formalised in the following remark:
\begin{remark}[Stationarity of the $(tL_{\halpha}(t))_{t>0}$ process]
\label{rem:stationarity_tLalpha(t)}
Take $\varepsilon>0$. Consider $L_{\halpha}(t)$ from Definition \ref{def:alphaLEP} for non-negative $t$. 
Then by \eqref{eq:GLEP=0}, $L_{\halpha}(t)$ is gamma distributed and, by the scaling property of the gamma distribution in \eqref{eq:gamma_scaling},
\begin{equation}
    \label{eq:statioraity_L(t)}
    tL_{\halpha}(t) \overset{d}{=}
    \GammaD(1/(\halpha-1),1),
\end{equation}
from which stationarity follows.
In particular \eqref{eq:statioraity_L(t)} is also true for $t=0$.
\end{remark}

\subsection{Scaling limit of \mtitle{$(W_k)_{k\geq 0}$} for \mtitle{$\alpha\in(0,1)$}{}}
\label{sec:Scaling_W_k_alpha_in_0,1}
We show the convergence of $(W_k)_{k\geq 0}$ for $\alpha\in(0,1)$ according to the same roadmap as used in Section~\ref{sec:Scaling_W_k_alpha>1}. 
The difference to Section~\ref{sec:Scaling_W_k_alpha>1} becomes clear in the upper and lower bounds on $(W_k)_{k\geq 0}$, see \eqref{eq:def_V_tilde_alpha<1}, as the jump probability is asymptotically \emph{independent} of the height of the process. This induces the exponential decay of $(W_k)_{k\geq 0}$, contrary to the linear decay of $(W_k)_{k\geq 0}$ for $\alpha\in(1,2)$ and $\alpha>2$. 
\begin{proof}[Proof of Theorem~\ref{thm-subthmWk} for $\alpha\in(0,1)$]
The proof is split into four parts and the proof structure is as follows:

\begin{enumerate}
    \item \hyperref[par:3.3-part1]{\textbf{Bounds on the jump chain.}} We approximate $
    W_k$ for $k$ large with a simpler process $(\tilde V_k(\eta))_{k\geq0}$, for some $\eta$ close to 0. In Lemma~\ref{lem:A-Jumpchain_(0,1)}, we show in particular that this process gives rise upper and lower bound processes on $(W_k)_{k\geq0}$ given by  $ (\tilde V_k(\eta))_{k\geq0}$ and $ (\tilde V_k(-\eta))_{k\geq0}$,  respectively.   
    \item \hyperref[par:3.3-part2]{\textbf{Continuous-time extension.}}
    We find a continuous-time equivalent for the upper and lower bounds by considering $(\tilde\V_{N(t)}(\eta))_{t\geq 0}$ with $(N(t))_{t\geq 0}$ a unit-rate Poisson process. It then follows  that $(\tilde\V_{N(t)}(\eta))_{t\geq 0}$ has the same distribution as the exponential Poisson decay process $(\L_{\eta,\alpha}(t))_{t\geq 0}$, as defined in Definition \ref{def:ExpPoidecayprocess}. Lemma~\ref{lem-lim-W(t)} shows that this Poissonisation again gives rise to upper and lower bounds on $((\tilde \V_{\lceil kt \rceil}(\eta))^{1/k})_{t\geq 0}$ that are close to $((\L_{\eta,\halpha}(kt))^{1/k})_{t\geq0}$.

    \item \hyperref[par:3.3-part3]{\textbf{Asymptotic point-wise convergence.}}
    The previous part implies that we need to find the scaling behaviour of $((\L_{\eta,\halpha}(kt))^{1/k})_{t\geq0}$. Contrary to  $\alpha\in(1,2)$,   Lemma~\ref{lem:scaling_EPDP} shows that the limiting process of $((\L_{\eta,\halpha}(kt))^{1/k})_{t\geq0}$ is a \emph{deterministic} process $A_\eta(t)$.
    
    \item \hyperref[par:3.3-part4]{\textbf{Extension to $J_1$-convergence.}} Finally, we extend one-dimensional convergence to convergence in the $J_1$-topology by showing convergence of the finite-dimensional distribution and tightness. This concludes the proof.
\end{enumerate}

\paragraphi{Part 1 - Bounds on the  jump chain}
\label{par:3.3-part1}
We start by finding bounds on $(W_k)_{k\geq 0}$, to simplify the limiting object. Let $\varepsilon\ll 1-\alpha$ and $\eta\in(-\varepsilon,\varepsilon) $, then we define the following non-increasing Markov jump-chain
$\tilde\V_{k}(\eta)$ that evolves as
\begin{equation}
\label{eq:def_V_tilde_alpha<1}
    \tilde\V_{k+1}(\eta) = \begin{cases}
     \V_k(\eta) & \text{w.p. } \alpha(1+
     \eta),\\
     U_k^{\frac{1-\alpha}{\alpha(1+\eta)}}\V_k(\eta) & \text{w.p. }1-\alpha(1+\eta)  ,\\
    \end{cases}
\end{equation}
under the condition that $\tilde\V_0(\eta) $ are distributed as $ W_0$.
Then similarly to the proof in  Lemma~\ref{lem:A-Jumpchain_(1,2)}, $\tilde\V_{k}(\eta)$ form upper and lower bounds for $(W_k)_{k\geq 0}$.

\begin{lemma}[Bounds on the jump chain for $\alpha\in(0,1)$]
\label{lem:A-Jumpchain_(0,1)}
Fix $\alpha\in(0,1)$.
Then for given $\eta\in[0,\varepsilon)$, there exists a coupling such that with probability 1 there exists some finite $K(\eta)$, such that, for all $k>K(\eta)$,
\begin{equation}
    \tilde\V_k(-\eta) < W_k < \tilde\V_k(\eta).
\end{equation}
\end{lemma}

\begin{proof}
By \eqref{eq:def_thm_wk},
\begin{equation}
    W_{k+1} = \begin{cases}
    W_k &\text{w.p. } 1-R(W_k)\theta(W_k),\\
    \theta^{-1}( U_k\theta(W_k)) &\text{w.p. } R(W_k)\theta(W_k).\\
    \end{cases}
\end{equation}
Recall, by \eqref{eq:def_R(p)_intro} and \eqref{eq:implicit_theta'}, that, for $p>p_c=0$,
\begin{equation}
\label{eq:rep-R(W)}
   R(p)\theta(p) = \expec[X(1-p\theta(p))^{X-1}] \frac{\theta(p)}{\theta'(p)}
   =1-p\expec[X(1-p\theta(p))^{X-1}]
   .
\end{equation}
We use a Tauberian theorem based on
\cite[Equation 8.1.12]{Bingham1987}  and the result of Lemma \ref{lem:theta_scaling_<1} to find for $p\searrow 0$
\begin{equation}
\begin{aligned}
    \label{eq:Bingham_Taylor_theta'_alpha(0,1)}
    \expec[X &(1-p\theta(p))^{X-1}]=
    \Big[\alpha \Gamma(1-\alpha) c_{\sss X}
    (-\log(1-p\theta(p)))^{\alpha-1}\Big](1+o(1))\\&
    =\Big[\alpha \Gamma(1-\alpha) c_{\sss X}
    C_{\theta}^{\alpha-1}p^{-1}(1+\o{1/p}{})\Big](1+\o{1/p}{})=
    \alpha p^{-1}(1+o(1)).
\end{aligned}
\end{equation}
Therefore, for $\alpha\in(0,1)$ we find,
\begin{equation}
\label{eq:Rpthetap_alpha(0,1)}
    \begin{aligned}
    R(p)\theta(p)=(1-\alpha)(1+o(1)).
    \end{aligned}
\end{equation}
Thus, asymptotically, $\prob(W_{k+1} = W_k\mid W_k )$ is close to $\alpha$.
We use that $W_k$ is a monotonically non-increasing sequence that converges a.s. to $p_c=0$. Using the result from \eqref{eq:Rpthetap_alpha(0,1)}, it therefore follows that, for fixed $\eta_1$, there exists a large $K_1:=K_1(\eta)$ such that, for all $k>K_1(\eta_1)$,
\begin{equation}
\label{eq:couplingbound_alpha<1_1}
   1-\alpha -\eta_1<R(W_k)\theta(W_k) < 1-\alpha +\eta_1.
\end{equation}
thereby bounding the jump distribution of $W_k$ between that of $\tilde\V_k(\eta)$ and $\tilde\V_k(-\eta)$ respectively.
Furthermore, according to the same argument, there exists a coupling such that for any given $\eta_2>0$, with probability 1 there exists a $K_2:=K_2(\eta_2)$, such that, for any $k>K_2(\eta_2)$,
\begin{equation}
\label{eq:couplingbound_alpha<1_2}
U_k^{\frac{1-\alpha}{\alpha(1-\eta_2)}} W_k <
    \theta^{-1}(U_k \theta(W_k) )< U_k^{\frac{1-\alpha}{\alpha(1 +\eta_2)}} W_k,
\end{equation}
where we recall that $U_k\overset{d}{=}\Unif[0,1]$ i.i.d.\ and thereby bounding also the jump size between that of $\tilde\V_k(-\eta)$ and
$\tilde\V_k(\eta)$, respectively.
These bounds are non-trivial and are proven in full detail in \aprefb
{2}{sec:app:boundjumpsize}.
Then, for $\eta\in(- \varepsilon,\varepsilon)$ fixed and conditionally on $\tilde\V_0(\eta) \overset{d}{=} W_0$, we find that there exists a coupling such that, for all $k>\max\{K_1(\eta),K_2(\eta)\}$,
\begin{equation}
    \tilde\V_{k}(-\eta) < W_k<
    \tilde\V_k(\eta).
\end{equation}
\end{proof}

\paragraphi{Part 2 - Continuous-time extension}
\label{par:3.3-part2}
We Poissonise $\tilde V_k(\eta)$ from \eqref{eq:def_V_tilde_alpha<1} and show the existence of an equivalent continuous-time stochastic process that behaves as the following process:

\begin{definition}[Exponential Poisson decay process]
\label{def:ExpPoidecayprocess}
Fix $\alpha\in(0,1)$ and $\eta\in(-\varepsilon,\varepsilon)$. Let $(N_\eta(t))_{t\geq 0}$ be a Poisson process with rate $1-(\alpha+\eta)$. Then we define the exponential Poisson decay process (EPDP) by
\begin{equation}
   \L_{\eta,\alpha}(t) 
    = 
  \prod_{i=1}^{N_\eta(t)}
  Y_i
  ,
\end{equation}
where $(Y_i)_{i\geq 1}$ are i.i.d.\
such that $Y_i \overset{d}{=} \Unif[0,1]^{(1-\alpha)/(\alpha(1+\eta))}$. 
\end{definition}
For $N(t)$ a unit-rate Poisson process, we show that $\tilde\V_{N(t)}(\eta)$  closely resembles $\L_{\eta,\alpha}(k)$, which we formalise as follows:
\begin{lemma}[Continuous-time equivalent of $\tilde\V_k(\eta)$]
\label{lem-lim-W(t)}
Let $\L_{\eta,\alpha}$ denote the exponential Poisson decay process as defined in Definition \ref{def:ExpPoidecayprocess} and $N(t)$ a unit-rate Poisson process.
 Define two a.s. finite stopping times as 
\begin{equation}
\label{eq:def_tau_stoptime_alpha01}
\tau_3 = \inf\{t\geq0:\L_{\eta,\alpha}(t) < \tilde V_0(\eta)\},
\qquad \text{ and } \qquad 
\tau_4 = \inf\{t\geq 0: N(t)=1\}.
\end{equation}
Then, for any $\xi>0$, there exists a coupling and random variable $K(\xi)<\infty$, such that, for any $k>K(\xi)$ and $t>\varepsilon$,
\begin{equation}
\label{eq:in_lemma:krkt_boundsonKtttV_alpha01}
     \L_{\eta,\alpha}(kt(1+\xi) +\tau_3) <   \tilde\V_{\lceil kt+\tau_4 \rceil}(\eta) < \L_{\eta,\alpha}(kt(1-\xi)+\tau_3).
\end{equation}
  \end{lemma}
  \begin{proof}
Consider $\tilde\V_k(\eta)$ as defined in \eqref{eq:def_V_tilde_alpha<1}.
Then we first observe that conditionally on $\L_{\eta,\alpha}(t_0)=\tilde\V_{N(t_0)}(\eta)$, 
there exists a coupling where $(\L_{\eta,\alpha}(t_0))_{t\geq t_0}=(\tilde\V_{N(t_0)}(\eta))_{t\geq t_0}$, as both are jump processes with the same inter-jump and jump-size distribution.

In order to align the starting position, we use the stopping times $\tau_3$ and $\tau_4$ and refer to Remark \ref{rem:start_pos} for an illustrative figure. 
Given that $\L_{\eta,\alpha}(t)$ has a jump at time $\tau_3$,
from some level $\ell>\tilde V_0(\eta)$, below $\tilde V_0(\eta)$.
Then it jumps to height $\ell U^{(1-\alpha)/(\alpha(1+\eta))}$,
that is conditionally smaller than $\tilde V_0(\eta)$. Based on the property for conditional uniform random variables
in \eqref{eq:uniform_jump_property}, this jump is distributed as a
$U^{ (1-\alpha)/(\alpha(1+\eta))}\tilde V_0(\eta)$ random variable.
Moreover by definition of $\tilde V_k(\eta)$ in \eqref{eq:def_V_tilde_alpha<1} and $\tau_4$ in \eqref{eq:def_tau_stoptime_alpha01}, $\tilde \V_{N(\tau_4)}(\eta)\overset{d}{=} \tilde V_0(\eta) U^{(1-\alpha)/(\alpha(1+\eta))}$.
Thus, the starting positions align after the respective stopping times.
This implies that
\begin{equation}
    (\tilde \V_{N(t+\tau_4)}(\eta))_{t\geq0}
    \overset{d}{=}
    (\L_{\eta,\alpha}(t+\tau_3))_{t\geq 0}.
\end{equation}

We then know by the law of large numbers $N(t)/t \to 1$ a.s., so that in the limit for some $\xi$, there exists a $K(\xi)$ such that, for all $k>K(\xi)$,
\begin{equation}
    N(kt)\in [ (1-\xi)kt, (1+\xi)kt],
\end{equation}
and by the non-increasing nature of $t\mapsto \L_{\eta,\alpha}(t)$ the result follows.
\end{proof}

\paragraphi{Part 3 - Asymptotic point-wise convergence}
\label{par:3.3-part3}
We consider the convergence of the  $((\L_{\eta,\alpha}(kt))^{1/k})_{t\geq 0}$ process:

\begin{lemma}[Scaling of $(\L_{\eta,\alpha}(t))_{t\geq 0}$]
\label{lem:scaling_EPDP}
Consider the exponential Poisson decay process $(\L_{\eta,\alpha}(t))_{t>0}$ as described in Definition \ref{def:ExpPoidecayprocess} and
\begin{equation}
\label{eq:def_A(t)}
    A_\eta(t)=\frac{(1-\alpha-\eta)(1-\alpha)}{\alpha(1+\eta)}t.
\end{equation}
Then, for every fixed $t>0$,
    \begin{equation}
        \L_{\eta,\alpha}(tk)^{1/k}\xrightarrow{\prob} \e^{-A_\eta(t)}, \qquad k\to\infty.
    \end{equation}
\end{lemma}
\begin{proof}
Consider the alternative process $\Psi_\eta(t) = -\log(\L_{\eta,\alpha}(t))$. We recall that, given $\Psi_\eta(t)=\psi$, $\Psi_\eta(t)$ remains at $\psi$ for an $\Exp(1-\alpha-\eta)$ amount of time, after which it jumps to $\psi - \log(Y_i)$, where $\log(Y_i) \overset{d}{=}\Exp(\alpha(1+\eta)/(1-\alpha))$.
We prove the result based on the Laplace transform of $\Psi_\eta(t)$, given by
	\begin{equation}
	\expec\Big[\e^{-s\frac{1}{k}\Psi_\eta(kt)}\Big]
	=
	\sum_{n=0}^\infty 
	\expec\Big[\e^{-s\frac{1}{k}\sum_{i=1}^n \log(1/Y_i)}\mid
	N_\eta(kt)=n\Big]
	\prob(N_\eta(kt)=n),
	\end{equation}
	where $(N_\eta(t))_{t\geq0}$ denotes the Poisson process with parameter $1-\alpha-\eta$. Recalling the Laplace transform of an exponential random variable, we find
	\begin{equation}
	\begin{aligned}
	\label{eq:regime3convhk}
	\expec\Big[\e^{-s\frac{1}{k}\Psi_\eta(kt)}\Big]
	&=
	\sum_{n=0}^\infty
	\Bigg( 
	\frac{ \frac{\alpha}{1-\alpha}(1+\eta) }{ \frac{s}{k} +\frac{\alpha}{1-\alpha}(1+\eta)}
	\Bigg)^n 
	\e^{-(1-\alpha-\eta) kt}
	\frac{ ((1-\alpha-\eta) kt)^n}{n!}
	\\&=
	\exp\bigg\{-
	(1-\alpha-\eta)\frac{(1-\alpha)s/k}{(1-\alpha)s/k+\alpha(1+\eta)}
	\bigg\}.
	\end{aligned}
	\end{equation}
	Taking $k\to\infty$ 
	gives
	\begin{equation}
	    \expec\Big[\e^{-s\frac{1}{k}\Psi(kt)}\Big]
	    \to 
	    \e^{-\frac{ts(1-\alpha-\eta)(1-\alpha)}{\alpha(1+\eta)}
	    },
	\end{equation}
	which corresponds to the Laplace transform of the degenerate process $A_\eta(t)$ as defined in \eqref{eq:def_A(t)}. From this result it follows that, for fixed $t$,
	\begin{equation}
	\frac{1}{k}
	\Psi_\eta(kt) 
	\xrightarrow{d}
	A_\eta(t).
	\end{equation}
The result then follows by the continuous mapping theorem.	
\end{proof}	
By Lemmas~\ref{lem:A-Jumpchain_(0,1)} and \ref{lem-lim-W(t)}, for all $t\geq\varepsilon>0$,
\begin{equation}
\label{eq:in_lemma:krkt_boundsonKtttV}
     \L_{\eta,\alpha}(k(t(1+\xi) +\tau_3/k))^{1/k}  \leq (W_{\lceil k(t+\tau_4/k) \rceil}(\eta) )^{1/k} \leq \L_{\eta,\alpha}(k(t(1-\xi)+\tau_3/k))^{1/k}.
\end{equation}
By Lemma \ref{lem:scaling_EPDP}, with high probability for $k$ large,
\begin{equation}
    \e^{-A_{-\eta}( t(1+\xi))}\leq
   (W_{\lceil kt\rceil})^{1/k}
    \leq \e^{-A_{\eta}( t(1-\xi))}
\end{equation}
Then, we take a limit for $\eta\to 0$, where we note that $A_0(t)=t(1-\alpha)^2/\alpha$.
Also taking $\xi\to 0$, gives the result for pointwise convergence of $W_{\lceil kt\rceil }^{1/k}$ to $A(t)$.

\paragraphi{Part 4 - Extension to $J_1$-convergence}
\label{par:3.3-part4}
According to a similar reasoning as in the proof for $\alpha\in(1,2)$, we use \cite[Theorem 13.1]{Billingsley1999}.
Based on the coupling in \eqref{eq:in_lemma:krkt_boundsonKtttV} being uniformly in $t$, the extension to convergence of the finite-dimensional distribution follows similarly to \eqref{eq:findimdistr_wk_alphain(1,2)}. Finally tightness follows from Appendix \aprefa{B.1}{app:3-1}.
\end{proof}

\begin{remark}[Fluctuations]
\label{rem:Fluctuations}
As the $W_{k}$ process is exponentially small in $k$ for $\alpha\in(0,1)$ (contrary to linear for $\alpha>1$), there is room for additional fluctuations in the convergence of $W_k$.
Intuitively, this means that next to the dominant convergence behaviour of $\e^{-k\alpha/(1-\alpha)}$, there is room for an additional term that diverges more slowly than exponentially. However, as in our analysis we mainly make use of the dominant behaviour, we restrict ourselves to that.
 
\end{remark}

Lastly, we consider $W_{k+i}$ divided by $W_k$ for $i>0$. This quantity is of significant importance when deriving the volume scaling.
As $W_k$ behaves asymptotically as a product of $k$ i.i.d. random variables, we expect that their quotient is also a product of $i$ random variables, as formalised in Lemma~\ref{lem:prod_Wk}:

\begin{lemma}[Quotient of $W_k$] 
\label{lem:prod_Wk} 
Define the i.i.d. sequence of random variables $(P_i)_{i\geq1}$ such that 
\begin{equation}
\label{eq:def_P}
    P_1 = \begin{cases}
     1 &\text{w.p. } \alpha, \\
     \Unif[0,1]^{\frac{1-\alpha}{\alpha}} &\text{w.p. } 1-\alpha.
    \end{cases}
\end{equation}
Fix $\ell>0$. Then, for $k\to\infty$,
\begin{equation}
   \Big(  \frac{W_{k+i}}{W_k}\Big)^\ell_{i=1}  \xrightarrow{d}
    \Bigg(\prod_{j=1}^i P_j\Bigg)^\ell_{i=1}.
\end{equation}
\end{lemma}

\begin{proof}
 Set $(g_i^{\sss(k)})_{i\geq 0} = (W_{k+i}/W_k)_{i\geq 0}$. Then $(g_i^{\sss(k)})_{i\geq 0}$ is again a Markov jump chain. We find  $g_0=1$ and for any $j>0$
 \begin{equation}
     g_j^{\sss(k)} = \begin{cases}
      g_{j-1}^{\sss(k)} &\text{w.p. } 1-R(W_{k+j-1})\theta(W_{k+j-1}), \\
      \theta^{-1}(U_{k+j-1}\theta(W_{k+j-1})) / W_k&\text{w.p. } R(W_{k+j-1})\theta(W_{k+j-1}). \\ 
     \end{cases}
 \end{equation}
 Note that the same bounds as in \eqref{eq:couplingbound_alpha<1_1} and \eqref{eq:couplingbound_alpha<1_2} can be used to define a coupling where, for given small $\eta$, there exist a sufficiently large $K(\eta)$ such that, for $k>K(\eta)$ and $i\geq 0$,
 \begin{equation}
     \frac{\tilde\V_{k+i}(-\eta)}{\tilde\V_{k}(-\eta)}
     \leq
     g_i^{\sss(k)} \leq \frac{\tilde \V_{k+i}(\eta)}{\tilde\V_{k}(\eta)},
 \end{equation}
 where $\tilde\V_k(\eta)$ is given in \eqref{eq:def_V_tilde_alpha<1}.
 We next consider the upper and lower bound processes of $g_j^{\sss(k)}$. By definition of $(\tilde\V_k(\eta))_{k\geq0}$, we find
 \begin{equation}
     \frac{\tilde \V_{k+i}(\eta)}{\tilde\V_{k}(\eta)}
     \overset{d}{=}
     \prod_{j=1}^i 
     P_i^{(\eta)},
 \end{equation}
 where
 \begin{equation}
     P_i^{(\eta)}=
     \begin{cases}
      1 & \text{ w.p. } \alpha(1+\eta),\\
      \Unif[0,1]^{\frac{1-\alpha}{\alpha(1+\eta)}}& \text{ w.p. } 1-\alpha(1+\eta).
     \end{cases}
 \end{equation}
 This means that for $k$ large enough
 \begin{equation}
     \prod_{j=1}^i 
     P_i^{(-\eta)}
     \leq 
      g_i^{\sss(k)} 
     \leq \prod_{j=1}^i 
     P_i^{(\eta)}.
 \end{equation}
 Finally, one can take the limit for $k\to\infty$ and the limit for $\eta\to0$ to find the desired result for fixed $i$.
Joint convergence follows in an identical way.
\end{proof}

\section{Backbone degree sequence}
\label{sec:DBK}
In this section we discuss the scaling of the degrees of backbone vertices. These vertices are special, as they continue the path to infinity. On the one hand, one would expect the backbone degrees to be large, as more edges imply more options to continue this path. On the other hand, the degree cannot be too large, as this path is unique, i.e. all other sub-trees must die out eventually. In the following, we show that this in general results in a \textit{size-biased} or in a \textit{doubly size-biased} offspring distribution, which we define as follows:
\begin{definition}[The (doubly) size-biased distribution]
\label{def:size-biased}
    Let $X$ be a discrete and non-negative random variable with $\expec[X]<\infty$. We define the \emph{size-biased distribution} of $X$, denoted by $X^\star$, by 
    \begin{equation}
        \prob(X^\star =x) = \frac{x\prob(X=x)}{\expec[X]}.
    \end{equation}
    Moreover, if also $\expec[X^2]<\infty$, we define the \emph{doubly size-biased distribution}, denoted by $X^{\star\star}$, by
    \begin{equation}
         \prob(X^{\star\star} =x) = \frac{x^2\prob(X=x)}{\expec[X^2]}.
    \end{equation}
    Specifically, $X^{\star\star}=(X^{\star})^\star$.
\end{definition}

Additionally, we also consider the \textit{effective degree} of the $k$-th backbone vertex $v_k$, by
    \begin{equation}
    \label{eq:def_relevant_degree}
    \hat{D}_{v_k} \overset{d}{=}
    \Binom(D_{v_k}-1, w_k).
    \end{equation}
In the following sections, we give more details on the scaling of the degrees that are relevant for the volume scaling in the main theorem. 
We start by giving a technical note on the specific type of convergence we need for the backbone degrees in Section \ref{sec:uniform_conv}. Details on the degrees along the backbone for $\alpha>2$ are further formalised in Section~\ref{sec:degree_bb_alpha>2}. For $\alpha\in(1,2)$, we see that $D^\star_{v_k}$ and $\expec[D_{v_k}]$ diverge for $k\to\infty$ and we analyse their scaling in Section~\ref{sec:degree_bb_alpha_in(1,2)}.
Lastly, for $\alpha\in(0,1)$, $D_{v_k}$ and $\hat D_{v_k}$ diverge and we analyse their scaling in Section~\ref{sec:DBK_inf_expec}, completing the proof of Theorem~\ref{thm-subthmDBK}.

\subsection{A note on uniform convergence for \mtitle{$\alpha>2$} and \mtitle{ $\alpha\in(1,2)$}}
\label{sec:uniform_conv}
Our analysis is performed conditionally on the future maximum weight, i.e. we condition on $W_k=w_k$, where $w_k\searrow p_c$, in a way that is determined per regime. For $\alpha>2$ and $\alpha\in(1,2)$, Theorem \ref{thm-subthmWk} implies that $W_k$ converges linearly to $p_c$. It turns out that in these regimes it is sufficient to consider specifically $p=p_c(1+a/k)$ and show uniform convergence in $a\in[\varepsilon,1/\varepsilon]$ for $\varepsilon>0$. We therefore introduce the following notation. For some sequences of non-negative constants $(a_i)_{i=1}^k$ and sequence $(n_i)_{i=1}^k$, we write $f=\o[g]{n_1,\ldots,n_k}{a_1,\cdots,a_l}$ if 
\begin{equation}
\label{eq:small-o}
    \lim_{n_1\to\infty}\cdots\lim_{n_k\to\infty}
   \frac{f(\cdot)}{g(\cdot)}
   =0,
\end{equation}
uniformly in $a_i$ on some predefined interval, for all $i\in[l]$.
An important observation is that the results of Lemma \ref{lem:theta_scaling_>1} extend to uniform convergence in $a\in[\varepsilon,1/\varepsilon]$, when we take $p=p_c(1+a/k)$, as formalised in Remark \ref{rem:unif_cnv_theta}.
\begin{remark}[Uniform convergence of $\theta(p)$]
\label{rem:unif_cnv_theta}
By Lemma \ref{lem:theta_scaling_<1}, we find that for given $\gamma>0$, there exists a $\delta>0$ such that $p-p_c<\delta$ implies that
\begin{equation}
    \Big| \frac{\theta(p)}{p-p_c} - C_\theta\Big|<\gamma.
\end{equation}
Setting $p=p_c(1+a/k)$ for $a\in[\varepsilon,1/\varepsilon]$, then, for $k>1/(\varepsilon\delta)$,
\begin{equation}
    \Big| \frac{k\theta(p_c(1+a/k))}{ap_c} - C_\theta\Big|
    \leq 
    \gamma.
\end{equation}
\end{remark}
We apply this result to show convergence of the degrees of backbone vertices. For fixed $k$ the degree distribution of the $k$-th backbone vertex was derived in \eqref{eq:def_thm_DBk}.
This is easily extended to a limiting result:
\begin{lemma}[Backbone vertices are eventually size-biased] 
\label{lem:Backbone_Degree_to_Size_Biased}
Consider the degree sequence $(D_{v_k})_{k\geq0}$ of backbone vertices on the IPC with $\alpha>2$ or $\alpha\in(1,2)$. Let $w_k=p_c(1+a/k)$, then, conditionally on $(W_k)_{k\geq 0}=(w_k)_{k\geq 0}$, uniformly in $a\in[\varepsilon,1/\varepsilon]$ and as $k\to\infty$
\begin{equation}
    D_{v_k}
    \xrightarrow{d}
    X^\star.
\end{equation}
\end{lemma}
\begin{proof}
 For $x>1$ and $k$ fixed, it is straightforward to show that  uniformly in $k$
 \begin{equation}
     0\leq x(1-\hat\theta(w_k))^{x-1}<x.
 \end{equation}
 Moreover, as $\alpha>1$, also $\expec[X]<\infty$. Then, by Remark \ref{rem:unif_cnv_theta}, we find by \eqref{eq:def_thm_DBk} and by dominated convergence,
 uniformly in $a\in[\varepsilon,1/\varepsilon]$,
   \begin{equation}
\label{eq:sizebiasdegree}
\lim_{k\to \infty} \prob(D_{v_k}=l\mid W_k=w_k)=
\frac{ l\prob(X=l)}{\expec[X]}.
\end{equation} 
\end{proof}
In case $\expec[X]<\infty$,
the degrees along the backbone converge to the size-biased distribution, which is well-defined. However, when $\alpha\in(1,2)$, $\expec[D_{v_k}]$ diverges and for $\alpha\in(0,1)$, the size-biased distribution is not even defined. In the following we derive more results for the degrees along the backbone per regime.

\subsection{No exceptionally high-degree backbone vertices for \mtitle{$\alpha>2$}}
\label{sec:degree_bb_alpha>2}
For $\alpha>2$, we condition on $W_k=w_k=p_c(1+a/k)$, for some $a\in[\varepsilon,1/\varepsilon]$ which is a natural choice by Theorem \ref{thm-subthmWk}. We therefore make extensive use of the notation $\o{k}{a}$, from \eqref{eq:small-o}, which we recall denotes an error term that is small in $k\to\infty$ uniformly in $a\in[\varepsilon,1/\varepsilon]$.
We first consider the expectation of $D_{v_k}$ and show that it converges as we would expect:
\begin{lemma}[Expectation of the degrees along the backbone]
\label{lem:DegreeBB_expec_alpha>2}
Take $(D_{v_k})_{k\geq0}$ as the degree sequence of the vertices on the backbone. Then, uniformly in $a\in[\varepsilon,1/\varepsilon]$,
\begin{equation}
    \expec_k[D_{v_k}] \to 
    \expec[X^\star] = \frac{\expec[X^2]}{\expec[X]}.
\end{equation}
\end{lemma}

\begin{proof}
 By \eqref{eq:def_thm_DBk}, 
 \begin{equation}
 \label{eq:expecD_Vk_alpha>2}
     \expec_k[D_{v_k}]=
     \frac{ \expec_k[ X^2(1-\hat\theta(p_c(1+a/k)))^{X-1}]}
     { \expec_k[ X(1-\hat\theta(p_c(1+a/k))^{X-1}]}.
 \end{equation}
 For $x>1$ and $k$ fixed, it is straightforward to show that this is  uniformly in $k$ such that 
$  0<x^2(1-\hat\theta(w_k))^{x-1}<x^2$.
 Moreover, as $\alpha>2$, also $\expec[X^2]<\infty$. Then, we find by Remark $\ref{rem:unif_cnv_theta}$ and dominated convergence, uniformly in $a$,

\begin{equation}
    \lim_{k\to \infty}
    \expec_k[D_{v_k}]
   =
     \frac{ \expec[ X^2(1-\hat\theta(p_c))^{X-1}]}
     { \expec[ X(1-\hat\theta(p_c))^{X-1}]}
     =\frac{\expec[X^2]}{\expec[X]}
     =
     \expec[X^\star]
     .
\end{equation}
\end{proof}
Lemma \ref{lem:DegreeBB_expec_alpha>2} implies that $D_{v_k}$ both converges in distribution and in mean. 
We continue by a similar analysis for $D_{v_k}^\star$, which is key to the upcoming analysis on the volume scaling. We start with showing $D^\star_{v_k}$ is tight and then we derive similar convergence results as we did for $D_{v_k}$:
\begin{lemma}[Size-biased degrees on the backbone]
    Let $D_{v_k}^\star$ denote the size-biased distribution of $(D_{v_k})_{k\geq0}$. Then, conditioned on $W_k=p_c(1+a/k)$, $D_{v_k}^\star$ is a tight sequence.
\end{lemma}
\begin{proof}
 Let $\varepsilon>0$ be given. We show that there exists an $M:=M(\varepsilon)$ such that, uniformly in $k$,
 \begin{equation}
     \prob_k( D_{v_k}^\star > M) <\varepsilon.
 \end{equation}
 We use the definition of the size-biased distribution to write
 \begin{equation}
     \prob_k( D_{v_k}^\star > M)
       =
   \frac{ \expec[D_{v_k}\1\{D_{v_k}>M\}]}{\expec_k[D_{v_k}]}.
 \end{equation}
 As $D_{v_k}$ converges in mean according to Lemma~\ref{lem:DegreeBB_expec_alpha>2}, $(D_{v_k})_{k\geq0}$ is uniformly integrable, meaning that for any $k$ and $a\in[\varepsilon,1/\varepsilon]$, and given $\varepsilon$, there exists an $M$ such that
 \begin{equation}
      \frac{ \expec_k[D_{v_k}\1\{D_{v_k}>M\}]}{\expec_k[D_{v_k}]}
      <\frac{\varepsilon}{\expec_k[D_{v_k}]}<\varepsilon,
 \end{equation}
 which can be made arbitrarily small. This
  implies that $(D_{v_k}^\star)_{k\geq0}$ is a tight sequence of random variables.
\end{proof}
Moreover, we can identify the limit of $D_{v_k}^\star$:

\begin{lemma}[Limiting distribution of $D_{v_k}^\star$]
\label{lem:limiting_distr_Dbk_alpha>2}
 Uniformly in $a$, $D_{v_k}^\star$ converges in distribution to the doubly size-biased distribution of $X$, denoted with $X^{\star\star}$, 
\end{lemma}
\begin{proof}
 By definition 
 \begin{equation}
 \lim_{k\to\infty}
     \prob_k(D_{v_k}^\star= x) 
     =
      \lim_{k\to\infty}
     \frac{x \prob_k(D_{v_k}=x)}{\expec_k[D_{v_k}]}
     =
     \frac{x \prob(X^\star =x)}{\expec[X^\star]}
     =\frac{x^2\prob_k(X=x)}{\expec[X^2]}
     ,
 \end{equation}
 where uniform convergence in $a\in[\varepsilon,1/\varepsilon]$ of the numerator follows from Lemma~\ref{lem:Backbone_Degree_to_Size_Biased} and the denominator from Lemma~\ref{lem:DegreeBB_expec_alpha>2}. We recognise this as the size-biased distribution of $X^\star$.
\end{proof}
We have argued that $D_{v_k}$ is uniformly integrable, which implies that its expected value is small when restricted to only large values. We extend this result to the expectation of the backbone degrees that grow as $k$:
\begin{lemma}[Expected contribution of exceptionally large degrees]
Consider $D_{v_k}$ the degree of the $k$-th backbone vertex for $\alpha>2$. Then there exists a $C=C(a)>0$, that is uniformly bounded in $a$ for $a\in[\varepsilon,1/\varepsilon]$, such that 
\label{lem:UI_DBK_inK}
 \begin{equation}
     \expec_k[D_{v_k} \1\{D_{v_k}  >\varepsilon k\}]
     = Ck^{-(\alpha-2)}(1+\o{k}{a}).
 \end{equation}
\end{lemma}
\begin{proof}
By \eqref{eq:def_thm_DBk}, where we use \aprefb{3}{sec:app_mass_tail} and bounded convergence, we find 
\begin{equation}
\begin{aligned}
    &\expec_k[D_{v_k} \1\{D_{v_k}  >\varepsilon k\}]
    = \frac{1}{\expec[X]}\sum_{y=\frac{\varepsilon k\hat{\theta}(p_c(1+a/k)) }{\hat\theta(p_c(1+a/k))}}^\infty
    y^2(1-\hat\theta(p_c(1+a/k)))^{y-1}\\
  &\quad=\frac{1}{\expec[X]}\int_{y=\frac{\varepsilon k\hat\theta(p_c(1+a/k)) }{\hat\theta(p_c(1+a/k))}}^\infty
    y^2(1-\hat\theta(p_c(1+a/k)))^{y-1}\dif y
   (1+\o{k}{a})
   \\&\quad=\frac{\alpha c_{\sss X}}{\expec[X]\hat\theta(p_c(1+a/k))^{2-\alpha}} \int_{ C_\theta p_c a}^\infty 
    \e^{-z}z^{\alpha}\dif z(1+\o{k}{a}).
\end{aligned}
\end{equation}
By Lemma~\ref{lem:theta_scaling_>1} and Remark \ref{rem:unif_cnv_theta}, we then conclude
\begin{equation}
\begin{aligned}
     \expec_k[D_{v_k} \1\{D_{v_k}  >\varepsilon k\}]
     &=
     (p_c a/k)^{-(2-\alpha)}\frac{\alpha c_X}{\expec[X]} \int_{C_\theta p_c a}^\infty 
    \e^{-z}z^{\alpha}\dif z(1+\o{k}{a}) 
    \\&= C(a) k^{-(\alpha-2)}(1+\o{k}{a}).
\end{aligned}
\end{equation}
\end{proof}

\subsection{Structure exceptionally high-degree backbone vertices for \mtitle{$\alpha\in(1,2)$}}
\label{sec:degree_bb_alpha_in(1,2)}

Similarly to $\alpha>2$, for $\alpha\in(1,2)$,
we condition on $W_k=w_k=p_c(1+a/k)$ and use the $\o{k}{a}$ notation from \eqref{eq:small-o}. Then, the limiting distribution of $D_{v_k}$ is given by $X^{\star}$, however now $\expec[X^{\star}]=\infty$. The following lemma describes how the expected degrees diverge to infinity:

\begin{lemma}[Divergence speed of expected degrees for $\alpha\in(1,2)$]
\label{lem:rare_degree_Vertices_infinite_Var}
Consider $(D_{v_k})_{k\geq 0}$ the degree sequence along the backbone of the IPC for $\alpha\in(1,2)$. Then, there exists a $C=C(a)>0$ such that $C(a)$ is uniformly bounded in $a\in[\varepsilon,1/\varepsilon]$ and
\begin{equation}
\label{eq:scaling_expec_BB_infvar}
\begin{aligned}
    \lim_{k\to\infty}\expec_k[D_{v_k}]k^{-(2-\alpha)/(\alpha-1)}
    = C(a).
\end{aligned}
\end{equation}
\end{lemma}

\begin{proof}
Recall $\expec[D_{v_k}]$ from \eqref{eq:expecD_Vk_alpha>2} and note that it also holds for $\alpha\in (1,2)$.
In this regime, the denominator converges to $\expec[X]<\infty$, uniformly in $a\in[\varepsilon,1/\varepsilon]$ for $k\to\infty$, by dominated convergence. However, now the numerator diverges. 
For the scaling behaviour of the numerator we use a result from \cite[Equation 8.1.12] {Bingham1987} 
\begin{equation}
\begin{aligned}
\label{eq:secondmomentX}
&\expec_k[X^2(1-\hat{\theta}(w_k))^{X-1}]
\\&\quad=\alpha c_{\sss X}\Gamma(2-\alpha)\Big(-\log(1-\hat{\theta}(p_c(1+a/k)))\Big)^{(\alpha-1)-1}(1+\o{k}{a})
\\&\quad
=\alpha c_{\sss X}\Gamma(2-\alpha)(C_\theta p_c (p_ca/k)^{1/(\alpha-1)})^{\alpha-2}(1+\o{k}{a})
\\&\quad=C(a)k^{(2-\alpha)/(\alpha-1)}(1+\o{k}{a}).
\end{aligned}
\end{equation}	
The last equality follows from a Taylor expansion of $\log(1-\hat\theta(p_c(1+a/k)))$ and the asymptotic behaviour of $\theta(p)$, in Remark \ref{rem:unif_cnv_theta}.
We note that $C(a)$ is bounded for $a\in[\varepsilon,1/\varepsilon]$.
Substituting this result in \eqref{eq:expecD_Vk_alpha>2} and taking the limit for $k\to\infty$ shows the result.
\end{proof}
We additionally show that the expected degree restricted to only degrees that grow as $k^{1/(\alpha-1)}$ diverges as well under the same scaling. This implies that the divergence of $\expec_k[D_{v_k}]$ is driven by degrees that grow at least as fast as $O(k^{1/(\alpha-1)})$.
\begin{lemma}[Scaling expectation of large degrees]
\label{lem:scaling_largedegrees_alphain(1,2)}
Consider $D_{v_k}$ the degree sequence of the backbone for $\alpha\in(1,2)$. Then there exists a $C=C(a,L)>0$ such that $C(a,L)$ uniformly bounded for $a\in[\varepsilon,1/\varepsilon]$ and $L>\varepsilon$. Moreover 
\begin{equation}
\begin{aligned}
   & \expec_k[ D_{v_k} \1\{D_{v_k} > L k^{1/(\alpha-1)}\}]
  & = 
  C(a,L) k^{-(\alpha-2)/(\alpha-1)}
   (1+\o{k}{a,L}).
    \end{aligned}
\end{equation}
\end{lemma}
\begin{proof}
We first observe that by \aprefb{3}{sec:app_mass_tail}
\begin{equation}
\begin{aligned}
  &  \expec_k[ D_{v_k} \1\{D_{v_k} > L k^{1/(\alpha-1)}\}]\\&\quad=
    \frac{ \sum_{x>Lk^{1/(\alpha-1)}} x^2(1-\hat\theta(p_c(1+a/k)))^{x-1}\prob(X=x)}{\expec[X(1-\hat\theta(p_c(1+a/k)))^{X-1}]}
    \\&\quad
    = \frac{\alpha(1+\o{k}{a,L})}{\expec[X]}
    \int^\infty_{Lk^{1/(\alpha-1)}} (1-\hat\theta(p_c(1+a/k))^{x-1}c_{\sss X}x^{-\alpha+1}\dif x.
\end{aligned}
\end{equation}
Substituting $z=-x\log(1-\hat\theta(p_c(1+a/k)))$ and using a Taylor expansion of $\log(1-\hat\theta(p_c(1+a/k)))$ implies
\begin{equation}
\begin{aligned}
   & \expec_k[ D_{v_k} \1\{D_{v_k} > L k^{1/(\alpha-1)}\}]\\&\quad=
    \frac{\alpha c_{\sss X}(1+\o{k}{a,L})}{\expec[X]}
    \hat\theta(p_c(1+a/k))^{\alpha-2}
    \int^\infty_{Lk^{1/(\alpha-1)}\hat\theta(p_c(1+a/k)))}
    \e^{-z}z^{-\alpha+1}\dif z
     \\&
    \quad=
     \frac{\alpha c_{\sss X}(1+\o{k}{a,L})}{\expec[X]}
     (C_\theta p_c (p_ca/k)^{1/(\alpha-1)})^{\alpha-2}
     \int_{LC_\theta p_c (p_c a)^{1/(\alpha-1)}}^\infty
     \e^{-z}z^{-\alpha+1}\dif z
    \\&
    \quad=
     \frac{\alpha c_{\sss X}(1+\o{k}{a,L})}{\expec[X]}
     k^{-(\alpha-2)/(\alpha-1)}
     \int_{L}^\infty
     \e^{-\frac{z}{C_\theta p_c (p_c a)^{1/(\alpha-1)}}}z^{-\alpha+1}\dif z.
\end{aligned}
\end{equation}
It is straightforward to verify that the constant is uniformly bounded for $a\in[\varepsilon,1/\varepsilon]$ and $L>\varepsilon$.
\end{proof}
As the limiting distribution of $D_{v_k}$ has an infinite first moment, the size-biased distribution of $D_{v_k}$ diverges for $k\to\infty$. We identify its scaling as follows:
\begin{lemma}[Size-biased degree scaling and limit for $\alpha\in(1,2)$]
\label{lem:SizeBiased_distr_BB_infvar}
Consider $\alpha\in(1,2)$ and let $(D_{v_k})_{k\geq 0}$ denote the degrees sequence of the backbone. Recall that $D^\star_{v_k}$ denotes the size-biased distribution of $D_{v_k}$. Then conditioned on $W_k=p_c(1+a/k)$, uniformly in $a\in[\varepsilon,1/\varepsilon]$, 
\begin{equation}
    \hat{\theta}(p_c(1+a/k))
    D^\star_{v_k}
    \xrightarrow{d} \GammaD(2-\alpha, 1). 
\end{equation}
\end{lemma}
\begin{proof} 
We compute
\begin{equation}
\label{eq:degree_alpha(1,2)_aux_0}
\begin{aligned}
    \prob_k(\hat{\theta}(p_c(1+a/k))D_{v_k}^\star >x) 
    &= \sum_{y > x/\hat{\theta}(p_c(1+a/k))}
    \frac{ \prob_k(D_{v_k}=y) y}{\expec_k[D_{v_k}]}\\&= 
    \sum_{y > x/\hat{\theta}(p_c(1+a/k))}
    \frac{ y^2(1-\hat\theta(p_c(1+a/k)))^{y-1}\prob(X=y)}{
    \expec_k[D_{v_k}]\expec_k[X(1-\hat{\theta}(p_c(1+a/k)))^{X-1}]},
\end{aligned}    
\end{equation}
where we know the asymptotic behaviour of $\expec[D_{v_k}]$ from Lemma~\ref{lem:rare_degree_Vertices_infinite_Var}, and it remains to find an expression for the sum in the numerator. By \aprefb{3}{sec:app_mass_tail} and Lemma~\ref{lem:rare_degree_Vertices_infinite_Var} we can rewrite this to
\begin{equation}
\begin{aligned}
    & \frac{\alpha c_{\sss X}\hat{\theta}(p_c(1+a/k))^{\alpha-2} }{\Big(\frac{\alpha c_{\sss X}\hat{\theta}(p_c(1+a/k))^{\alpha-2} \Gamma(2-\alpha) }{\expec[X]} \Big)\expec[X]}
    \int^\infty_{x}
    \e^{-z} z^{(2-\alpha)-1} \dif z(1+\o{k}{a})
    \\&\qquad\to\int^\infty_{x}
    \e^{-z} \frac{z^{(2-\alpha)-1}}{\Gamma(2-\alpha)} \dif z,
\end{aligned}
\end{equation}
for $k\to\infty$ and uniformly in $a\in[\varepsilon,1/\varepsilon]$.
Here, we recognise the tail distribution of a gamma random variable, which shows the claim.
\end{proof}
Finally, we consider the size-biased version of the effective degree $(\hat D_{v_k})^\star$ and show that it 
scales as $D_{v_k}^\star$, but to a different limit. This is formalised in the following result:

\begin{lemma}[Convergence of $
 k^{-1/(\alpha-1)}(\hat D_{v_k})^\star$]
\label{lem:converence_binomial_sizebiased_alpha(1,2)}
Suppose that, conditioned on $W_k=p_c(1+a/k)$, $k^{-1/(\alpha-1)}D_{v_k}^\star\xrightarrow{d} X^\star$, a non-trivial limiting distribution for $k\to\infty$ and uniformly in $a\in[\varepsilon,1/\varepsilon]$. Recall the effective degree of $\hat D_{v_k}$ from \eqref{eq:def_relevant_degree}. Then conditioned on $W_k=p_c(1+a/k)$ and uniformly in $a\in[\varepsilon,1/\varepsilon]$,
\begin{equation}
    \frac{ (\hat D_{v_k})^\star}{k^{1/(\alpha-1)}}
    \xrightarrow{d}
    p_cX^\star.
\end{equation}
\end{lemma}
\begin{proof}
 The proof is based on the convergence of the Laplace transforms. For the size-biased distribution we find
 \begin{equation}
 \label{eq:converence_binomial_sizebiased_alpha(1,2)_1}
     \expec[
     \e^{-sk^{-1/(\alpha-1)}(\hat D_{v_k})^\star }
     ] =
     \sum_{i=0}^\infty 
     \e^{-sk^{-1/(\alpha-1)}i} \frac{ \prob(\hat D_{v_k} = i)  i}{\expec[\hat D_{v_k}]}
     =
     \frac{\expec[\hat D_{v_k}\e^{-s k^{-1/(\alpha-1)}\hat D_{v_k}}] }{\expec[\hat D_{v_k}]}.
 \end{equation}
We simplify the expectation in the numerator first. We find 
\begin{equation}
    \expec[\hat D_{v_k}\e^{-s k^{-1/(\alpha-1)} \hat D_{v_k}}]
    =
    -\frac{\partial}{\partial \hat s}
    \expec[\e^{-\hat s \hat D_{v_k}}]_{\hat s = sk^{-1/(\alpha-1)}}.
\end{equation}
Write $w_k=p_c(1+a/k)$.
Recalling that $\hat D_{v_k}$ is a binomial distribution, we can substitute the known Laplace transform 
\begin{equation}
\begin{aligned}
  \expec\Big[\hat D_{v_k}\e^{-s k^{-1/(\alpha-1)}\hat D_{v_k}}\Big]
    &=
   - \frac{\partial}{\partial \hat s}
    \expec\Big[
    (1-w_k+w_k\e^{-\hat s})^{D_{v_k}}
    \Big]\\&=w_k\e^{-\hat s}
    \expec\Big[ D_{v_k} (1-w_k(1-\e^{-\hat s}))^{D_{v_k}}\Big]\\&
    =w_k\e^{-\hat s} \expec\Big[D_{v_k}\e^{-\hat sw_k D_{v_k}}\Big](1+\o{k}{a}).
\end{aligned}
\end{equation}
Substituting $\hat s = sk^{-1/(\alpha-1)}$ so that $\e^{-\hat s}=1+\o{k}{a}$, implies
 \begin{equation}
 \begin{aligned}
     \expec\Big[
     \e^{-s(k^{-1/(\alpha-1)}\hat D_{v_k})^\star}
    \Big]& =
     \frac{
       w_k\expec\Big[D_{v_k}\e^{-s w_k k^{-1/(\alpha-1)}D_{v_k}}
       \Big]}{w_k\expec[ D_{v_k}]}(1+\o{k}{a})\\&=
\expec\Big[ \e^{-sp_c k^{-1/(\alpha-1)}D_{v_k}^\star }\Big](1+\o{k}{a}),
\end{aligned}
 \end{equation}
 which converges to $\expec[\e^{-sp_c X^\star}]$ uniformly in $a\in[\varepsilon,1/\varepsilon]$ by the conditions of the lemma and where
 the last equality follows from bounded convergence.
 This concludes the proof.
\end{proof}

\subsection{Structure degrees backbone vertices for \mtitle{$\alpha\in(0,1)$}}
\label{sec:DBK_inf_expec}
We next investigate the scaling of the degrees along the backbone in the infinite-mean setting. Here, it suffices to condition on $W_k=w_k$, where $w_k$ is some sequence such that $w_k\searrow p_c$ holds. We therefore recall the notation $\o{k}{}$ from \eqref{eq:small-o} for an error small in $k\to\infty$. 
We start with the scaling of $D_{v_k}$:

\begin{proposition}[Degrees along the backbone for infinite-mean setting]
\label{prop-degrees-BB-(0,1)}
Fix $\alpha\in(0,1)$ and recall $\hat\theta(\cdot)$ from \eqref{eq:def_hattheta}.  Then,
conditioned on $W_k=w_k$,
    \begin{equation}
    \hat{\theta}(w_k) D_{v_k}
    \xrightarrow{d}
    \emph{\texttt{Gamma}}(1-\alpha,1).
    \end{equation}
\end{proposition}

\begin{proof}
We use the expression of the mass function of $D_{v_k}$ in \eqref{eq:def_thm_DBk} to find
\begin{equation}
\label{eq:Dbk_alpha<1_auxeq_0}
    \prob_k( \hat{\theta}(w_k) D_{v_k} > x)
    =
    \sum_{y>x/\hat{\theta}(w_k)} \frac{ y(1-\hat{\theta}(w_k))^{y-1} 
    \prob(X=y)}{\expec_k[X(1-\hat{\theta}(w_k))^{X-1}]},
\end{equation}
where we consider the numerator and denominator separately. The expectation in the denominator diverges for $p\searrow p_c$. We consider its asymptotic behaviour based on \cite[Equation 8.1.12]{Bingham1987} to find
\begin{equation}
\label{eq:Asymptotic_behaviour_E[1-ptheta]_alpha<1}
    \expec_k[X(1-\hat\theta(w_k))^{X-1}]
    = c_{\sss X}\alpha\Gamma(1-\alpha) \hat \theta(w_k)^{-(1-\alpha)}(1+\o{k}{}).
\end{equation}
By \aprefb{3}{sec:app_mass_tail}, we find an expression for the numerator as
\begin{equation}
\begin{aligned}
\label{eq:Dbk_alpha<1_auxeq_0-1}
    &\sum_{y>x/\hat{\theta}(w_k)} y(1-\hat{\theta}(w_k))^{y-1} 
    \prob(X=y)\\&\quad = (1+\o{k}{})\alpha c_{\sss X}\hat\theta(w_k)^{-(1-\alpha)}\int_x^\infty \e^{-z} z^{-\alpha} \dif z .
    \end{aligned}
\end{equation}
By substitution of \eqref{eq:Asymptotic_behaviour_E[1-ptheta]_alpha<1} and \eqref{eq:Dbk_alpha<1_auxeq_0-1} in \eqref{eq:Dbk_alpha<1_auxeq_0} gives
\begin{equation}
\begin{aligned}
     &\prob_k( \hat{\theta}(W_k) D_{v_k} > x)
     =
     (1+\o{k}{})\int_x^\infty \frac{\e^{-z}}{\Gamma(1-\alpha)} z^{-\alpha} \dif z ,
     \end{aligned}
\end{equation}
which is exactly the tail distribution of a gamma distribution with shape parameter $1-\alpha$ as $k\to\infty$.
\end{proof}

We conclude with the scaling of the effective degree $\hat D_{v_k}$, as we recall from \eqref{eq:def_relevant_degree}, in distribution and in expectation. 
For $\alpha>2$ or $\alpha\in(1,2)$ the scaling is the same as $D_{v_k}$ but as 
$w_k\to0$ for $\alpha\in(0,1)$, we expect that $\hat D_{v_k}$ scales differently, which is formalised as follows:

\begin{lemma}[Convergence of the effective degree sequence $\hat D_{v_k}$]
\label{lem:convergence_hatDBK_alpha(0,1)}
\label{lem:conv_expecdegree_(0,1)}
Fix $\alpha\in(0,1)$. Then, conditioned on $W_k=w_k$,
\begin{equation}
    \theta(w_k)
    \Binom( D_{v_k},w_k) \xrightarrow{d}
   \GammaD(1-\alpha,1)
\end{equation}
and
\begin{equation}
    \expec[\theta(w_k)\hat D_{v_k}]
    \to  1-\alpha.
\end{equation}
\end{lemma}
\begin{proof}
 We prove this using the Laplace transform, and use the Laplace transform of a binomial random variable, as follows: 
 \begin{equation}
 \begin{aligned}
     \expec_k[\e^{-s \theta(w_k)   \Binom( D_{v_k},w_k)} ] &
     =\expec_k[ (1-w_k(1-\e^{-s \theta(w_k)}) )^{D_{v_k}}]
    \\& =
     \expec_k[ \e^{D_{v_k} \log(1- w_k(1-\e^{-s \theta(w_k)}))}]
     \\&=
     \expec_k[ \e^{-sw_k\theta(w_k) D_{v_k}(1+o(1)) } ]
     \to   \expec_k[ \e^{-s\GammaD(1-\alpha,1) }],
    \end{aligned}
 \end{equation}
where the last convergence result holds by Proposition~\ref{prop-degrees-BB-(0,1)}. 
 
For the expectation,
we recall that by \eqref{eq:Asymptotic_behaviour_E[1-ptheta]_alpha<1}
and by the monotone density theorem \cite[Theorem 1.2.9]{Mikosch1999}
\begin{equation}
    \expec_k[X^2(1-\hat\theta(w_k))^{X-2}]
    = c_{\sss X}\alpha\Gamma(1-\alpha)(1-\alpha) \hat \theta(w_k)^{-(2-\alpha)}(1+\o{k}{}).
\end{equation}
This implies that for $k$ large enough
\begin{equation}
    \expec_k[\theta(w_k)\hat D_{v_k}] = 
    \frac{\expec_k[X^2(1-\hat\theta(w_k))^{X-2}]}{\expec_k[X(1-\hat\theta(w_k))^{X-1}]}
    \to
   1-\alpha.
\end{equation}
\end{proof}

\section{Volume of finite forests attached to the backbone}
\label{sec:TP_scale}
So far, we have focused on properties of the backbone, such as the scaling of the maximal weight accepted after the $k$-th backbone vertex and the degree distribution along the backbone. These results are key to analyse the (conditionally finite)  forests that are attached to the backbone, which is done in this section. We provide the detailed proof of Theorem~\ref{thm:Cs_scaling} and start by sketching the proof.

\subsection{Overview of the proof}
\label{sec-size-trees-off-backbone-overview}
The total volume of the forest attached to the $k$-th backbone vertex is equal to the sum of $\hat{D}_{v_k}$ i.i.d.\ total progenies of the attached trees, where we recall $\hat D_{v_k}$ the effective degree as defined in \eqref{eq:def_relevant_degree}.
In the following, we ignore the effect of one edge being used to continue the path to infinity. This effect is represented by the $(-1)$ in the binomial in the definition of $\hat D_{v_k}$, but its contribution is negligible. 

By the observations of Section~\ref{sec:FiniteForest_properties_BB}, we recall that for fixed $k$ all attached trees are sub-critical with the dual offspring distribution $\tilde X^{\sss(k)}$ in \eqref{eq:def_tildeX2} that is unfortunately no longer exactly binomial.
Conditionally on $W_k=w_k$, we define the law of the size of the forest attached to the $k$-th backbone vertex by
\begin{equation}
    \label{eq:def-H_k}
    H_k=\sum_{i=1}^{\hat{D}_{v_k}} \tilde{T}^{\sss(k)}_{i},
\end{equation}
where $(\tilde{T}^{\sss(k)}_{i})_{i\geq 1}$ are  i.i.d.\ total offsprings of a BP tree starting with one individual, and percolated on $w_k>p_c$, conditionally on extinction. By the generalised hitting time theorem  \cite[Theorem 3.14]{VanderHofstad2016},
\begin{equation}
    \label{tilde-Tik-def}
    \prob_k(H_k=n\mid  \hat{D}_{v_k}=\ell) =\frac{\ell}{n}\prob(\tilde{X}_1^{\sss(k)}+\cdots+\tilde{X}_n^{\sss(k)}=n-\ell).
\end{equation}
Note that, for every $n$ and $\ell\in\{1,...,n\}$,
\begin{equation}
    \label{eq:massfunction_eta}
    \prob(\tilde{X}_1^{\sss(k)}+\cdots+\tilde{X}_n^{\sss(k)}=n-\ell)=\frac{1}{\eta(w_k)^\ell} \prob(X_1^{\sss(k)}+\cdots+X_n^{\sss(k)}=n-\ell),
\end{equation}
where now $(X_i^{\sss(k)})_{i\geq 1}$ is an i.i.d.\ sequence of $\Binom(X,w_k)$ random variables. As a result,
\begin{equation}
    \label{tilde-Hk-law-binomial}
    \prob_k(H_k=n\mid \hat{D}_{v_k}=\ell)=
    \frac{\ell}{n\eta(w_k)^\ell} 
    \prob_k(X_1^{\sss(k)}+\cdots+X_n^{\sss(k)}=n-\ell),
\end{equation}
so that also 
\begin{equation}
    \label{tilde-Hk-law-binomial-b}
    \prob_k(H_k=n)=
    \expec_k\bigg[\frac{\hat{D}_{v_k}}{n\eta(w_k)^{\hat{D}_{v_k}}} 
    \prob_k(X_1^{\sss(k)}+\cdots+X_n^{\sss(k)}=n-\hat{D}_{v_k}\mid \hat{D}_{v_k})\bigg].
\end{equation}
Also, the distribution of $(X_i^{\sss(k)})_{i\geq 1}$ only depends on $k$ through the parameter $w_k$ that appears as the binomial success probability. However, the above rewrite comes at the expense of a factor $1/\eta(w_k)^\ell$ in \eqref{tilde-Hk-law-binomial}, which can be inconvenient when $\hat{D}_{v_k}=\ell$ is large. 
Our analysis shows that, uniformly in $x>\varepsilon$, 
\begin{equation}
\label{aim-Hk}
     \prob_k(H_k=n)=
     \begin{cases}
      \Theta(k^{-3}) &\text{ for $\alpha>2$ and $n=\lceil x k^2\rceil$},\\
      \Theta(k^{-(2\alpha-1)/(\alpha-1)}) &\text{ for $\alpha\in(1,2)$ and $n=\lceil xk^{\alpha/(\alpha-1)}\rceil$},\\
      \Theta(w_k^{\alpha/(1-\alpha)}) &\text{ for $\alpha\in(0,1)$ and $n=\lceil x w_k^{-\alpha/(1-\alpha)}\rceil$}.\\
     \end{cases}
\end{equation}
By \eqref{tilde-Hk-law-binomial-b}, we see that the binomial distribution plays a key role in this analysis. In the following we make extensive use of its local limit approximation, on which we provide the details here.

\paragraphi{Binomial local limit approximations}
Let $S_m\overset{d}{=} \Binom(m,p)$. Then, a local central limit approximation shows that
\begin{equation}
    \label{Bin-LCLT-pre}
    \prob(S_m=n)=\frac{1}{\sqrt{2 \pi mp(1-p)}}{\mathrm{e}}^{-(n-mp)^2/[2mp(1-p)]}(1+o(1)),
\end{equation}
when $|n-mp|=o((mp)^{2/3})$, uniformly in $n$ and $p$. Such approximations have a long history, we refer to the books by Bollob\'as \cite{Bollobas2001} and Janson, \strokeL uczak and Ruci\'nsky \cite{JanLucRuc00}. In particular, the upper bound in \eqref{Bin-LCLT-pre} follows from \cite[Theorem 1.2]{Bollobas2001}, the lower bound from \cite[Theorem 1.5]{Bollobas2001}, and an analysis of the error terms present there, showing that they are $o(1)$ when $|n-mp|=o((mp)^{2/3})$. \\
We apply \eqref{Bin-LCLT-pre} in a slightly unusual setting where $m=Q_n+D_{v_k}^\star-1$ is {\em random} itself, while $n$ remains fixed. Thus, we next massage \eqref{Bin-LCLT-pre} in a form that is more convenient to us. Note that when $|n-mp|=o((mp)^{2/3})$, then also $mp=n(1+o(1))$, so that we may replace $\sqrt{2 \pi mp(1-p)}$ in \eqref{Bin-LCLT-pre} by $\sqrt{2 \pi n(1-p)}(1+o(1)).$ Note further that
$1/(mp)=1/n +(mp-n)/(nmp)$, so that, in this regime,
    \begin{equation}
    \begin{aligned}
    \label{exp-simpl-Bin}
    \frac{(n-mp)^2}{mp}
    &=\frac{(n-mp)^2}{n}
    +O\Big(\frac{|n-mp|^3}{nmp}\Big)
    \\&=\frac{(n-mp)^2}{n}
    +O\Big(\frac{|n-mp|^3}{(mp)^2}\Big)
    =\frac{(n-mp)^2}{n}
    +o(1).    
    \end{aligned}
    \end{equation}
Applying these two approximations yields
 \begin{equation}
    \label{Bin-LCLT}
    \prob(S_m=n)=\frac{1}{\sqrt{2 \pi n(1-p)}}{\mathrm{e}}^{-(n-mp)^2/[2n(1-p)]}(1+o(1)).
\end{equation}
Of course, \eqref{Bin-LCLT} only covers a part of the range of possible $n$. For $n$ for which $|n-mp|\geq (mp)^{7/12}$, say, \cite[Corollary 1.4(i)]{Bollobas2001} shows that 
\begin{equation}
    \label{mod-dev-Sn-Bin-pre}
    \prob(|S_m-mp|\geq h)\leq \frac{\sqrt{mp(1-p)}}{h}{\mathrm{e}}^{-h^2/[2mp(1-p)]}(1+o(1))).
\end{equation}
Here, the power $7/12$ in $|n-mp|\geq (mp)^{7/12}$ is chosen to be in between $1/2$ and $2/3$. The bound in \eqref{mod-dev-Sn-Bin} for $h=(mp)^{7/12}$ is bounded by $\e^{-(mp)^{1/6}},$ which is so small that we will ignore it. Indeed, all the error terms arising from using \eqref{exp-simpl-Bin} in \eqref{tilde-Hk-law-binomial-b} 
will be much larger than this. Again using that $mp=n(1+o(1))$ and \eqref{exp-simpl-Bin}, we arrive at
\begin{equation}
    \label{mod-dev-Sn-Bin}
    \prob(|S_m-mp|\geq h)\leq \frac{\sqrt{n(1-p)}}{h}{\mathrm{e}}^{-h^2/[2n(1-p)]}(1+o(1))).
\end{equation}
This leads to an error term that is negligible under the scaling regime we consider.

\paragraphi{Central limit approximations}
A second useful approximation that we apply here is based on the established central limit theorem. 
\begin{lemma}[Normal approximations for finite means]
\label{lem:Norm_approx_fin_mean}
\label{lem:Normal_approx_alpha(0,1)}
Consider variables $z_k:=z$, $m_k:=m$, $\sigma_k^2:=\sigma^2>0$ and $b_k:=b\in\mathbb{R}$ such that $x+m \to\infty$ for $k\to\infty$ and $\sigma/m = \o{k}{}$. Then for all bounded and continuous functions $y\mapsto \chi(y)$,
\begin{equation}
 \label{eq:Normal_approx_chi}
     \sum_{\ell \geq 1}
     \chi( (\ell-z)b/m )
     \e^{-
     \frac{(\ell-z-m)^2}{2\sigma^2}}=
     \chi(b)\sqrt{2\pi\sigma^2}(1+\o{k}{}).
\end{equation}
\end{lemma}

\begin{proof}[Proof of Lemma~\ref{lem:Norm_approx_fin_mean}] 
We note that since $\chi$ is bounded 
\begin{equation}
    \sum_{\ell\geq 1} 
    \chi\Big(b
    \frac{\ell-x}{m}\Big) \e^{ -\frac{ (\ell-x-m)^2}{2\sigma^2}}
     =
    \int_1^\infty 
    \chi\Big(b
    \frac{\ell-x}{m}\Big) \e^{ -\frac{ (\ell-x-m)^2}{2\sigma^2}}
    \dif \ell(1+\o{k}{}).
\end{equation}
by Euler-Maclauren, where the error is small for $x+m$ large. We recognise the density of a normal distribution and define $N_1\sim\mathcal{N}(m+x, \sigma^2)$,
a normally distributed random variable with mean $m+x$ and variance $\sigma^2$,
so that by a proper scaling, we find 
\begin{equation}
    \int_1^\infty 
    \chi\Big(b
    \frac{\ell-x}{m}\Big) \e^{ -\frac{ (\ell-x-m)^2}{2\sigma^2}}\dif \ell
    (1+\o{k}{})= \sqrt{2\pi\sigma^2}
     \expec\bigg[ 
    \chi\Big(b
    \frac{N_1-x}{m}\Big)
    \1\{N_1>1\}
    \bigg].
\end{equation}
Let $N_2\sim \mathcal{N}(0,1)$, then
this can be written as 
\begin{equation}
\label{eq:app_proofnormal_aux1}
    \sqrt{2\pi\sigma^2}
     \expec \bigg[ 
    \chi\Big(b\Big(1+
    \frac{\sigma}{m}N_2\Big)\Big)
    \1\{x+m+\sigma N_2>1\} \bigg]=
    \sqrt{2\pi\sigma^2}
    \chi(b)
   (1+\o{k}{}),
\end{equation}
where the error term converges in $k$
as $\sigma/mN_2\xrightarrow{\prob}0$ 
by the assumption $\sigma/m=\o{k}{}$ and $\1\{x+m+\sigma N_2>1\}\xrightarrow{\prob}0$ by the assumption
$m+x\to\infty$. Furthermore, 
because $\chi$ is continuous and bounded, we can take the limit in the expectation by bounded convergence.
 \end{proof}

\paragraphi{Mass function and tail distributions}
Our main result describes the scaling of the tail distribution of $H_k$. As our main tool, namely the random-walk hitting time theorem in \eqref{tilde-Tik-def}, addresses the mass function, we need to translate the scaling of the mass function to the scaling of the distribution. We provide the non-trivial bridge between the two, so that in the following parts it suffices to focus on the mass function:
\begin{lemma}[Mass and distribution functions]
\label{App:lem:scaling_mass_distr}
Consider a sequence of discrete random variables $(X_k)_{k\geq 0}$. Take constants $x$ and $0<\beta<2$ and $\gamma\geq 2$. Assume, uniformly in $x>\varepsilon$ and $a\in[\varepsilon,1/\varepsilon]$,
\begin{equation}
\label{eq:app:scaling_mass_assumption}
    k^{\beta+\gamma} \prob_k( X= \lceil xk^{\gamma} \rceil )
    \to g(x),
\end{equation}
where $g(x)$ is some Riemann integrable function on $(0,\infty)$. Furthermore, assume that for all $n\in\mathbb{N}$
\begin{equation}
\label{eq:mass_func_unif_bound}
    \prob_k(X=n) \leq Ck^{\gamma-2}n^{-\frac{\beta+2\gamma-2}{\gamma}},
\end{equation}
for some fixed constant $C$.
Then, for every $x>\varepsilon$,
\begin{equation}
    k^\beta \prob_k(X>xk^\gamma) \to \int_x^\infty g(y) \dif y. 
\end{equation}
\end{lemma}
\begin{proof}
We first find
\begin{equation}
\begin{aligned}
     k^{\beta} \prob_k(X>xk^\gamma)
     &= \sum_{i=\lceil xk^\gamma\rceil}^\infty
      k^{\beta} \prob_k(X=i)
       =
     \int_{i=\lceil xk^\gamma\rceil}^\infty
       k^{\beta} \prob_k(X=\lceil i \rceil )\dif i
     \\& =  \int_{i= x}^\infty k^{\beta+\gamma}
       \prob_k(X=\lceil ik^\gamma \rceil )\dif i(1+\o{k}{a}).
\end{aligned}
\end{equation}
If we can show that $k^{\beta+\gamma} \prob_k(X=\lceil ik^\gamma \rceil )\leq q(i)$ where $q(i)$ is integrable on $[\varepsilon,\infty]$, the result follows from bounded convergence. Based on \eqref{eq:mass_func_unif_bound}, we find
\begin{equation}
    k^{\beta+\gamma}
       \prob_k(X=\lceil ik^\gamma \rceil )
       \leq 
        Ck^{\beta+\gamma}
       k^{\gamma-2}(ik^\gamma)^{-\frac{\beta+2\gamma-2}{\gamma}}
       = C i^{-\frac{\beta+2\gamma-2}{\gamma}}.
\end{equation}
Under the assumption of $\gamma\geq2$ and $\beta<2$, we find that 
\begin{equation}
    \int_{i=x}^\infty  C i^{-\frac{\beta+2\gamma-2}{\gamma}}\dif i
    \leq  \int_{i=\varepsilon}^\infty Ci^{-(1+\beta/2)}\dif i<\infty.
\end{equation}
This concludes the proof.
\end{proof}

\paragraphi{Organisation of this section}
In the following we formalise and prove the claim in \eqref{aim-Hk}. 
For $\alpha>2$ and $\alpha\in(1,2)$, we restrict ourselves to the conditional distribution of $H_k$, given $W_k=p_c(1+a/k)$. 
The analysis is split per regime, where we discuss $\alpha>2$ in Section~\ref{sec-size-trees-off-backbone-alpha>2},  $\alpha\in(1,2)$ in Section~\ref{sec-size-trees-off-backbone-alpha(1,2)} and $\alpha\in(0,1)$ in Section~\ref{sec-size-trees-off-backbone-alpha(0,1)}. 
Besides the claim in \eqref{aim-Hk}, we additionally show that $n^{3/2}\prob_k(H_k=n)$ for $\alpha>2$, and $n^{1+1/\alpha}k^{-(2-\alpha)/(\alpha-1)}\prob_k(H_k=n)$ for $\alpha\in(1,2)$, are uniformly bounded in $n$ and $k$, note that this implies that assumption of \eqref{eq:mass_func_unif_bound} is satisfied. Finally, we also argue that for these regimes the convergence is uniform in $a$ on compact intervals in $(0,\infty)$.

\subsection{Analysis of the total trees attached to the backbone for \mtitle{$\alpha>2$}}
\label{sec-size-trees-off-backbone-alpha>2}
In this section we provide an extensive proof of \eqref{aim-Hk} for $\alpha>2$ and identify the limiting constants.
We start by formalising the result for the mass function in Proposition~\ref{prop-size-trees-backbone->2} and provide the proof, where we start by giving a concise proof overview. 

\begin{proposition}[Size of forest off backbone for $\alpha>2$]
\label{prop-size-trees-backbone->2}
Fix $w_k=p_c(1+a/k)$, $\alpha>2$
and $n=\lceil xk^2\rceil$.
Define the function $h_\alpha\colon (0,\infty)^2\to (0,\infty)$ by
\begin{equation}
\label{eq:def_h_alpha>2}
    h_\alpha(x,a) = 
    \frac{p_c\expec[X^\star]x^{-3/2}}{\sqrt{2\pi( 1-p_c+ p_c^2\sigma^2)}}
     \exp\Big\{-\frac{(ax)^2}{2x (1-p_c + p_c^2\sigma^2)}\Big\}.
\end{equation}
Recall $\prob_k$ from \eqref{eq:def_prob_k}. Then
\begin{equation}
    \lim_{k\rightarrow \infty} k^3\prob_k(H_k=\lceil k^2 x\rceil )=h_\alpha(x,a),
\end{equation}
where the convergence is {\em uniform} for $x>\varepsilon$ and $a\in(\vep,1/\vep)$ for any $\vep>0$. Furthermore,
\begin{equation}
\label{eq:toshow_uniformbound_alpha>2}
n^{3/2}\prob(H_k=n\mid w_k=p_c(1+a/k))
\end{equation}
is uniformly bounded in $n\geq1$ and $a\in (\varepsilon,1/\varepsilon)$.
\end{proposition}
By Lemma \ref{App:lem:scaling_mass_distr}, Proposition~\ref{prop-size-trees-backbone->2}
provides the proof of Theorem~\ref{thm:Cs_scaling}, restricted to $\alpha>2$. 
In the proof we make extensive use of the $\o{k}{a,x}$ notation from \eqref{eq:small-o} to indicate a smaller order term in $k\to\infty$ uniformly in $a\in[\varepsilon,1/\varepsilon]$ and $x>\varepsilon$ and the notation $\prob_k(\cdot)$ and $\expec_k[\cdot]$ as defined in \eqref{eq:def_prob_k} for $w_k=p_c(1+a/k)$.
 We continue with the proof of Proposition \ref{prop-size-trees-backbone->2}:

\begin{proof}

The proof consists of three major parts which we outline now:
\begin{enumerate}
    \item  \hyperref[par:5.2-part1]{\textbf{ No significant effect of $\eta(w_k)^{-\hat D_{v_k}}$.}}
    We start with \eqref{tilde-Hk-law-binomial-b}.
    In Lemma~\ref{lem-Goal_no_eta}, we first show that the factor $\eta(w_k)^{-\hat D_{v_k}}$ in \eqref{tilde-Hk-law-binomial-b} only contributes an error term that vanishes, even when scaled with a factor $k^{3}$. 
    Its methods are also used to show the uniform bound of \eqref{eq:toshow_uniformbound_alpha>2}.
    \item \hyperref[par:5.2-part2]{\textbf{ Local limit results.}} Once the $\eta(w_k)^{-\hat D_{v_k}}$ term in \eqref{tilde-Hk-law-binomial-b} is shown to be insignificant, we can ignore it and focus on the binomial mass function in \eqref{tilde-Hk-law-binomial-b}. In particular, in Corollary \ref{cor-trees-off-backbone}, we rewrite this mass function in order to apply a local limit theorem on it, as outlined in \eqref{Bin-LCLT-pre}.
    \item \hyperref[par:5.2-part3]{\textbf{ A central limit result.}} Next, in Lemma~\ref{lem:step_3_alpha>2}, we apply an additional limit result and simplify the result to derive $h_\alpha(x,a)$ as presented in
    \eqref{eq:def_h_alpha>2}.
    This convergence is based on the central limit theorem, as the variables have finite variance in this regime.
\end{enumerate}

\paragraphi{Part 1 - No significant effect of $\eta(w_k)^{\hat D_{v_k}}$}
\label{par:5.2-part1}
The starting point of the analysis is the expression in \eqref{tilde-Hk-law-binomial-b}. In this part, we show that $\eta(w_k)^{-\hat D_{v_k}}$ only adds an $o(k^3)$ error term.
Moreover, this analysis also provides the uniform bound of $n^{3/2}\prob_k(H_k=n)$. For both results we make use of a
general bound on the mass function of the sum of i.i.d. $(\tilde X_{i}^{\sss(k)})_{k\geq1}$ random variables.
\begin{lemma}[Uniform bound on the mass function for $\alpha>2$]
\label{lem:local_lims_upper_bound_mass}
Let $(\tilde{X}^{\sss(k)}_i)_{i\geq 1}$ be an i.i.d. sequence of random variables as given in 
\eqref{eq:def_tilde_X}, where $\tilde{X}^{\sss(k)}$  has finite variance. Then there exists some $C$ such that
\begin{equation}
   \sup_m \prob(\tilde{X}^{\sss(k)}_1+\cdots+\tilde{X}^{\sss(k)}_n=m)
    \leq \frac{c}{n^{1/2}}.
\end{equation}
\end{lemma}

\begin{proof} 
Without loss of generality, we show the result in the case $\expec[\tilde X_1^{\sss(k)}]=1$.
Consider $S_n$ a sum of i.i.d. counting random variables with finite variance and mean 1. Then \cite[Lemma 2.1]{Janson2006} shows for some finite $c,C>0$ 
\begin{equation}
    \prob(S_N= N-k) \leq CN^{-1/2}\e^{-ck^2/N}.
\end{equation}
Taking the supremum over $k$ gives the result for $\alpha>2$. 
\end{proof}

From this result the uniform bound on $n^{3/2}\prob_k(H_k=n)$ can be proven.
This follows almost directly from \eqref{tilde-Tik-def} and the result of Lemma \ref{lem:local_lims_upper_bound_mass}.
Indeed, by taking an expectation over $\hat D_{v_k}$ we find
\begin{equation}
    \prob(H_k=n) = 
    \expec_k\Big[ \frac{ \hat D_{v_k}}{n}
    \prob_k(\tilde X^{(k)}_1+\cdots+\tilde X^{(k)}_n=n-\hat D_{v_k}\mid \hat D_{v_k})\Big]
    \leq 
    n^{-3/2} C\expec[\hat D_{v_k}].
\end{equation}
We find $\expec[D_{v_k}]<C$ for all $k\geq 0$, by Lemma \ref{lem:DegreeBB_expec_alpha>2} and uniformly in $a\in[\varepsilon,1/\varepsilon]$. This concludes the uniform bound.
We then show the following result:

\begin{lemma}[No significant effect of $\eta(w_k)^{\hat D_{v_k}}$]
\label{lem-Goal_no_eta}
    Let $H_k$ be given as in \eqref{eq:def-H_k}. Recall $\prob_k$ from \eqref{eq:def_prob_k}, $n = \lceil xk^{2}\rceil$ and $X^{\sss(k)}\overset{d}{=} \Binom(X,w_k)$.
    Then, for $\alpha>2$,
    \begin{equation}
    \label{eq:Goal_no_eta}
        \prob_k(H_k=n) = \expec_k\bigg[\frac{\hat{D}_{v_k}}{n}
    \prob_k\Big(X_1^{\sss(k)}+\cdots +X_n^{\sss(k)}=n-\hat{D}_{v_k}\mid  \hat{D}_{v_k}\Big)\bigg](1+\o{k}{a}) +\o[k^{-3}]{k}{a},
    \end{equation}
\end{lemma}

\begin{proof}
We consider the mass function as derived in \eqref{tilde-Hk-law-binomial-b}
and
aim to prove that the contribution in the expectation where $\hat D_{v_k}$ is large, is negligible and if $\hat D_{v_k} $ is relatively small, the extra conditioning on extinction has no significant effect. Let $\varepsilon$ be given, then we split
\begin{equation} 
\label{eq:Goal_no_eta_proof1}
    \prob_k(H_k=n) = \prob_k(H_k=n, \hat{D}_{v_k}\leq \varepsilon k)+\prob_k(H_k=n, \hat{D}_{v_k}> \varepsilon k).
\end{equation}
We show that the first term is dominating, whereas the second term contributes only $\o[k^{-3}]{k}{a}$. 

Starting with the first term in \eqref{eq:Goal_no_eta_proof1}, we condition on $\hat{D}_{v_k}$ so that we can use the generalised hitting time theorem  \cite[Theorem 3.14]{VanderHofstad2016} to write 
\begin{equation}
\label{eq:Goal_no_eta_proof2}
\begin{aligned}
    \prob_k&(H_k=n, \hat D_{v_k}\leq \varepsilon k)\\&=
     \expec_k\bigg[\frac{\hat{D}_{v_k}}{n\eta(W_k)^{\hat{D}_{v_k}}}
    \prob_k\Big(X_1^{\sss(k)}+\cdots +X_n^{\sss(k)}=n-\hat{D}_{v_k}\mid  \hat{D}_{v_k}\Big)\1\{\hat D_{v_k}\leq \varepsilon k\}\bigg]\\&\leq 
     \expec_k\bigg[ \e^{-\varepsilon k \log(1-\theta(w_k))} \frac{\hat{D}_{v_k}}{n}
    \prob_k\Big(X_1^{\sss(k)}+\cdots +X_n^{\sss(k)}=n-\hat{D}_{v_k}\mid \hat{D}_{v_k}\Big)\bigg]\\&
    = \e^{-\varepsilon C_\theta a}\expec_k\bigg[ \frac{\hat{D}_{v_k}}{n}
    \prob_k\Big(X_1^{\sss(k)}+\cdots +X_n^{\sss(k)}=n-\hat{D}_{v_k}\mid  \hat{D}_{v_k}\Big)\bigg],
\end{aligned}
\end{equation}
as $\theta(w_k)= C_\theta p_c a/k(1+\o{k}{a})$, taking a limit for $\varepsilon\to 0$ provides an upper bound uniformly in $a\in(\varepsilon,1/\varepsilon)$ and $x>\varepsilon$.
Since $\eta(w_k)^{-\hat{D}_{v_k}}>1$,
the lower bound on \eqref{eq:Goal_no_eta_proof2} is of the same form.

Finally, we show that the second term on the right-hand side of \eqref{eq:Goal_no_eta_proof1} vanishes in the limit. By conditioning on $\hat{D}_{v_k}$,
\begin{equation}
\label{eq:eta_condition_>k_eq1}
    \begin{aligned}
    \prob_k(H_k=n, \ \hat{D}_{v_k} > \varepsilon k) &= 
    \sum_{i\geq \varepsilon k}
    \prob_k(H_k=n\mid \hat D_{v_k}=i)\prob_k(\hat D_{v_k}=i).
    \end{aligned}
\end{equation}
Using the generalised hitting time theorem as in \eqref{tilde-Tik-def}, we obtain that
\begin{equation}
\label{eq:eta_condition_>k_eq2}
\begin{aligned}
     \prob_k(H_k=n,& \ \hat{D}_{v_k} > \varepsilon k)=\sum_{i\geq \varepsilon k} \prob_k( \tilde{T}^{\sss(k)}_1 + \cdots +  \tilde{T}^{\sss(k)}_i = n)
    \prob_k(\hat D_{v_k}=i)
   \\ &=\sum_{i\geq \varepsilon k} \frac{i}{n} 
    \prob(   \tilde{X}_1^{\sss(k)}+\cdots+\tilde{X}_n^{\sss(k)}=n-i)
     \prob_k(\hat D_{v_k}=i)\\&
    =\expec_k\bigg[ \frac{\hat D_{v_k}}{n} 
    \prob\Big(\tilde{X}_1^{\sss(k)}+\cdots+\tilde{X}_n^{\sss(k)}=n-\hat D_{v_k}\Big) \1\{\hat D_{v_k}>\varepsilon k\}
    \bigg] .
\end{aligned}
\end{equation}
We upper bound the mass function in  \eqref{eq:eta_condition_>k_eq2} by the result of Lemma~\ref{lem:local_lims_upper_bound_mass}. We note that, conditionally on $W_k=w_k$, the random variables $(\tilde{X}_i^{\sss(k)})_{i\geq 1}$ are i.i.d. with finite moments and thus its conditions are met. We can therefore bound \eqref{eq:eta_condition_>k_eq2} by
\begin{equation}
\label{eq:Goal_no_eta_proof3}
\begin{aligned}
   &\frac{1}{n} \expec_k\bigg[\hat D_{v_k}\1\{\hat D_{v_k}>\varepsilon k\}
    \ \sup_m\Big\{ \prob(\tilde{X}_1^{\sss(k)}+\cdots+\tilde{X}_n^{\sss(k)}=m)\Big\} \bigg]
   \\&\quad \leq \frac{c}{n^{3/2}} \expec_k[\hat D_{v_k}\1\{\hat D_{v_k}>\varepsilon k\}].
\end{aligned} 
\end{equation}
 We recall that, by Lemma~\ref{lem:UI_DBK_inK}, 
\begin{equation}
\label{eq:Goal_no_eta_proof4}
    \expec[ D_{v_k} \1\{D_{v_k} > \varepsilon k\}]
   =C(a)k^{-(\alpha-2)}(1+\o{k}{a}),
\end{equation}
where $C(a)$ is uniformly bounded for $a\in[\varepsilon,1/\varepsilon]$.
This also extends to $\hat D_{v_k}$, as $w_k\in(p_c,1)$ and $p_c>0$.
We conclude the proof of Lemma~\ref{lem-Goal_no_eta} by the identity in \eqref{eq:Goal_no_eta_proof1}, in combination with \eqref{eq:Goal_no_eta_proof2} and
\eqref{eq:Goal_no_eta_proof4}, so that there exists some uniformly bounded $C^*(a)$ for $a\in[\varepsilon,1/\varepsilon]$ such that 
\begin{equation}
\begin{aligned}
    \prob_k(H_k=n)& =
    \e^{-\varepsilon C_\theta}\expec\bigg[ \frac{\hat{D}_{v_k}}{n}
    \prob_k\Big(X_1^{\sss(k)}+\cdots +X_n^{\sss(k)}=n-\hat{D}_{v_k}\mid  \hat{D}_{v_k}\Big)\bigg](1+\o{k}{a,x})
    \\ &\quad+
    C^*(a)  n^{-3/2} k^{-(\alpha-2)}(1+\o{k}{a}).
    \end{aligned}
\end{equation}
Taking the limit for $\varepsilon \to 0$ and recalling that $n = \lceil xk^{2}\rceil$, we find 
\begin{equation}
\begin{aligned}
     \prob_k(H_k=n) &=
\expec\bigg[ \frac{\hat{D}_{v_k}}{n}
    \prob\Big(X_1^{\sss(k)}+\cdots +X_n^{\sss(k)}=n-\hat{D}_{v_k}\mid W_k=w_k, \hat{D}_{v_k}\Big)\bigg](1+\o{k}{a,x})\\& \qquad
     +C^*(a)  n^{-3/2} k^{-(\alpha-2)}(1+\o{k}{a}).
\end{aligned}    
\end{equation}
Note that the error term satisfies $C^*(a) x^{-3/2}k^{-\alpha-1}(1+\o{k}{a})  = \o[k^{-3}]{k}{a,x}$, as $\alpha>2$. 
\end{proof}

\paragraphi{Part 2 - Local limit results}
\label{par:5.2-part2}
By Part 1, we can complete the analysis while ignoring the $\eta(w_k)^{-\hat D_{v_k}}$ term.
With $H_k$ denoting the size of the  forest attached to the $k$-th backbone vertex, Lemma~\ref{lem-Goal_no_eta} implies that
\begin{equation}
\begin{aligned}
    \label{size-trees-off-backbone-pre}
    \prob_k(H_k=n)
   & =\expec\bigg[\frac{\hat{D}_{v_k}}{n}
    \prob\Big(X_1^{\sss(k)}+\cdots +X_n^{\sss(k)}=n-\hat{D}_{v_k}\mid W_k=w_k, \hat{D}_{v_k}\Big)\bigg]\\&
    \quad \times (1+\o{k}{a,x})+\o[k^{-3}]{k}{a,x}.
\end{aligned}
\end{equation}
In this part, we 
first rewrite this expression in a more convenient form:

\begin{lemma}[Law of trees hanging off backbone]
\label{lem-trees-off-backbone}
For $w_k>p_c$ and $n\geq 1$, let $(X_i)_{i\geq 1}$ be i.i.d. copies of $X$, and
\begin{equation}
    \label{Qn-def}
    Q_n=\sum_{i=1}^n X_i.
\end{equation}
Then,
\begin{equation}
    \begin{aligned}
    \label{size-trees-off-backbone}
    \prob_k(H_k=n)
    =\frac{w_k \expec_k[D_{v_k}]}{n}&\expec_k\Big[
    \prob_k\Big(\Binom\big(Q_n+D_{v_k}^\star-1,w_k\big)=n-1\Big)\Big]
    \\&\times(1+\o{k}{a,x})
    +\o[k^{-3}]{k}{a,x},
    \end{aligned}
\end{equation}
\end{lemma}
\begin{proof}
 We write $S_m\overset{d}{=}\Binom(m,p)$, and recall that $S_m=\sum_{i=1}^m I_i,$ where $(I_i)_{i=1}^m$ are i.i.d.\ Bernoulli random variables with success probability $p$. Then, for an arbitrary function $f$,
    \begin{equation}
        \begin{aligned}
            \label{Binomial-computation}
    \expec[S_m f(S_m)]
    &=\sum_{i=1}^m \expec[I_if(S_m)]=m \expec[I_mf(S_m)]\\&=mp \expec[f(S_m)\mid I_m=1]=mp \expec[f(S_{m-1}+1)].
        \end{aligned}
    \end{equation}
Applying this in \eqref{size-trees-off-backbone-pre} to $\hat{D}_{v_k}\overset{d}{=} \Binom(D_{v_k},w_k)$ gives
\begin{equation}
\begin{aligned}
    \label{size-trees-off-backbone-b}
    &\prob_k(H_k=n)
    \\&=\expec_k\bigg[\frac{D_{v_k}w_k}{n}
    \prob_k\Big(X_1^{\sss(k)}+\cdots +X_n^{\sss(k)}=n-\Binom(D_{v_k}-1,w_k)-1\Big) \bigg]\\&\quad\times (1+\o{k}{a,x})
    +\o[k^{-3}]{k}{a,x}\\
    &\quad=\expec_k\bigg[\frac{D_{v_k}w_k}{n}
    \prob_k\Big(\Binom(Q_n+D_{v_k}-1,w_k)=n-1\Big)\bigg](1+\o{k}{a,x})\\&\quad+\o[k^{-3}]{k}{a,x}.
    \end{aligned}
\end{equation}
Further, for any non-negative random variable $Z$ with $\expec[Z]>0,$
\begin{equation}
    \label{eq:size_biased_result}
    \expec[Zf(Z)]=\expec[Z] \expec[f(Z^\star)].
\end{equation}
Applying this to $D_{v_k}$, conditionally on $W_k=w_k$, leads to the claim.
\end{proof}
In the remainder of this section, we investigate the right-hand side of \eqref{size-trees-off-backbone} in more detail. We rely crucially on local limit theorems (LCLTs), as outlined in detail in Section~\ref{sec-size-trees-off-backbone-overview}. This implies the following corollary:
\begin{corollary}[Law of trees hanging off backbone]
\label{cor-trees-off-backbone}
For $w_k>p_c$ and $n\geq 1$,
\begin{equation}
\begin{aligned}
    \label{size-trees-off-backbone-simplified}
    &\prob_k(H_k=n)
    \\&=\frac{w_k \expec_k[D_{v_k}](1+\o{k}{a,x})}{\sqrt{2\pi n^3 (1-w_k)}}\expec_k\Big[{\mathrm{e}}^{-(w_k(Q_n+D_{v_k}^\star-1)-n)^2/[2n(1-w_k)]}\Big]
    +\o[k^{-3}]{k}{a,x}.
    \end{aligned}
\end{equation}
\end{corollary}
\paragraphi{Part 3 - A central limit result}
\label{par:5.2-part3}
Finally, we focus on the expectation in   \eqref{size-trees-off-backbone-simplified} and simplify this result based on the central limit theorem.

\begin{lemma}[Expression for $h_\alpha(x,a)$ for $\alpha>2$]
\label{lem:step_3_alpha>2} 
Fix $w_k=p_c(1+a/k)$ and $n\geq 1$. Then
\begin{equation}
\begin{aligned}
     \prob_k(H_k=n)&=
     \frac{1}{k^3}
     \frac{p_c \expec[X^\star]}{\sqrt{2 \pi (1-p_c)x^3}}\expec\big[{\mathrm{e}}^{-(\sigma p_c Z-a\sqrt{x})^2/[2(1-p_c)]}\big]
     \\&\quad \times (1+\o{k}{a,x}) +\o[k^{-3}]{k}{a,x}.
\end{aligned}
\end{equation}
\end{lemma}
\begin{proof}
Our starting point is \eqref{size-trees-off-backbone-simplified}.
We define
\begin{equation}
    f(z)={\mathrm{e}}^{-z/2},
\end{equation}
so that
\begin{equation}
    \begin{aligned}
     \label{Bin-LCLT-cons-a2}
    &\prob_k(H_k=n)
    \\&\quad=\frac{w_k \expec_k[D_{v_k}]}{\sqrt{2 \pi (1-p_c)n^3}} \expec_k\bigg[f\Big(\frac{(n-(Q_n+D_{v_k}^\star)w_k)^2}{n(1-w_k)}\Big)\bigg](1+\o{k}{a,x}) +\o[k^{-3}]{k}{a,x}.
    \end{aligned}
\end{equation}

We conclude from Lemma~\ref{lem:limiting_distr_Dbk_alpha>2} that $D_{v_k}^\star\xrightarrow{d}X^{\star\star}$ is tight, so that $D_{v_k}^\star/\sqrt{n}\convp 0$. Thus, by conditional independence of $Q_n$ and $D_{v_k}^\star$ on $W_k$ and by the central limit theorem, 
\begin{equation}
    \Big( 
    \frac{Q_n-n/p_c}{\sqrt{n}}
    ,
    \frac{D_{v_k}^\star}{\sqrt{n}}
    \Big)\xrightarrow{d}
    (\sigma Z,0) . 
\end{equation}
Then 
\begin{equation}
\label{eq:54_uniformly_in_a}
\begin{aligned}
\frac{(n-(Q_n+D_{v_k}^\star)w_k)^2}{n(1-w_k)}
&=\frac{  \Big(\frac{n-p_cQ_n}{\sqrt{n}}
-\frac{p_caQ_n\sqrt{x}}{n}+\frac{w_k D^\star_{v_k}}{\sqrt{n}}
\Big)^2
}{1-w_k}
\xrightarrow{d} \frac{(\sigma p_c Z+a\sqrt{x})^2}{1-p_c},
\end{aligned}
\end{equation}
where $Z$ a standard normally distributed random variable and $\sigma^2=\var[X]$. We note that this convergence is uniform in $a$, for which we refer to \aprefb{4}{app:sec_uniform_conv} for more details.
Thus, by the Portmanteau theorem,
\begin{equation}
\begin{aligned}
    \label{Bin-LCLT-cons-a3}
    \prob_k(H_k=n)
    &=\frac{p_c \expec_k[D_{v_k}]}{\sqrt{2 \pi (1-p_c)n^3}}\expec\big[{\mathrm{e}}^{-(\sigma p_c Z-a\sqrt{x})^2/[2(1-p_c)]}\big]\\&\quad\times(1+\o{k}{a,x})+\o[k^{-3}]{k}{a,x}.
\end{aligned}    
\end{equation}
Whilst intuitive, this is not directly obvious, and we refer to \aprefb{4}{app:sec_uniform_conv} for a detailed argument. 
Recalling $n=\lceil x k^2\rceil,$ this leads to
\begin{equation}
    \label{Bin-LCLT-cons-b1}
    \prob_k(H_k=\lceil x k^2\rceil)
    = \frac{h(x,a)+\o{k}{a,x}}{k^3},
\end{equation}
where
\begin{equation}
    \label{eq:h(x,a)_for_alpha>2}
    h(x,a)=\frac{p_c \expec[X^\star]}{\sqrt{2 \pi (1-p_c)x^3}}\expec\big[{\mathrm{e}}^{-(\sigma p_c Z-a\sqrt{x})^2/[2(1-p_c)]}\big],
\end{equation}
 as we recall that $\expec_k[D_{v_k}] \to \expec[X^\star]$ by Lemma~\ref{lem:DegreeBB_expec_alpha>2}.
\end{proof}
Based on the explicit expression of the Laplace transform of a squared normal random variable, 
it is straightforward to show that for some $\mu$ and $s>0$, one finds 
\begin{equation}
    \expec[\e^{ -u(\mu+s Z)^2}]=
    (1+2us^2)^{-1/2} \e^{-\mu^2\frac{u}{1+2us^2}}.
\end{equation}
Applying this to \eqref{eq:h(x,a)_for_alpha>2} we can simplify it to
\begin{equation}
\begin{aligned}
 \expec\big[& {\mathrm{e}}^{-
 \frac{1}{2x(1-p_c)}(\sigma\sqrt{x}p_c Z-ax)^2}\big]
   =
   \Big(1+\frac{2\sigma^2 p_c^2 x}{2x(1-p_c)}\Big)^{-1/2}
    \exp\bigg\{ -\frac{(ax)^2\frac{1}{2x(1-p_c)}}{1+\frac{2(\sigma\sqrt{x}p_c)^2}{2x(1-p_c)}}\bigg\}
  \\&  =
 \Big(1+\frac{\sigma^2p_c^2}{(1-p_c)}\Big)^{-1/2}
     \exp\bigg\{-\frac{(ax)^2}{2x(1-p_c)+2(\sigma\sqrt{x}p_c)^2}\bigg\},
     \end{aligned}
\end{equation}
so that
\begin{equation}
    h_\alpha(x,a) = 
    \frac{p_c\expec[X^\star]}{x\sqrt{2\pi x( 1-p_c+ p_c^2\sigma^2)}}
     \exp\Big\{-\frac{(ax)^2}{2x (1-p_c + p_c^2\sigma^2)}\Big\}.
\end{equation}
\end{proof}

\subsection{Analysis of the forest attached to the backbone for \mtitle{$\alpha\in(1,2)$}}
\label{sec-size-trees-off-backbone-alpha(1,2)}
In this section we provide an extensive proof of \eqref{aim-Hk} for $\alpha\in(1,2)$.
We start by formalising the result in Proposition~\ref{prop-size-trees-backbone-(1,2)}.
In this section we consider again $w_k=p_c(1+a/k)$, we recall the notation of $\o{k}{x,a}$ from \eqref{eq:small-o} for $a\in[\varepsilon,1/\varepsilon]$ and $x>\varepsilon$ and now define $n=\lceil xk^{\alpha/(\alpha-1)}\rceil$:
\begin{proposition}
\label{prop-size-trees-backbone-(1,2)}
Fix $w_k=p_c(1+a/k)$, $\alpha\in(1,2)$ and $n=\lceil xk^{\alpha/(\alpha-1)}\rceil$.  Let $\psi_\alpha(\cdot)$ be a stable density with exponent $\alpha$ and some scale parameter $c_\alpha$.
Define the function $h:(0,\infty)^2\to(0,\infty)$ by
\begin{equation}
\label{eq:def_h_alpha(1,2)}
    h_\alpha(x,a) =
   \alpha c_x p_c^{\alpha-1}x^{(1-2\alpha)/\alpha}
       \int^{\infty}_0 
      \psi_\alpha\Big(
     -\frac{z}{p_c}-\frac{a\mu}{ x^{1/\alpha-1} }\Big) z^{1-\alpha} \dif z.
\end{equation}
Recall $\prob_k$ from \eqref{eq:def_prob_k}, then 
\begin{equation}
    \lim_{k\to\infty} 
    k^{ (2\alpha-1)/(\alpha-1)}
    \prob_k(H_k= \lceil k^{\alpha/(\alpha-1)}x\rceil)
    =h_\alpha(x,a),
\end{equation}
where the convergence is {\em uniform} for $x>\varepsilon$ and  $a\in[\vep,1/\vep]$ for any $\vep>0$. Furthermore,     
\begin{equation}
\label{eq:toshow_uniformbound_alpha(1,2)}
     n^{1+1/\alpha} k^{- (2-\alpha)/(\alpha-1)}\prob_k(H_k=n)
\end{equation}
is uniformly bounded in $n$ and $a\in[\varepsilon,1/\varepsilon]$.
\end{proposition}
By Lemma \ref{App:lem:scaling_mass_distr}, Proposition~\ref{prop-size-trees-backbone-(1,2)} also provides the proof of Theorem~\ref{thm:Cs_scaling}, for $\alpha\in(1,2)$. We now provide the proof for Proposition \ref{prop-size-trees-backbone-(1,2)}.
\begin{proof}
  For $\alpha\in(1,2)$ we expect that the extra factor $\eta(w_k)^{-\hat D_{v_k}}$ is not negligible and therefore, results deviate from $\alpha>2$. The proof overview for $\alpha\in(1,2)$ is split in different parts which are outlined below.
  \begin{enumerate}
      \item  \hyperref[par:5.3-part1]{{\bf No contribution for large values of $\hat D_{v_k}$.}} The starting point is \eqref{tilde-Tik-def}. Here, we split the expectation in two cases based on $\{\hat D_{v_k}\lessgtr K n^{1/\alpha}\}$, for some large $K$.
      In Lemma~\ref{lem:No_high_contribution_DBK_alpha(1,2)}, we prove that the contribution in the case of $ \{\hat D_{v_k}> K n^{1/\alpha}\}$ vanishes in the limit for $K\to\infty$, even when scaled with a factor $k^{ (2\alpha-1)/(\alpha-1)}$.
      Furthermore, we show here the uniform bound of \eqref{eq:toshow_uniformbound_alpha(1,2)}.
      This part shows similarities with part 1 of the proof for $\alpha>2$.
      \item \hyperref[par:5.3-part2]{{\bf Local limit results.}} Restricted to $ \{\hat D_{v_k}\leq K n^{1/\alpha}\}$, we use the binomial local limit theorem and a normal approximation, catered to this specific case, to simplify the mass function in Lemma~\ref{lem:Norm_approx_fin_mean}. These parts are similar to parts 2 and 3 in the proof for $\alpha>2$. 
      \item \hyperref[par:5.3-part3]{{\bf Finalising the result.}} After this, we are left with an expectation of a function of $\expec[\hat D_{v_k}]$ and $\hat D_{v_k}^\star$. We apply results from Section~\ref{sec:degree_bb_alpha_in(1,2)} to derive the convergence result.
  \end{enumerate}
  
  \paragraphi{Part 1 - No contribution for large values of $\hat D_{v_k}$}
  \label{par:5.3-part1}
  We consider the mass function of $H_k$ in \eqref{tilde-Tik-def}. We fix $K$ large, and define
  \begin{equation}
      I_\alpha^{\sss (\leq)}= 
      \1\{ \hat D_{v_k}\leq K n^{1/\alpha}
    \} 
  \end{equation}
  and $I_\alpha^{\sss (>)}$ in an equivalent way. By taking an expectation over $\hat D_{v_k}$, we split
\begin{equation}
\label{eq:mass_function_h_k_split}
\begin{aligned}
    \prob_k(H_k=n)=&
    \expec_k\bigg[\frac{\hat{D}_{v_k}}{n} 
    \prob_k(\tilde X_1^{\sss(k)}+\cdots+\tilde X_n^{\sss(k)}=n-\hat{D}_{v_k}\mid \hat{D}_{v_k})
   I_\alpha^{\sss (>)}
    \bigg]\\
   &+  \expec_k\bigg[\frac{\hat{D}_{v_k}}{n} 
    \prob_k(\tilde X_1^{\sss(k)}+\cdots+\tilde X_n^{\sss(k)}=n-\hat{D}_{v_k}\mid  \hat{D}_{v_k})
   I_\alpha^{\sss (\leq)}
    \bigg].
\end{aligned}    
\end{equation}
 In the following we show that the expectation where $\hat D_{v_k}>Kn^{1/\alpha}$ is sufficiently small for $n$ and  $K$ large enough. Moreover, the method is also used to show the uniform bound in \eqref{eq:toshow_uniformbound_alpha(1,2)}. For both results we make use of an upper bound on the mass function of the sum of i.i.d. $(\tilde X_{i}^{\sss(k)})_{i\geq 1}$. This is a generalisation of Lemma~\ref{lem:local_lims_upper_bound_mass} to $\alpha\in(1,2)$:
  
 \begin{lemma}[Local limit theorem bound for a power-law mass function with $\alpha\in(1,2)$]
\label{lem:local_lims_upper_bound_mass_alpha(1,2)}
Let $(\tilde{X}^{\sss(k)}_i)_{i\geq 1}$ be an i.i.d. sequence of random variables as given in 
\eqref{eq:def_tilde_X}. Then, for $n$ large enough, 
\begin{equation}
   \sup_m \prob(\tilde{X}^{\sss(k)}_1+\cdots+\tilde{X}^{\sss(k)}_n=m)
    \leq \frac{c}{n^{1/\alpha}}.
\end{equation}
\end{lemma}
  \begin{proof} See
   \aprefb{5}{app:sec_local_lims_upper_bound_mass_alpha(1,2)}.
  \end{proof}
From this we can prove the uniform bound on $n^{1+1/\alpha}k^{(2-\alpha)/(\alpha-1)}\prob_k(H_k=n)$.
By the result of Lemma~\ref{lem:local_lims_upper_bound_mass_alpha(1,2)} we can bound the expectation in \eqref{eq:No_high_contribution_DBK_alpha(1,2)_goal} by 
\begin{equation}
\begin{aligned} \label{eq:No_high_contribution_DBK_alpha(1,2)_goal_3}
    \prob_k(H_k=n)&=
    \expec_k\bigg[\frac{\hat{D}_{v_k}}{n} 
    \prob_k(\tilde X_1^{\sss(k)}+\cdots+\tilde X_n^{\sss(k)}=n-\hat{D}_{v_k}\mid \hat{D}_{v_k})
    \bigg]
    \\& \leq 
    \frac{c }{n^{1+1/\alpha}} \expec_k[ D_{v_k}  ].
\end{aligned}
\end{equation}
We noticing that $\expec[D_{v_k}]/k^{(2-\alpha)/(\alpha-1)}$
converges to constant $C_\alpha$ that is uniformly bounded in $a\in[\varepsilon,1/\varepsilon]$ according to Lemma~\ref{lem:rare_degree_Vertices_infinite_Var}. 
This implies 
\begin{equation}
\label{eq:result_bound_mass_alpha(1,2)_c}
     n^{1+1/\alpha} k^{ -(2-\alpha)/(\alpha-1)}
     \prob_k(H_k=n) < C_\alpha.
\end{equation}
We next formalise the scaling of the first term in \eqref{eq:mass_function_h_k_split}:
 
 \begin{lemma}[No contribution for high values of $\hat D_{v_k}$]
 \label{lem:No_high_contribution_DBK_alpha(1,2)}
 Fix $\prob_k$ as in \eqref{eq:def_prob_k} and recall $n=\lceil x k^{\alpha/(\alpha-1)}\rceil$. Then 
 for $K\to\infty$ 
 \begin{equation}
 \label{eq:No_high_contribution_DBK_alpha(1,2)_goal}
    \expec_k\bigg[\frac{\hat{D}_{v_k}}{n} 
    \prob_k(\tilde X_1^{\sss(k)}+\cdots+\tilde X_n^{\sss(k)}=n-\hat{D}_{v_k}\mid \hat{D}_{v_k})
    I_\alpha^{\sss (>)}
    \bigg]=
    k^{-\frac{2\alpha-1}{\alpha-1}}
    \o{K}{a,x},
 \end{equation}
 \end{lemma}
 \begin{proof}
 Take $K>0$ and by Lemma~\ref{lem:scaling_largedegrees_alphain(1,2)}, where $w_k=p_c(1+a/k)$, $n=\lceil xk^{\alpha/(\alpha-1)}\rceil$ and $\hat C(a)= C_\theta (p_c a)^{1/(\alpha-1)}$,
  \begin{equation}
  \begin{aligned}
  \label{eq:No_high_contribution_DBK_alpha(1,2)_goal_4}
       &\expec_k[ D_{v_k} \1\{D_{v_k} > Kn^{1/\alpha}\} ]
      \\&\quad =
       \frac{\alpha c_X x^{(2-\alpha)/\alpha}
       }{\expec[X]k^{-(2-\alpha)/(\alpha-1)}}
       \int^\infty_{K}
       \e^{-z \frac{x^{1/\alpha}}{p_c\hat C(a)} } z^{1-\alpha} \dif z (1+\o{k}{a}).
       \end{aligned}
  \end{equation}
  Using the result of  Lemma \ref{lem:No_high_contribution_DBK_alpha(1,2)}, and \eqref{eq:No_high_contribution_DBK_alpha(1,2)_goal_4} we find that there exists some $C$ such that 
  \begin{equation}
  \begin{aligned}
  \label{eq:bound_large_Dbk_alpha(1,2)}
   \expec_k\bigg[\frac{\hat{D}_{v_k}}{n} 
    \prob_k(\tilde X_1^{\sss(k)}+\cdots+\tilde X_n^{\sss(k)}&=n-\hat{D}_{v_k}\mid \hat{D}_{v_k})
    I_\alpha^{\sss (>)}
    \bigg]\leq\frac{ c }{n^{1+1/\alpha}} \expec_k[\hat D_{v_k} I_\alpha^{\sss (>)} ]
    \\&\leq 
      \frac{C x^{(1-2\alpha)/\alpha}}{
      k^{ (2\alpha-1)/(\alpha-1)}}
       \int_{K}^\infty \e^{-z\frac{x^{1/\alpha}}{p_c\hat C(a)}} z^{1-\alpha}\dif z (1+\o{k}{a}).
       \end{aligned}
  \end{equation}
  Note that this term is bounded uniformly in $x>\varepsilon$  and in $a\in[\varepsilon,1/\varepsilon]$, as then $\hat C(a) \leq C_\theta p_c (\varepsilon^{-1} p_c)^{1/(\alpha-1)}$,
  which concludes the proof.
 \end{proof}
 We can now focus on the expectation under the condition $\{\hat D_{v_k} \leq Kn^{1/\alpha}\}$, which we simplify in the second part.
 We use the result from \eqref{eq:massfunction_eta} to find 
 \begin{equation}
 \begin{aligned}
 \label{eq:part1_end_alpha(1,2)_clusters}
    & \prob_k(H_k= n)
   \\ &=\expec_k\bigg[\frac{\hat{D}_{v_k}}{n\eta(w_k)^{\hat{D}_{v_k}}} 
    \prob_k(X_1^{\sss(k)}+\cdots+X_n^{\sss(k)}=n-\hat{D}_{v_k})\mid \hat{D}_{v_k})
    I_\alpha^{\sss (\leq)}
    \bigg] +
    k^{-\frac{2\alpha-1}{\alpha-1}}
    \o{K}{a,x}.
    \end{aligned}
 \end{equation}
 \paragraphi{Part 2 - Local limit results}
 \label{par:5.3-part2}
 We continue from \eqref{eq:part1_end_alpha(1,2)_clusters}.
 Define $Q_n$ as in \eqref{Qn-def}, for $\alpha\in(1,2)$. 
 We use the binomial local limit result of \eqref{Bin-LCLT-pre} on the expression in \eqref{eq:part1_end_alpha(1,2)_clusters} to find 
 \begin{equation}
 \label{eq:cluster_scaling_alphain(1,2)_1}
 \begin{aligned}
     \prob_k(H_k= n)=
     \expec_k\Bigg[&
     \frac{ 1+\o{k}{a,x}}{n\sqrt{2\pi (n-\hat D_{v_k}) (1-w_k)} }
     \frac{\hat D_{v_k}}{\eta(w_k)^{\hat D_{v_k}}}
     \\& \hspace{0cm} \times\exp\bigg\{-
     \frac{(n-\hat D_{v_k}-w_kQ_n)^2}{2(n-\hat D_{v_k})(1-w_k)}
     \bigg\}
    I_\alpha^{\sss (\leq)}
     \Bigg]+k^{-\frac{2\alpha-1}{\alpha-1}}
    \o{K}{a,x}.
     \end{aligned}
 \end{equation}
 Note that the binomial local limit result is uniform in its parameters and therefore also uniform in $a$ in this case.
 Next, we apply the local central limit result for stable laws
 \cite[Chapter 9 \textsection 50]{Gnedenko1968}
 for $Q_n$ that states
 \begin{equation}
     \prob(Q_n =\ell) = 
     \frac{1}{n^{1/\alpha}}\bigg(
     \psi_\alpha\Big(
     \frac{\ell-n\mu}{n^{1/\alpha}}
     \Big) +o(1)\bigg),
 \end{equation}
 where $\psi_\alpha(\cdot)$ is the density of a stable random variable with exponent $\alpha$, scale parameter $c_\alpha$ and characteristic function is given by $f(t)=\e^{-|c_\alpha t|^{\alpha}}$. 
Conditioning on $Q_n$ results in 
 \begin{equation}
 \begin{aligned}
  \label{eq:cluster_scaling_alphain(1,2)_2}
\frac{ 1+\o{k}{x,a}}{n\sqrt{2\pi (1-w_k)} }&\expec_k\Bigg[
     \frac{\hat D_{v_k}}{\eta(w_k)^{\hat D_{v_k}}
     \sqrt{n-\hat D_{v_k}}
     }
     \sum_{\ell\geq 1}
      \frac{1}{n^{1/\alpha}}\bigg(
     \psi_\alpha\Big(
     \frac{\ell-n\mu}{n^{1/\alpha}}
     \Big) +o(1)\bigg)\\&\hspace{0cm}\times
     \exp\bigg\{-
     \frac{(n-\hat D_{v_k}-w_k\ell)^2}{2(n-\hat D_{v_k})(1-w_k)}
     \bigg\}
     I_\alpha^{\sss (\leq)}
     \Bigg]+k^{-\frac{2\alpha-1}{\alpha-1}}
    \o{K}{a,x}.
     \end{aligned}
 \end{equation}
 We simplify \eqref{eq:cluster_scaling_alphain(1,2)_2} by using Lemma \ref{lem:Norm_approx_fin_mean} with the  parameters 
  \begin{equation}
      \begin{aligned}
 m = -n\mu + (n-\hat D_{v_k})/w_k;
      \qquad 
      \sigma^2 = \frac{ (n-\hat D_{v_k})(1-w_k)}{w_k^2};\\
     b= \frac{1}{n^{1/\alpha}}
      (-n\mu + (n-\hat D_{v_k})/w_k); \qquad
      z=n\mu;\qquad \chi(\cdot) = \psi_\alpha(\cdot).
      \end{aligned}
  \end{equation}
  Then the conditions of Lemma~\ref{lem:Norm_approx_fin_mean} are met because $z+m=(n-\hat D_{v_k})/w_k\to\infty$ uniformly in $a\in[\varepsilon,1/\varepsilon]$ and $x>\varepsilon$, as $\hat D_{v_k}<Kn^{1/\alpha}$. 
  Moreover, $\sigma/m\to 0$
  uniformly in $x>\varepsilon$ and $a\in[\varepsilon,1/\varepsilon]$.
  Finally, we know that $\psi_\alpha(\cdot)$ is the density of a stable law, which is bounded and continuous.
  This implies that the conditions of Lemma~\ref{lem:Norm_approx_fin_mean} are met and that we can 
 substitute the result of \eqref{eq:Normal_approx_chi} in \eqref{eq:cluster_scaling_alphain(1,2)_2}.
 First, we simplify $b$ and recall that $w_k=p_c(1+a/k)$
 and therefore
\begin{equation}
      n\mu - \frac{n}{p_c(1+a/k)}
      =n\mu(1-\frac{k}{k+a})=
      xk^{\alpha/(\alpha-1)}\mu\frac{a}{k+a} =
      ax\mu k^{1/(\alpha-1)}(1+\o{k}{x,a}),
\end{equation}
where the error is uniform in $a\in[\varepsilon,1/\varepsilon]$
  This implies that \eqref{eq:cluster_scaling_alphain(1,2)_2} can be rewritten as
 \begin{equation}
 \label{eq:cluster_scaling_alphain(1,2)_3}
     \frac{1+\o{k}{x,a}}{n^{1+1/\alpha} w_k}\expec_k\Bigg[
     \frac{\hat D_{v_k}}{\eta(w_k)^{\hat D_{v_k}}}
    \psi_\alpha\Big(
     -\frac{\hat D_{v_k}}{w_kn^{1/\alpha}}
     -\frac{a x\mu k^{1/(\alpha-1)}}{ n^{1/\alpha} }
     \Big)I_\alpha^{\sss (\leq)}
     \Bigg]+k^{-\frac{2\alpha-1}{\alpha-1}}
    \o{K}{a,x}.
 \end{equation}
 Recall that $\hat C(a) =C_\theta( p_c a)^{1/(\alpha-1)}$, where $C_\theta$ is defined in \eqref{eq:def_C_theta}.
 For the $\eta(w_k)^{-\hat D_{v_k}}$ factor, we note that, by  Remark \ref{rem:unif_cnv_theta} and under the assumption $w_k=p_c(1+a/k)$,
 \begin{equation}
 \begin{aligned}
     \eta(w_k)^{\hat D_{v_k}}
     &=(1-\theta(w_k))^{\hat D_{v_k}}
     =\bigg(
     1-\frac{\hat C(a) }{k^{1/(\alpha-1)}}
     \bigg)^{\hat D_{v_k}
     \frac{k^{1/(\alpha-1)}}{k^{1/(\alpha-1)}}}
    \\& =\e^{-\hat C(a)\hat{D}_{v_k} k^{-1/(\alpha-1)}}(1+\o{k}{a,x}).
     \end{aligned}
 \end{equation}
 Substituting this result and recalling that $n=\lceil x k^{\alpha/(\alpha-1)}\rceil$ simplifies 
 \eqref{eq:cluster_scaling_alphain(1,2)_3} to
 \begin{equation}
 \begin{aligned}
 \label{eq:cluster_scaling_alphain(1,2)_4}
     &\frac{1+\o{k}{a,x}}{n w_k}\expec_k\bigg[
     \frac{\hat D_{v_k}}{n^{1/\alpha}}
     \e^{\hat C(a)\hat D_{v_k}  \frac{x^{1/\alpha}}{n^{1/\alpha}} }
    \psi_\alpha\Big(
     -\frac{\hat D_{v_k}}{w_kn^{1/\alpha}}
     -\frac{a x \mu k^{1/(\alpha-1)}}{ n^{1/\alpha} }
     \Big)I_\alpha^{\sss (\leq)}
     \bigg]\\&\quad+k^{-\frac{2\alpha-1}{\alpha-1}}
    \o{K}{a,x}.
     \end{aligned}
 \end{equation}
 Finally, we note that we now can write \eqref{eq:cluster_scaling_alphain(1,2)_4} as
 \begin{equation}
 \begin{aligned}
 \label{eq:cluster_scaling_alphain(1,2)_44} 
      &\prob(H_k = \lceil xk^{\alpha/(\alpha-1)} \rceil )\\&\quad=
     \frac{1+\o{k}{a,x}}{n w_k}\expec_k\bigg[
    \hat D_{v_k}n^{-1/\alpha}
     g\Big(x, \hat D_{v_k}n^{-1/\alpha}\Big)
     \bigg]+k^{-\frac{2\alpha-1}{\alpha-1}}
    \o{K}{a,x},
     \end{aligned}
 \end{equation}
where
\begin{equation}
    g(x,y) = 
    \e^{\hat C(a)x^{1/\alpha} y }
    \psi_\alpha\Big(
     -\frac{y}{w_k}
     -\frac{a x \mu k^{1/(\alpha-1)}}{ n^{1/\alpha} }
     \Big)\1\{ y\leq K \}.
\end{equation}
In the following we simplify the expectation in \eqref{eq:cluster_scaling_alphain(1,2)_44}.

\paragraphi{Part 3 - Finalising the result}
\label{par:5.3-part3}
We use the result on size-biased distributions introduced in \eqref{eq:size_biased_result}, to rewrite 
 \eqref{eq:cluster_scaling_alphain(1,2)_44} to
 \begin{equation}
 \begin{aligned}
      \label{eq:cluster_scaling_alphain(1,2)_41}
      &\prob_k(H_k = \lceil k^{\alpha/(\alpha-1)} x\rceil 
      )\\&\quad=
       \frac{1+\o{k}{a,x})}{n w_k}\expec_k\Big[\hat D_{v_k}n^{-1/\alpha}\Big]
       \expec_k\Big[ g\Big(x, ( \hat D_{v_k})^\star n^{-1/\alpha}\Big) \Big]+k^{-\frac{2\alpha-1}{\alpha-1}}
    \o{K}{a,x}.
       \end{aligned}
 \end{equation}
The scaling behaviour of the first expectation is known  by Lemma~\ref{lem:rare_degree_Vertices_infinite_Var}, which, for $w_k=p_c(1+a/k)$, implies
\begin{equation}
\begin{aligned}
\label{eq:expec_1_alpha(1,2)}
    n^{-1/\alpha}\expec_k&[\hat D_{v_k}] = 
     \\&=
    \frac{p_c\alpha c_{\sss X} \Gamma(2-\alpha)}{(\hat C(a) p_c)^{(2-\alpha)}
    x^{1/\alpha}\expec[X]}
    k^{-1}(1+o(1))=
    C_E k^{-1}x^{-1/\alpha}(1+\o{k}{a,x}),
\end{aligned}
\end{equation}
where 
\begin{equation}
\label{eq:def_C_E}
    C_E =  \frac{\alpha c_{\sss X} \Gamma(2-\alpha)p_c}{(\hat C(a) p_c)^{(2-\alpha)}\expec[X]},
\end{equation}
It remains to show convergence for the second expectation in \eqref{eq:cluster_scaling_alphain(1,2)_41}. However, we know that $g_x(\cdot)$ is almost surely bounded (due to inclusion of $\1\{ y\leq K \} $ in $g_x(y)$) and it is therefore sufficient to show that $
( \hat D_{v_k})^\star/n^{1/\alpha}$ converges in distribution.
First, under the condition that $w_k=p_c(1+a/k)$, by Lemma~\ref{lem:SizeBiased_distr_BB_infvar}, 
\begin{equation}
    \frac{  D_{v_k}^\star }{n^{1/\alpha}} \xrightarrow{d}
    \frac{ \GammaD(2-\alpha,1)}{x^{1/\alpha}p_c\hat C(a)},
\end{equation}
uniformly in $a\in[\varepsilon,1/\varepsilon]$ and $x>\varepsilon$.
Then Lemma~\ref{lem:converence_binomial_sizebiased_alpha(1,2)} implies 
\begin{equation}
     \frac{ (\hat D_{v_k})^\star}{n^{1/\alpha}}
    \xrightarrow{d}
    \frac{ \GammaD(2-\alpha,1)}{x^{1/\alpha}\hat C(a)}.
\end{equation}
Finally, $g(x,y)$ has a single discontinuity in $y=K$. However, as $\GammaD(2-\alpha,1)/(x^{1/\alpha}\hat C(a))= K$ w.p. 0, we can use the Portmanteau theorem, that can be extended to uniform convergence in $a$, see \cite{Bengs2019}, to find
\begin{equation}
\label{eq:custer_scaling_portmanteau_alphain(1,2)}
 \expec[g(x, (\hat D_{v_k})^\star n^{-1/\alpha})]
 \to 
  \expec[g(x,\GammaD(2-\alpha,1)/(x^{1/\alpha}\hat C(a)) )].
\end{equation}
Let $\Gamma\overset{d}{=} \GammaD(2-\alpha,1)$ and $I^{\sss (\Gamma\leq )}_\alpha = \1\{ \Gamma\leq K\hat C(a) x^{1/\alpha}\}$. Then 
we substitute the results of \eqref{eq:expec_1_alpha(1,2)} and  \eqref{eq:custer_scaling_portmanteau_alphain(1,2)} into \eqref{eq:cluster_scaling_alphain(1,2)_41} to find that 
\begin{equation}
\label{eq:cluster_scaling_alphain(1,2)_43}
\begin{aligned}
    &\prob_k(H_k = \lceil k^{\alpha/(\alpha-1)} x\rceil 
      )
      \\& \quad =
       \frac{1+\o{k}{a,x}}{n w_k}
       C_E
    k^{-1}x^{-1/\alpha}
   \expec_k\bigg[
     \e^{ \Gamma  }
     \psi_\alpha\Big(
     -\frac{\Gamma}{w_k \hat C(a) x^{1/\alpha}}
     -\frac{a\mu}{ x^{1/\alpha-1} }
     \Big)I^{\sss (\Gamma\leq )}_\alpha
     \bigg]\\&\qquad+k^{-\frac{2\alpha-1}{\alpha-1}}
    \o{K}{a,x}
     \\&\quad=
      \frac{ C_E (1+\o{k}{a,x})}{ w_k x^{1+1/\alpha} k^{\frac{ 2\alpha-1}{\alpha-1}}}
   \expec_k\bigg[
     \e^{ \Gamma }
     \psi_\alpha\Big(
     -\frac{\Gamma}{w_k \hat C(a) x^{1/\alpha}}
     -\frac{a\mu}{ x^{1/\alpha-1} }
     \Big)I^{\sss (\Gamma\leq )}_\alpha
     \bigg]\\&\qquad+k^{-\frac{2\alpha-1}{\alpha-1}}
    \o{K}{a,x},
\end{aligned}
\end{equation}
We simplify the expectation and find
\begin{equation}
\label{eq:cluster_scaling_alphain(1,2)_47}
\begin{aligned}
&\expec_k\bigg[
     \e^{ \Gamma }
     \psi_\alpha\Big(
     -\frac{\Gamma}{w_k \hat C(a) x^{1/\alpha}}
     -\frac{a\mu}{ x^{1/\alpha-1} }
     \Big)I^{\sss (\Gamma\leq )}_\alpha
     \bigg]
     \\&= 
      \int_0^{K\hat C(a) x^{1/\alpha}}
      \e^y     \psi_\alpha\Big(
     -\frac{y}{w_k \hat C(a) x^{1/\alpha}}
     -\frac{a\mu}{ x^{1/\alpha-1} }
     \Big) \frac{ y^{1-\alpha} }{\Gamma (2-\alpha)} \e^{-y}\dif y 
     \\&=
     \frac{\hat C(a)^{2-\alpha}x^{(2-\alpha)/\alpha }}{\Gamma(2-\alpha)}
     \int^{K }_0 
     z^{1-\alpha} 
      \psi_\alpha\Big(
     -\frac{z}{w_k }-\frac{a\mu}{ x^{1/\alpha-1} }\Big)\dif z,
    \end{aligned}
\end{equation}
where the final equality follows from a substitution of $\hat C(a)x^{1/\alpha}z=y$. Using the result of  \eqref{eq:cluster_scaling_alphain(1,2)_47} in \eqref{eq:cluster_scaling_alphain(1,2)_43} and recalling the definition of $\hat C(a) = C_\theta (p_c a)^{1/(\alpha-1)}$ and $C_E$ from \eqref{eq:def_C_E}, we let $K\to\infty$, which results in 
\begin{equation}
\begin{aligned}
\label{eq:cluster_scaling_alphain(1,2)_45}
    &\prob_k(H_k = \lceil k^{\alpha/(\alpha-1)} x\rceil 
      )\\&=
      \frac{ C_E \hat C(a)^{2-\alpha}x^{(1-2\alpha)/\alpha}(1+o(1))}{\Gamma(2-\alpha)w_k k^{(2\alpha-1)/(\alpha-1)}}
       \int^{\infty}_0 
      \psi_\alpha\Big(
     -\frac{z}{w_k }-\frac{a}{ x^{1/\alpha-1} }\Big) z^{1-\alpha} \dif z\\&\quad+\o[k^{-\frac{2\alpha-1}{\alpha-1}}]{k}{a,x}
     \\&=
     \frac{\alpha c_x p_c^{\alpha-1}x^{(1-2\alpha)/\alpha}(1+o(1))}{k^{(2\alpha-1)/(\alpha-1)}}
            \int^{\infty}_0 
      \psi_\alpha\Big(
     -\frac{z}{p_c}-\frac{a}{ x^{1/\alpha-1} }\Big) z^{1-\alpha} \dif z\\&\quad +\o[k^{-\frac{2\alpha-1}{\alpha-1}}]{k}{x,a}.
\end{aligned}
\end{equation}
\end{proof}
This concludes the proof of the convergence of the mass function.

\subsection{Analysis of the forests attached to the backbone for \mtitle{$\alpha\in(0,1)$}}
\label{sec-size-trees-off-backbone-alpha(0,1)}
We end the section with an analysis on the size of the finite forests attached to the backbone for $\alpha\in(0,1)$.
In the following we condition on $W_k=w_k$ and take $n=\lceil xw_k^{\alpha/(1-\alpha)}\rceil$.
The main result that we prove is the following:
\begin{proposition}[Size of total trees off backbone for $\alpha\in(0,1)$]
\label{prop-size-trees-backbone-(0,1)}
Condition on $W_k=w_k$, where $w_k\searrow 0$, fix $\alpha\in(0,1)$ and  $n=\lceil xw_k^{\alpha/(1-\alpha)}\rceil$. 
Let $\psi_\alpha(\cdot)$ be a stable density with exponent $\alpha$ and some scale parameter $c_\alpha$.
Then, define the function $h_\alpha\colon (0,\infty)\to (0,\infty)$ by
\begin{equation}
  \label{eq:def_h_alpha(0,1)}
    h_\alpha(x)=
     c_{\sss X} x^{1-\alpha-1/\alpha}
     \int_0^1  \psi_\alpha( (1-u)/x^{1/\alpha-1} )  u^{1-\alpha} \dif u.
\end{equation}
Recall $\prob_k$ from \eqref{eq:def_prob_k},
then
\begin{equation}
\label{eq:pw_conv_alpha(0,1)}
    \lim_{k\rightarrow \infty} w_k^{-\alpha/(1-\alpha)}\prob_k(H_k=\lceil w_k^{-\alpha/(1-\alpha)} x\rceil)= h_\alpha(x).
\end{equation}
\end{proposition}
Contrary to the approach used for $\alpha>2$ and $\alpha\in(1,2)$, we next show the result of Theorem \ref{thm:Cs_scaling} for $\alpha\in(0,1)$ directly instead of relying on Lemma \ref{App:lem:scaling_mass_distr}.
Indeed, for $\alpha\in(0,1)$, Proposition \ref{prop-size-trees-backbone-(0,1)} implies convergence of the rescaled mass function to a density, which adds an additional useful property.
It is straightforward to verify that 
\begin{equation}
    \int_0^\infty w_k^{-\alpha/(1-\alpha)}\prob_k(H_k=\lceil w_k^{-\alpha/(1-\alpha)} x\rceil)\dif x =1.
\end{equation}
If the limiting density $h_\alpha(x)$ also integrates to 1, i.e. the limiting random variable has no atoms, then Theorem \ref{thm:Cs_scaling} for $\alpha\in(0,1)$ follows from Scheff\'e's lemma \cite[Chapter 5.10]{Williams1991}.

Firstly, by \eqref{eq:pw_conv_alpha(0,1)}, it suffices to check if there is no mass at $0$ and $\infty$.
There is no mass at infinity as,
by Wald's identity and conditionally on $W_k=w_k$,
\begin{equation}
\begin{aligned}
    \expec_k[w_k^{\alpha/(1-\alpha)} H_k]=
    \expec_k\bigg[ w_k^{\alpha/(1-\alpha)}
    \sum_{i=1}^{\hat D_{v_k}} \tilde T_{i,k}
    \bigg]
    =
     w_k^{\alpha/(1-\alpha)}
     \expec_k[\hat D_{v_k}]
     \expec_k[T_{1,k}].
\end{aligned}
\end{equation} 
Then, by Lemma \ref{lem:conv_expecdegree_(0,1)} and \eqref{eq:Conditional_final_clusters_mean_a<1}, we find
\begin{equation}
      \expec_k[w_k^{\alpha/(1-\alpha)} H_k]  \to \frac{1-\alpha}{C_\theta} \frac{1}{1-\alpha}
    <\infty.
\end{equation}
There is also no mass at $0$ as, by Lemma \ref{lem:convergence_hatDBK_alpha(0,1)},
\begin{equation}
  \prob_k(w_k^{\alpha/(1-\alpha)} H_k < \varepsilon) 
    \leq    \prob_k(w_k^{\alpha/(1-\alpha)} \hat D_{v_k} < \varepsilon)\to \prob(\GammaD(1-\alpha,1)<\varepsilon C_\theta),  
\end{equation}
which can be made arbitrarily small. This implies that $h_\alpha$ integrates to 1 as we claimed.

Finally,
Lemma~\ref{lem:convergence_hatDBK_alpha(0,1)} implies that the effective degree $\hat D_{v_k}$ also scales as $w_k^{-\alpha/(1-\alpha)}$. This implies that the dominating contribution to the volume scaling is induced by the degree of $v_k$ rather than the volume of the attached forest itself. 

\begin{proof} 
The proof setup is similar to that for $\alpha>2$ and $\alpha\in(1,2)$, except that we do not need a uniform bound on the mass function anymore:
\begin{enumerate}
\item \hyperref[par:5.4-part05]{\textbf{No contribution for large values of $\hat D_{v_k}$.}}
We start with \eqref{tilde-Tik-def}.
In the proof we need some control on $\hat D_{v_k}$ and therefore first prove that the contribution restricted to $\{\hat D_{v_k}\in( (1-\varepsilon)n,n)\}$ vanishes. Lemma~\ref{lem:split_mass_alpha(0,1)} shows that this contribution vanishes exponentially fast in $n$.
    \item \hyperref[par:5.4-part1]{\textbf{Local limit results.}}
    We start again at  \eqref{tilde-Hk-law-binomial-b}. Similar to the proof for $\alpha\in(1,2)$, we use the binomial local limit theorem on  $\prob(X_1^{\sss(k)}+\cdots+X_n^{\sss(k)}
    =n-\hat D_{v_k})
    =\prob(Q_n =n-\hat D_{v_k})
    $. After this, we simplify the expression further by a normal approximation in Lemma~\ref{lem:Normal_approx_alpha(0,1)} that is close to that of Lemma~\ref{lem:Norm_approx_fin_mean}, but catered to $\alpha\in(0,1)$.  
    \item \hyperref[par:5.4-part2]{\textbf{Convergence of $ w_k^{\alpha/(1-\alpha)}\hat D_{v_k}$. }} After the normal approximation, the remaining terms are expressed in an expectation of a bounded function of $\hat D_{v_k} w_k^{\alpha/(1-\alpha)}$, for which we show its convergence. This concludes the proof. 
\end{enumerate}
\paragraphi{No contribution for large values of $\hat D_{v_k}$}
\label{par:5.4-part05}
Our starting point is \eqref{tilde-Hk-law-binomial-b},
where we aim to show that the contribution of large values of $\hat D_{v_k}$ add a negligible amount to the volume scaling. Fix $\varepsilon>0$ small and define 
\begin{equation}
\label{eq:Indic_alpha<1}
    I^{\sss (\leq)}_\alpha(\varepsilon) = 
    \1\{\hat D_{v_k} \leq (1-\varepsilon)n\},
\end{equation}
then we split
\begin{equation}
\begin{aligned}
    \label{split_mass_alpha(0,1)}
   \prob_k(H_k=n) &=
    \expec_k\Big[\frac{\hat{D}_{v_k}}{n} 
    \prob_k(\tilde X_1^{\sss(k)}+\cdots+\tilde X_n^{\sss(k)}=n-\hat{D}_{v_k}\mid  \hat{D}_{v_k})
     I^{\sss (\leq)}_\alpha(\varepsilon)
    \Big]\\
    &\quad +\expec_k\Big[\frac{\hat{D}_{v_k}}{n} 
    \prob_k(\tilde X_1^{\sss(k)}+\cdots+\tilde X_n^{\sss(k)}=n-\hat{D}_{v_k}\mid  \hat{D}_{v_k})
    I^{\sss (>)}_\alpha(\varepsilon)
    \Big].
    \end{aligned}
\end{equation}
We show that the second expectation vanishes even when scaled with a factor $n$, by showing convergence is exponential in $n$:
\begin{lemma}
 \label{lem:split_mass_alpha(0,1)}
 Recall $\prob_k$ from \eqref{eq:def_prob_k} and $n=\lceil xw_k^{\alpha/(1-\alpha)}\rceil$,
 then for $\varepsilon\leq 1/2$
 \begin{equation}
     n \expec_k\Big[\frac{\hat{D}_{v_k}}{n} 
    \prob_k(\tilde X_1^{\sss(k)}+\cdots+\tilde X_n^{\sss(k)}=n-\hat{D}_{v_k}\mid  \hat{D}_{v_k})
    I^{\sss (>)}_\alpha(\varepsilon)
    \Big]= \o{k}{}.
 \end{equation}
\end{lemma}
\begin{proof}
    If $\hat D_{v_k}>n$, then $\prob(H_k=n)=0$ and therefore we may restrict to $\hat D_{v_k}\leq n$. This implies
    \begin{equation}
        \begin{aligned}
            n &\expec_k\Big[\frac{\hat{D}_{v_k}}{n} 
    \prob_k(\tilde X_1^{\sss(k)}+\cdots+\tilde X_n^{\sss(k)}=n-\hat{D}_{v_k}\mid  \hat{D}_{v_k})
    I^{\sss (>)}_\alpha(\varepsilon)
    \Big]\\&
    \leq 
    n 
    \prob_k(\tilde X_1^{\sss(k)}+\cdots+\tilde X_n^{\sss(k)}\leq \varepsilon n)
    .
        \end{aligned}
    \end{equation}
   We apply the transformation $x\mapsto \e^{-x}$ on both sides. Then, by the Markov inequality and the fact that $\tilde X_1^{\sss(k)}+\cdots+\tilde X_n^{\sss(k)}$ are i.i.d., we find
    \begin{equation}
          n 
    \prob_k(\tilde X_1^{\sss(k)}+\cdots+\tilde X_n^{\sss(k)}\leq \varepsilon n)
    \leq 
    n\expec[\e^{- (\tilde X_1^{\sss(k)}+\cdots+\tilde X_n^{\sss(k)})}]\e^{\varepsilon n}
    =
     n\expec[\e^{- \tilde X_n^{\sss(k)}}]^n\e^{\varepsilon n}
    .
    \end{equation}
   We find that $p_0:=\prob(\tilde X^{\sss(k)}=0)<1$ and thus there exists some $p_1\in(0,1-p_0)$ such that 
    $ \expec[\e^{-\tilde X_n^{\sss(k)}}]\leq
        p_0 + p_1\e^{-1}$. Therefore
    \begin{equation}
        n\expec[\e^{- \tilde X_n^{\sss(k)}}]^n\e^{\varepsilon n}\leq 
        n (p_0 + p_1\e^{-1})^n 
        \e^{\varepsilon n}
        =n(p_0  \e^{\varepsilon} + p_1\e^{-1+\varepsilon})^n. 
    \end{equation}
    It is straightforward to show that for $\varepsilon$ small enough, $p_0  \e^{\varepsilon} + p_1\e^{-1+\varepsilon}<1$, which concludes the result.
\end{proof}

\paragraphi{ Local limit results}
\label{par:5.4-part1}
We continue with \eqref{split_mass_alpha(0,1)} where we bound the second term as $o(n^{-1})=\o[w_k^{\alpha/(1-\alpha)}]{k}{}$. On the first term we apply the result from \eqref{eq:massfunction_eta}.
 We also apply the binomial local limit theorem  from \eqref{Bin-LCLT} and find
\begin{equation}
    \begin{aligned}
    \label{Bin-LCLT-cons-d-(0,1)}
    &\prob_k(H_k=n)\\&=
    \expec_k\Big[\frac{\hat{D}_{v_k}}{n\eta(w_k)^{\hat{D}_{v_k}}} 
    \prob_k(X_1^{\sss(k)}+\cdots+X_n^{\sss(k)}=n-\hat{D}_{v_k}\mid  \hat{D}_{v_k})
    I^{\sss (\leq)}_\alpha(\varepsilon) \Big] \\&\qquad+\o[w_k^{\alpha/(1-\alpha)}]{k}{}
    \\
    &=\frac{(1+\o{k}{}}{n\sqrt{2 \pi  (1-w_k)}}\expec_k\Big[\frac{\hat{D}_{v_k}}{\eta(w_k)^{\hat{D}_{v_k}} \sqrt{n-\hat D_{v_k}}}\\&\qquad\times{\mathrm{e}}^{-(n-\hat{D}_{v_k}-w_kQ_n)^2/[2 (n-\hat{D}_{v_k})(1-w_k)]} I^{\sss (\leq)}_\alpha(\varepsilon)\Big]+\o[w_k^{\alpha/(1-\alpha)}]{k}{}\\
    &=\frac{(1+\o{k}{})}{n\sqrt{2 \pi }}\expec_k\Big[\frac{\hat{D}_{v_k}}{\eta(w_k)^{\hat{D}_{v_k}} \sqrt{n-\hat D_{v_k}}}\\&\qquad \times\mathrm{e}^{-(n-\hat{D}_{v_k}-w_kQ_n)^2/[2 (n-\hat{D}_{v_k})]} I^{\sss (\leq)}_\alpha(\varepsilon) \Big] +\o[w_k^{\alpha/(1-\alpha)}]{k}{},
    \end{aligned}
\end{equation}
 where we have used that $w_k\ll 1$ to replace $1-w_k$ by 1 to simplify the formula. 
We then use the local limit theorem for stable laws \cite[Chapter 9 \textsection 50]{Gnedenko1968} to write
\begin{equation}
    \prob(Q_n=x)
    =\frac{1}{n^{1/\alpha}} (\psi_\alpha(x/n^{1/\alpha})+o(1)),
\end{equation}
where $\psi_\alpha(\cdot)$ is the density of a stable law with exponent $\alpha$ and scale parameter $c_\alpha$. 
Substitution leads to
\begin{equation}
    \begin{aligned}
    \label{Bin-LCLT-cons-d-(0,1)-b}
    \prob_k(H_k=n)
    =\frac{1+\o{k}{}}{n\sqrt{2 \pi }}\expec_k&\Big[\frac{\hat{D}_{v_k}}{\eta(w_k)^{\hat{D}_{v_k}}\sqrt{n-\hat D_{v_k}}}
    \sum_{u\geq 1} \frac{1}{n^{1/\alpha}} (\psi_\alpha(u/n^{1/\alpha})+o(1))\\&\times {\mathrm{e}}^{-(n-\hat{D}_{v_k}-w_ku)^2/[2 (n-\hat{D}_{v_k})]}  I^{\sss (\leq)}_\alpha(\varepsilon)\Big] +\o[w_k^{\alpha/(1-\alpha)}]{k}{}.
    \end{aligned}
\end{equation}
We simplify the expectation based on Lemma~\ref{lem:Norm_approx_fin_mean}, 
we apply this with $z=0$, $m=(n-\hat{D}_{v_k})/w_k$, $\sigma^2=(n-\hat{D}_{v_k})/w_k^2$, and $b=(n-\hat{D}_{v_k})/(w_kn^{1/\alpha})$.
We note that the assumptions are met as, restricted to $\hat D_{v_k} < (1-\varepsilon)n$, we find that $n-\hat D_{v_k}$ diverges in $n$. Therefore
$m \to\infty$ and $\sigma/m = 1/\sqrt{n-\hat D_{v_k}} = \o{k}{}$. 
This leads to
\begin{equation}
    \begin{aligned}
   \label{eq:result_normal_approx_alpha(0,1)}
    &\sum_{a} \psi_\alpha(a/n^{1/\alpha}){\mathrm{e}}^{-(n-\hat{D}_{v_k}-w_ka)^2/[2 (n-\hat{D}_{v_k})]}
    \\&\quad=
    (1+\o{k}{})\psi_\alpha\Big((n-\hat{D}_{v_k})/(w_kn^{1/\alpha})\Big)\sqrt{2\pi (n-\hat{D}_{v_k})/w_k^2}. 
    \end{aligned}
\end{equation}
   
We apply Lemma~\ref{lem:Normal_approx_alpha(0,1)} and its implied result on this specific case, as outlined in \eqref{eq:result_normal_approx_alpha(0,1)} to find
\begin{equation}
    \begin{aligned}
    \label{Bin-LCLT-cons-d-(0,1)-c}
    \prob_k(H_k=n)
    &=\frac{1+\o{k}{}}{ w_kn^{1/\alpha}}\expec_k\Big[\frac{\hat{D}_{v_k}/n}{\eta(w_k)^{\hat{D}_{v_k}}}
    \psi_\alpha\Big((n-\hat{D}_{v_k})/(w_kn^{1/\alpha})\Big)  I^{\sss (\leq)}_\alpha(\varepsilon) \Big]\\&\quad  +\o[w_k^{\alpha/(1-\alpha)}]{k}{}.
    \end{aligned}
\end{equation}
    In the following, we write the expectation of \eqref{Bin-LCLT-cons-d-(0,1)-c} as the expectation of some bounded function in $\hat D_{v_k}/n = \hat D_{v_k} w_k^{\alpha/(1-\alpha)}$.  Convergence would then be implied if $\hat D_{v_k}/n$ converges in distribution. In the following we analyse this in more detail.
    
\paragraphi{ Convergence of $ w_k^{\alpha/(1-\alpha)} \hat D_{v_k}$}
\label{par:5.4-part2}
We substitute $n=\lceil w_k^{-\alpha/(1-\alpha)} x\rceil$ and take $\varepsilon\to 0$.
Then \eqref{Bin-LCLT-cons-d-(0,1)-c} can be written as 
\begin{equation}
\begin{aligned}
    \label{Bin-LCLT-cons-d-(0,1)-d}
    \prob_k(H_k=\lceil &w_k^{-\alpha/(1-\alpha)} x\rceil)=\frac{w_k^{\alpha/(1-\alpha)}(1+\o{k}{})}{ x^{1+1/\alpha}}\\&\times\expec_k\Big[\frac{\hat{D}_{v_k}w_k^{\alpha/(1-\alpha)}}{\eta(w_k)^{\hat{D}_{v_k}}}
    \psi_\alpha\Big((x-\hat{D}_{v_k}w_k^{\alpha/(1-\alpha)})/(x^{1/\alpha})\Big) I^{\sss (\leq)}_\alpha(0)\Big].
 \end{aligned}   
\end{equation}
    Convergence of $w_k^{\alpha/(1-\alpha)} \hat D_{v_k}$, follows according to Lemma~\ref{lem:convergence_hatDBK_alpha(0,1)}. With $\Gamma\overset{d}{=}\GammaD(1-\alpha,1)$, we find, conditionally on $W_k=w_k$
\begin{equation}
    \label{hat-D-conv}
    \hat{D}_{v_k} w_k^{\alpha/(1-\alpha)} \xrightarrow{d} C_\theta^{-1}\Gamma.
\end{equation}
By the continuous mapping theorem, we also find 
\begin{equation}
 \label{eta-conv}
    \begin{aligned}
      \eta(w_k)^{\hat{D}_{v_k}}
    =\exp\Big\{\hat D_{v_k}w_k^{\alpha/(1-\alpha)}\cdot w_k^{-\alpha/(1-\alpha)}\log(1-\theta(w_k))\Big\}
   \xrightarrow{d}
   \e^{-\Gamma },
    \end{aligned}
\end{equation}
Now the expectation in \eqref{Bin-LCLT-cons-d-(0,1)-d} is of a bounded function of $\hat D_{v_k}w_k^{\alpha/(1-\alpha)}$, with a single discontinuity at $\hat D_{v_k} w_k^{\alpha/(1-\alpha)}=x$ induced by $ I^{\sss (\leq)}_\alpha$ as we recall from \eqref{eq:Indic_alpha<1}.
By \eqref{hat-D-conv} and because $C_\theta^{-1}\Gamma=x$ w.p. 0, we can use the Portmanteau theorem so that by \eqref{eta-conv}, we
can write \eqref{Bin-LCLT-cons-d-(0,1)-d} as
\begin{equation}
\begin{aligned}
    \label{Bin-LCLT-cons-d-(0,1)-e}
    &\prob_k(H_k=\lceil w_k^{-\alpha/(1-\alpha)} x\rceil)
   \\&\qquad =\frac{w_k^{\alpha/(1-\alpha)}(1+\o{k}{})}{ x^{1+1/\alpha}}\expec_k\Big[\frac{\Gamma}{C_\theta}  \mathrm{e}^{\Gamma}
    \psi_\alpha\Big((x-\Gamma/C_\theta)/(x^{1/\alpha})\Big)\indic{\Gamma\leq C_\theta x}\Big]\\
&\qquad = \frac{w_k^{\alpha/(1-\alpha)}(1+\o{k}{})}{ x^{1+1/\alpha}}\int_0^{C_\theta x} \psi_{\alpha}((x-y/C_\theta)/x^{1/\alpha})
    \frac{y}{C_\theta}
   \frac{y^{-\alpha}}{\Gamma(1-\alpha)}\dif y.
 \end{aligned}  
\end{equation}
We can simplify the expression by a substitution of $(C_\theta x)^{-1}y= u$, so that
\begin{equation}
     \label{Bin-LCLT-cons-d-(0,1)-z}
     \begin{aligned}
     &\prob_k(H_k=\lceil w_k^{-\alpha/(1-\alpha)} x\rceil)
   \\&\qquad=
     \frac{w_k^{\alpha/(1-\alpha)}  (1+\o{k}{})}{ C_\theta^{\alpha-1}x^{\alpha+1/\alpha-1} }
     \int_0^1  \psi_\alpha( (1-u)/x^{1/\alpha-1} ) \frac{ u^{1-\alpha}}{\Gamma(1-\alpha)} \dif u
     \\&\qquad
      =
     \frac{ c_{\sss X}w_k^{\alpha/(1-\alpha)}  (1+\o{k}{})}{x^{\alpha+1/\alpha-1} }
     \int_0^1  \psi_\alpha( (1-u)/x^{1/\alpha-1} )  u^{1-\alpha} \dif u,
    \end{aligned}
\end{equation}
    where the last equality follows when we recall the definition of $C_\theta$, as given in \eqref{eq:def_C_theta}. This proves the convergence of the mass function.
\end{proof}

\section{Combining the ingredients}
\label{sec:Combine}
In this section, we combine the results derived in the previous sections in order to prove our main results in Theorem~\ref{thm-main}. 
Recall that $M_k$ denotes the size of the $k$-cut IPC, defined in \eqref{eq:def_M_k}. 
The scaling of $(M_k)_{k\geq 0}$ depends sensitively on the precise range of $\alpha$ that we consider. Indeed, for $\alpha>2$ and $\alpha\in(1,2)$, the main contribution to $M_k$ is from the attached forests that have size $\Theta(k^2)$ and $\Theta(k^{\alpha/(\alpha-1)})$, respectively. Even though these contributions are rare (see Propositions \ref{prop-size-trees-backbone->2} and  \ref{prop-size-trees-backbone-(1,2)}, respectively), their mass is so large that they dominate $M_k$. For $\alpha\in(0,1)$, the main contribution arises from the forests furthest away from the root and have size $\Theta( W_k^{-\alpha/(1-\alpha)})$. The proof is split in two parts, where $\alpha>2$ and $\alpha\in(1,2)$ are discussed together in Section~\ref{sec:volume_scaling_alpha>1} and $\alpha\in(0,1)$ in Section~\ref{sec:volume_scaling_alpha<1}.

\subsection{ \mtitle{$k$-}cut IPC scaling for  \mtitle{$\alpha>2$} and \mtitle{$\alpha\in(1,2)$} }
\label{sec:volume_scaling_alpha>1}
For $\alpha>2$  and $\alpha\in(1,2)$, we recall $\halpha=\alpha \wedge 2$ and define 
\begin{equation}
    \label{eq:def_gamma_exponent2}
        \gamma =\begin{cases}
        2 &\text{ for } \alpha>2,\\
        \frac{\alpha}{\alpha-1} &\text{ for }  \alpha \in(1,2).
        \end{cases}
    \end{equation}
In this section we prove Theorem~\ref{thm-main} restricted to $\alpha>2$ and $\alpha\in(1,2)$, i.e., we show that

\begin{equation}
\label{eq:Poisson_measure_goal}
    (k^{-\gamma} M_{\lceil kt\rceil })_{t>0} \xrightarrow{\mathcal{D}} 
    \bigg( \int_0^t 
    \int_0^\infty 
    x\Pi_{\lambda_\alpha}(\dif x,\dif s)\bigg)_{t>0},
\end{equation}
where $\Pi_{\lambda_\alpha}$ the intensity measure of an  inhomogenous Cox process, driven by the random measure $\lambda_\alpha(x,s)$, where

\begin{equation}
\label{eq:def_lambda_PPP}
    \lambda_\alpha( x,s) = 
    \frac{h_\alpha(xs^{-\gamma},sL_{\halpha}(s)/(\halpha-1))}{s^{\gamma+1}} \dif x\dif s,
\end{equation}
and $h_\alpha$ is given in \eqref{eq:def_h_alpha>2} for $\alpha>2$ and in  \eqref{eq:def_h_alpha(1,2)} for $\alpha\in(1,2)$.  Note that $\lambda_\alpha$ is a measure if the integral on the right-hand side of \eqref{eq:Poisson_measure_goal} is finite as shown in Appendix \aprefa{E.1}{app:sec_lambda_finite}.

\begin{proof}[Proof of Theorem~\ref{thm-main} for $\alpha>2$ and $\alpha\in(1,2)$] 
We ignore the contribution of the backbone vertices, as these contribute size $k$, that is negligible compared with $k^\gamma$. 
The proof is then structured based on the following parts:
\begin{enumerate}
    \item \hyperref[par:6.1-part1]{\textbf{No contribution of small forests.}} 
    In Lemma~\ref{lem:volume_no_small_clusters_alpha>1}, we first argue that forests that are of size $H_k \leq \delta k^{\gamma}$ or are within a distance $\varepsilon k$ from the root, do not contribute significantly in the limit. This allows us to restrict the further analysis to the volume of only large values of $H_k$ that are far away from the root. We denote this with $\dot M_k^{(\delta,\varepsilon)}$. 
    \item \hyperref[par:6.1-part2]{\textbf{Convergence of the conditional one-dimensional Laplace transform. }}
    We define the Laplace transform of $k^{-\gamma}\tilde M^{(\delta,\varepsilon)}_{\lceil kt\rceil }$ conditioned on $(W_k)_{k\geq0}$ fixed and show its convergence in Lemma
    \ref{lem:volume_scaling_kcut_a>1_1D}.
    \item \hyperref[par:6.1-part3]{\textbf{Extension to the unconditional convergence in distribution.}} We consider the unconditioned Laplace transform, i.e. as a random variable driven by $(W_k)_{k\geq 0}$ and show in Lemma~\ref{lem:LT_conv_J1_topology} that the convergence result of Theorem~\ref{thm-subthmWk} extends to convergence (in distribution) of the unconditioned Laplace transform. 
    \item \hyperref[par:6.1-part4]{\textbf{Convergence to the Cox process.}} By taking expectations on the result of the previous part, convergence of the unconditional Laplace transform follows. In Corollary \ref{cor:convergence_volume_alpha>1} we show that this also matches the Laplace transform of the Cox process with intensity measure $\Pi_{\lambda_\alpha}$ for some fixed $t$. 
    \item \hyperref[par:6.1-part5]{\textbf{Extension to $J_1$-convergence.}}
    Results thus far work for a fixed $t$ and convergence of the one-dimensional distribution therefore follows. In the last part we strengthen this result by showing process convergence in the $J_1$-topology, by showing convergence of the finite-dimensional distributions and tightness.
    This concludes the proof.
\end{enumerate}
\paragraphi{Part 1 - No contribution of small forests}
\label{par:6.1-part1}
We first show that small forests or forests that are close to the root do not contribute to the volume scaling in the limit. A forest is considered small if its size is smaller than $\delta k^{\gamma}$ for some small $\delta$ and is considered close to root if its respective backbone vertex is closer than $\varepsilon k$ from the root, for some small $\varepsilon$. We formulate this in the following lemma: 

\begin{lemma}[No small forests]
\label{lem:volume_no_small_clusters_alpha>1}
Recall the volume scaling of $M_k$ from \eqref{eq:def_M_k} and define for small $\varepsilon >0$ and $\delta>0$:
\begin{equation}
    \label{eq:def_M_k_smalllarge}
    \dot M^{(\delta,\varepsilon)}_k
    =
    \sum_{\ell=0}^{\lfloor \varepsilon k\rfloor} 
    H_\ell+
    \sum_{\ell=\lceil\varepsilon k\rceil }^k 
    H_\ell\1\{ H_\ell < \delta k^\gamma \} =
    I_1(k,\varepsilon
    )+I_2(k,\delta,\varepsilon
    ),
\end{equation}
then, for every $\eta>0$, $\prob(k^{-\gamma} \dot M^{(\delta,\varepsilon)}_k>\eta)= 
\o{1/\varepsilon,1/\delta}{k}$, where we recall from \eqref{eq:small-o} that this error vanishes for both $\varepsilon\to 0$ and $\delta\to 0$ uniformly in $k>0$.
\end{lemma}
\begin{proof}
We show that both sums vanish for $\varepsilon$ and $\delta$ small. Firstly, $I_1(k,\varepsilon
)$ is small in $\varepsilon$ uniformly in $k$ with high probability by Appendix \aprefa{E.2}{App:tightness_volume>1}. Then for $I_2(k,\delta,\varepsilon)$,
by the Markov inequality it suffices to show that $\expec[  k^{-\gamma} I_2(k,\delta,\varepsilon)]$ converges to $0$ for $\varepsilon,\delta\to 0$ uniformly in $k$. 
By \eqref{eq:toshow_uniformbound_alpha>2} and \eqref{eq:toshow_uniformbound_alpha(1,2)} there exists some finite $C_\alpha>0$ such that, for all $n$ and $k$,
\begin{equation}
   n^{2-1/\gamma}k^{2-\gamma}\prob_k(H_k=n)
   <C_\alpha.
\end{equation}
 This, together with the definition of $I_2(k,\delta,\varepsilon)$ as defined in \eqref{eq:def_M_k_smalllarge}, implies 
\begin{equation}
\begin{aligned}
\label{eq:small_volume_small_clusters1}
 k^{-\gamma} \expec[I_2(k,\delta,\varepsilon)\mid (W_i)_{i=0}^k = (w_i)_{i=0}^k]& \leq k^{-\gamma} 
\sum_{\ell = 0}^k 
\sum_{i=0}^{\lceil \delta k^{\gamma}\rceil}i
\prob_k(H_\ell = i) 
\\&\leq  k^{-\gamma} 
\sum_{\ell = 0}^k 
\sum_{i=0}^{\lceil \delta k^{\gamma}\rceil}
C_\alpha i^{1/\gamma-1}\ell^{\gamma-2}.
\end{aligned}
\end{equation}
Both sums can be bounded by the equivalent integral expression and for some positive $C$ we find
\begin{equation}
\begin{aligned}
\label{eq:small_volume_small_clusters2}
    k^{-\gamma} 
\sum_{\ell = 0}^k 
\sum_{i=0}^{\lceil \delta k^{\gamma}\rceil }
C_\alpha i^{1/\gamma-1}\ell^{\gamma-2}
&= C_\alpha k^{-\gamma} 
\sum_{\ell=0}^k 
\ell^{\gamma-2}
\sum_{i=0}^{\lceil \delta k^{\gamma}\rceil }
i^{1/\gamma-1}\\&\leq
C k^{-\gamma} k^{\gamma-1} (\delta k^\gamma)^{1/\gamma}
=\o{1/\delta}{k}. 
\end{aligned}
\end{equation}
Combining this with the error term introduced by $I_1(k,\varepsilon)$ concludes the proof.
\end{proof}
By Lemma~\ref{lem:volume_no_small_clusters_alpha>1} it is sufficient to consider the volume scaling 
\begin{equation}
\label{eq:def_Mk_tilde}
    \tilde{M}^{(\delta,\varepsilon)}_k
    = \sum_{\ell=\lceil\varepsilon k\rceil}^k 
    H_\ell\1\{ H_\ell > \delta k^\gamma \},
\end{equation}
at the cost of some error $\o{1/\varepsilon,1/\delta}{k}$.
The convergence of this object is considered in the following part.

\paragraphi{Part 2 - Convergence of the conditional one-dimensional Laplace transform}
\label{par:6.1-part2}
Based on the previous part, it suffices to consider the one-dimension convergence of the $(\tilde M^{(\delta,\varepsilon)}_k)_{\lceil kt\rceil})_{k\geq 0}$ process. We first introduce some notation. Firstly, we slightly abuse notation and write for some function $q(x)$ 
\begin{equation}
\label{eq:def_intervalnotation}
    q([a,b]) = (q(x))_{x\in[a,b]},
\end{equation}
Then, we define 
\begin{equation}
\label{eq:def_L_halpha^(k)}
    L_{\halpha}^{\sss(k)}(t)=\frac{k(W_{kt}-p_c)}{p_c}(\halpha-1),
\end{equation}
that can be seen as the stochastic process with limiting stochastic process $ L_{\halpha}(t)$.
We recall that by Theorem \ref{thm-subthmWk}, this convergence was in the $J_1$-topology on $D[\varepsilon,t)$. 
In this part, we condition on $L^{\sss(k)}_{\halpha}([\varepsilon,t])=a^{\sss(k)}([\varepsilon,t])$ and therefore consider the conditional Laplace transform of $k^{-\gamma}\tilde M_{\lceil kt\rceil}$. We show that it converges for fixed $t$ and $k\to\infty$:

\begin{lemma}[Convergence of the conditional Laplace transform.]
\label{lem:volume_scaling_kcut_a>1_1D}
Fix $t>\varepsilon$, consider a c\`adl\`ag process $a^{\sss(k)}([\varepsilon,t])$ and
assume that $ a^{\sss(k)}([\varepsilon,t])\xrightarrow{\mathcal{D}} a^*([\varepsilon,t])$ for $k\to\infty$, where convergence is in the $J_1$-topology and $a^*$ non-increasing. Define 
\begin{equation}
\label{eq;def_bigXi}
    \Xi^{(a^*,\delta)}(s) =\int_{\delta t^\gamma}^\infty \frac{1-\e^{- ux}}{s^{\gamma+1}} h_\alpha\Big(xs^{-\gamma},
    \frac{s a^{*}(s)}{\halpha-1}\Big) \dif x ,
\end{equation}
then 
\begin{equation}
\label{eq:LT_volumescaling_limit1}
     \expec\Big[ \e^{-u k^{-\gamma}\tilde M^{
     (\delta,\varepsilon)}_{\lceil kt\rceil}}
    \mid L_{\halpha}^{\sss(k)}([\varepsilon,t]) = a^{\sss(k)}([\varepsilon,t])\Big]
    \to
    \expec\Big[ \e^{-\int_\varepsilon^t\Xi^{(a^*,\delta)}(s)\dif s}\Big].
\end{equation}

\end{lemma}

\begin{proof}
   For convenience we write
  \begin{equation}
  \label{eq:def_E[kepsilon,kt]}
       \expec_{[k\varepsilon,kt]}[\ \cdot\ ]=
       \expec[\ \cdot \mid  L_{\halpha}^{\sss(k)}([\varepsilon,t]) = a^{\sss(k)}([\varepsilon,t])].
  \end{equation}    
    Fix $s>0$ and $0<\varepsilon<t$ then by the conditioned independence of the attached forests on the backbone, as is implied by the conditioning on $L^{\sss(k)}_{\halpha}([\varepsilon,t])$:
    \begin{equation}
    \label{eq:volume_growth_alpha>2_eq1}
        \expec_{[k\varepsilon,kt]}\Big[
        \e^{- \frac{u}{k^\gamma }\tilde M^{
     (\delta,\varepsilon)}_{\lceil kt\rceil}}
        \Big] 
        = \prod_{\ell= \lceil k\varepsilon\rceil}^{\lceil kt\rceil } \expec_\ell\Big[
        \e^{- \frac{u}{k^\gamma }H_\ell \1\{H_\ell>\delta (kt)^\gamma \}}
        \Big] .
    \end{equation}
Conditioning on the event in the indicator, we simplify the expression in     \eqref{eq:volume_growth_alpha>2_eq1} to
\begin{equation}
    \label{eq:volume_growth_alpha>2_eq2}
\begin{aligned}
    & \expec_{[k\varepsilon,kt]}\Big[
        \e^{- \frac{u}{k^\gamma }\tilde M^{
     (\delta,\varepsilon)}_{\lceil kt\rceil}}
        \Big] \\&\quad = \prod_{\ell=\lceil k\varepsilon\rceil }^{\lceil kt\rceil}
     \bigg(
     \expec_\ell\Big[
        \e^{- \frac{u}{k^\gamma }H_\ell }\mid H_\ell>\delta (kt)^\gamma  
        \Big] 
        \prob_\ell(H_\ell>\delta (kt)^\gamma  ) 
        +  \prob_\ell(H_\ell\leq \delta (kt)^\gamma  ) 
     \bigg) 
     \\&\quad= 
     \prod_{\ell=\lceil k\varepsilon \rceil}^{\lceil kt\rceil}
     \bigg(1-\prob_\ell(H_\ell>\delta (kt)^\gamma  )\Big( 
     1- \expec_\ell\Big[
        \e^{- \frac{u}{k^\gamma }H_\ell }\mid H_\ell>\delta (kt)^\gamma 
        \Big] 
     \Big)
     \bigg).
 \end{aligned}    
\end{equation}
    We rewrite     \eqref{eq:volume_growth_alpha>2_eq2} by taking a logarithm, and putting this in the exponent. As the logarithm can be expanded by a Taylor series as $\log(1-x)=-x(1+o(1))$ and for $\delta>0$ fixed, then \eqref{eq:volume_growth_alpha>2_eq2} equals
    \begin{equation}
    \label{eq:volume_growth_alpha>2_eq3}
    \begin{aligned}
       & \expec_{[k\varepsilon,kt]}\Big[
        \e^{- \frac{u}{k^\gamma }\tilde{M}^{(\delta,\varepsilon)}_{ \lceil kt\rceil}}
        \Big] \\&= \exp\bigg\{- \sum_{\ell=\lceil k\varepsilon\rceil}^{\lceil kt\rceil}  \prob_\ell(H_\ell > \delta (kt)^\gamma )\\&\qquad \times
        \Big(1-\expec_\ell\Big[
        \e^{- \frac{u}{k^\gamma }H_\ell }\mid H_\ell>\delta (kt)^\gamma  
        \Big] \Big)\bigg\}(1+\o{k}{})
        \\&=
        \exp\bigg\{- \sum_{\ell=\lceil k\varepsilon\rceil}^{\lceil kt\rceil}  
         \sum_{i\geq \delta (kt)^\gamma }
        (1-\e^{-\frac{u}{k^\gamma }i})\prob_\ell(H_\ell=i )\bigg\}(1+\o{k}{})\\&=
        \exp\bigg\{- \sum_{\ell=\lceil k\varepsilon\rceil }^{\lceil kt\rceil}
        \int_{\delta(kt)^\gamma}^\infty
        (1-
         \e^{-\frac{u}{k^\gamma }\lceil y\rceil})
       \prob_\ell(H_\ell =\lceil y\rceil)\dif y \bigg\}(1+\o{k}{}).
    \end{aligned}     
    \end{equation}
 Recall that for the conditioning on $L^{\sss(k)}_{\halpha}(x)$, 
 \begin{equation}
     a_\ell = \ell (W_\ell-p_c) / p_c
    =(\ell/k) \frac{L_{\halpha}^{\sss(k)}(\ell/k)}{(\halpha-1)},
 \end{equation}
  where we make the conditioning on $a_\ell$ explicit from this point. We perform a substitution of $y=xk^\gamma$ to find 
\begin{equation}
    \begin{aligned}
   & \expec_{[k\varepsilon,kt]}\Big[
        \e^{- \frac{u}{k^\gamma }\tilde{M}^{(\delta,\varepsilon)}_{ \lceil kt\rceil}}
        \Big]
    \\&=    
    \exp\bigg\{- \sum_{\ell=\lceil k\varepsilon\rceil }^{\lceil kt\rceil}
        \int_{\delta t^\gamma}^\infty
        (1-
         \e^{-ux})\\&\qquad \times
       \prob \Big(H_\ell =x (k/\ell)^\gamma \ell^\gamma
       \mid 
       a_\ell= (\ell/k) \frac{a^{\sss(k)}(\ell/k)}{\halpha-1}
       \Big )
       k^\gamma\dif x
       \bigg\}(1+\o{k}{})\\
         &=    
    \exp\bigg\{-\frac{1}{k} \sum_{\ell=\lceil k\varepsilon\rceil }^{\lceil kt\rceil}
        \int_{\delta t^\gamma}^\infty
       \Big(\frac{k}{\ell}\Big)^{\gamma+1}
        (1-
         \e^{-ux})
       \ell^{\gamma+1}\\&\qquad\qquad \times\prob \Big(H_\ell =x (k/\ell)^\gamma \ell^\gamma
       \mid 
       a_\ell =  (\ell/k)\frac{ a^{\sss(k)}(\ell/k)}{\halpha-1}
       \Big)\dif x\bigg\}(1+\o{k}{}).
    \end{aligned}
\end{equation}  
 This function in the exponent can now be seen as a Riemann sum over a function that also converges in $k$. We write this as 
 \begin{equation}
 \label{eq:Xi_riemann_volume_alpha>1}
 \expec_{[k\varepsilon,kt]}\Big[
        \e^{-\frac{u}{k^\gamma }\tilde{M}^{(\delta,\varepsilon)}_{ \lceil kt\rceil}}
        \Big]
        =
        \exp\bigg\{-
     \frac{1}{k} 
     \sum_{\ell=\lceil k\varepsilon\rceil }^{\lceil kt\rceil}
     \Xi^{(a^{\sss(k)},\delta)}_{\ell}(\ell/k)\bigg\}(1+\o{k}{}),
 \end{equation}
 with 
 \begin{equation}
 \label{eq:def_Xi^ak}
     \Xi^{(a^{\sss(k)},\delta)}_{\ell}(s)
     =
     \int_{\delta t^\gamma}^\infty 
     \frac{ 1-\e^{-ux}}{s^{\gamma+1}}
      \ell^{\gamma+1}\prob\Big(H_\ell =x s^{-\gamma} \ell^\gamma
       \mid 
       a_\ell = s \frac{a^{\sss(k)}(s)}{\halpha-1}
       \Big)\dif x.
 \end{equation}
 Note that the right-hand side of \eqref{eq:Xi_riemann_volume_alpha>1} is a Riemann sum where the function itself also depends on $k$. To show convergence to the Riemann integral, we need to show that
 $\Xi^{(a^{\sss(k)},\delta)}_{\ell}(s)$ converges uniformly for $\ell<k\to\infty$ in $s\in(\varepsilon,t)$. In order to do so, it remains to argue that for $\ell,k\to\infty$,\ $\Xi^{(a^{\sss(k)},\delta)}_{\ell}(s)$ converges uniformly in $s \in(\varepsilon,t)$.
 We rewrite \eqref{eq:def_Xi^ak} with a substitution of $z=xs^{-\gamma}$, then 
 \begin{equation}
 \label{eq:def_Xi^ak20}
     \Xi^{(a^{\sss(k)},\delta)}_{\ell}(s)
     =
     \int_{\delta t^\gamma}^\infty 
     \frac{ 1-\e^{-uzs^{\gamma}}}{s}
      \ell^{\gamma+1}\prob\Big(H_\ell =z\ell^\gamma
       \mid 
       a_\ell = s \frac{a^{\sss(k)}(s)}{\halpha-1}
       \Big)\dif z.
 \end{equation}
 As $(1-\e^{-uzs^\gamma})/s$ is uniformly bounded for $s \in(\varepsilon,t)$, we only need to show uniform convergence of the probability mass function.
 By Proposition~\ref{prop-size-trees-backbone->2} and \ref{prop-size-trees-backbone-(1,2)}, we know that $\prob_\ell(H_\ell = z\ell^\gamma\mid W_\ell= p_c(1+a/\ell) )$ converges uniformly in $a$ on bounded intervals and $z>\delta t^\gamma$. 
 By assumption of this lemma $sa^{\sss(k)}(s)$ is bounded for $s \in(\varepsilon,t)$. Indeed, as $a^{\sss(k)}\in D[\varepsilon,t]$, it is continuous from the left and therefore bounded in $\varepsilon$. Moreover, as we also assume that $a^{\sss(k)}(s)$  non-increasing, the limit of $sa^{\sss(k)}(s)$ for $s\to t$ is bounded too.
 Therefore the mass function converges uniformly in $sa^{\sss(k)}(s)$ and by extension in $s$. 
  Therefore,  $\Xi^{(a^{\sss(k)},\delta)}_{\ell}(s)\to \Xi_\ell^{(a^{*},\delta)}(s)$ uniformly in $s$. We conclude that
  \begin{equation}
  \begin{aligned}
      \lim_{k\to\infty }
      \expec_{[k\varepsilon,kt]}\Big[
        \e^{- \frac{u}{k^\gamma }\tilde{M}^{(\delta,\varepsilon)}_{ \lceil kt\rceil}}
        \Big]
    &=    
     \exp\bigg\{-\lim_{k\to\infty}
     \frac{1}{k} 
     \sum_{\ell=\lceil k\varepsilon\rceil }^{\lceil kt\rceil} 
     \Xi^{(a^{\sss(k)},\delta)}_{\ell}(\ell/k)
     +\o{k}{}\bigg\}
     \\&=\exp\bigg\{
     \int_\varepsilon^t 
     \Xi^{(a^*,\delta)}(s) \dif s
     \bigg\} .
     \end{aligned}
  \end{equation}
\end{proof}

In the following we aim to extend this result to the unconditional Laplace transform. Results should extend naturally as also $s L_{\halpha}(s)$ is bounded with high probability for all $s\geq 0$ according to Remark \ref{rem:stationarity_tLalpha(t)}.

\paragraphi{Part 3 - Extension to the unconditional convergence in distribution}
\label{par:6.1-part3}
In the previous part, we have shown that the conditional Laplace transform converges. But in the volume scaling of $M_k$, there exists additional randomness induced by the $L_{\halpha}^{\sss(k)}(t)$ process. In this part, we show the convergence of the unconditional Laplace transform. 
Consider the space $D[\varepsilon,t)$ the space of c\`adl\`ag functions on $[\varepsilon,t)$.
We define $f_k:D[\varepsilon,t)\to \mathbb{R}$ and $f:D[\varepsilon,t)\to \mathbb{R}$ as
\begin{equation}
\label{eq:def_fk_laplace}
    f_k(L^{\sss(k)}_\alpha([\varepsilon,t]) ) 
    =
    \exp\bigg\{-\frac{1}{k} \sum_{\ell=\lceil k\varepsilon\rceil}^{\lceil kt\rceil} \Xi^{(L^{\sss(k)}_\alpha ,\delta)}_\ell(\ell/k) + \o{k}{}\bigg\},
\end{equation}
and
\begin{equation}
\label{eq:def_f_LT}
    f(L_{\halpha}([\varepsilon,t])) = 
    \exp\bigg\{-\int_\varepsilon^t
    \Xi^{(L_{\halpha},\delta)}(s) 
    \dif s\bigg\}.
\end{equation}
Furthermore, by Theorem~\ref{thm-subthmWk}, we know that 
\begin{equation}
\label{eq:reminder_conv_Lprocess}
    L_{\halpha}^{\sss(k)}([\varepsilon,t]) \xrightarrow{\mathcal{D}}
    L_{\halpha}([\varepsilon,t]).
\end{equation}
In the following, we show convergence of $f_k(L_{\halpha}^{\sss(k)}([\varepsilon,t]))$ to $f(  L_{\halpha}([\varepsilon, t]))$ in distribution. 

\begin{lemma}[Convergence of the Laplace transform]
\label{lem:LT_conv_J1_topology}
Recall the function sequence $(f_k)_{k\geq 1}$ from \eqref{eq:def_fk_laplace} and $f$ from \eqref{eq:def_f_LT}. 
Furthermore, we recall
$L_{\halpha}^{\sss(k)}([\varepsilon,t])$ from \eqref{eq:def_L_halpha^(k)}, then, for every fixed $t$,
\begin{equation}
    f_k(L_{\halpha}^{\sss(k)}([\varepsilon,t]))
    \xrightarrow{d}
    f(L_{\halpha}([\varepsilon,t])).
\end{equation}
\end{lemma}
\begin{proof}
 By Lemma~\ref{lem:volume_scaling_kcut_a>1_1D},
 $f_k(x_k)\to f(x)$ for $x_k\xrightarrow{\mathcal{D}} x$ uniformly if $x$ is non-increasing.
 The claim then follows from the generalised continuous function theorem \cite[Theorem 3.4.4]{Whitt2002}, as $L_{\halpha}([\varepsilon,t])\in D[\varepsilon,t]$ and is  non-increasing by definition.
\end{proof}
The result of Lemma \ref{lem:LT_conv_J1_topology} implies that, for fixed $\delta>0$ and for $k\to\infty$,
\begin{equation}
\label{eq:volume_scaling_alphain(1,2)_convind}
      \expec[ \e^{-\frac{u}{ k^{\gamma}}\tilde M^{(\delta,\varepsilon)}_{ \lceil kt\rceil}}
    \mid L^{\sss(k)}_{\halpha}([\varepsilon,t]) ]
    \xrightarrow{d} 
    \expec[ \e^{-\int_\varepsilon^t\Xi^{(L_{\halpha},\delta)}(s)\dif s }
    \mid 
    L_{\halpha}([\varepsilon,t])].
\end{equation}

\paragraphi{Part 4 - Convergence to the Cox process}
\label{par:6.1-part4}
Next, we show that the right-hand side of \eqref{eq:volume_scaling_alphain(1,2)_convind} equals the transform of a Cox process as defined in \eqref{eq:Poisson_measure_goal}.
By \cite[Theorem 3.9, Exercise 3.4]{Last2017}, we find the Laplace transform of the Cox process in \eqref{eq:Poisson_measure_goal} with corresponding intensity measure in $\lambda_\alpha$ as defined in \eqref{eq:def_lambda_PPP} as
\begin{equation}
\label{eq:volume_growth_alpha>2_eq10}
\begin{aligned}
   & \expec\Big[
    \e^{-u \int_0^t\int_{0}^\infty x
    \Pi(\dif x,\dif s) }
    \mid L_{\halpha}([\varepsilon,t])
    \Big] 
    \\&\qquad =
    \expec\Big[\e^{-\int_0^t\int_{0}^\infty x
    (1-\e^{-ux}) \lambda_\alpha(\dif x,\dif s)}\mid  L_{\halpha}([\varepsilon,t])\Big]
    \\&\qquad =
    \exp\bigg\{-\int_0^t\int_{0}^\infty (1-\e^{-ux})
    \frac{h_\alpha(xs^{-\gamma},s \frac{L_{\halpha}(s)}{\halpha-1})}{s^{\gamma+1}}\dif x\dif s
    \bigg\},
\end{aligned} 
\end{equation}
which is in correspondence with the right-hand side expression of \eqref{eq:volume_scaling_alphain(1,2)_convind}, where we recall the definition of $\Xi^{(L_{\halpha}(s),\delta)}$ from \eqref{eq;def_bigXi}.
This is sufficient to extend the results to convergence of $k^{-\gamma}M_{\lceil kt\rceil}$ for some fixed $t$, which we conclude by the following corollary:
\begin{corollary}
\label{cor:convergence_volume_alpha>1}
Consider the volume scaling of the $k$-cut IPC $M_{k}$ from \eqref{eq:def_M_k} and $\gamma$ from \eqref{eq:def_gamma_exponent2}. Then, for fixed $t$, 
\begin{equation}
    k^{-\gamma} M_{\lceil kt\rceil }
    \xrightarrow{d}
    \int_0^t\int_0^\infty x
    \Pi_{\lambda_\alpha}(\dif x,\dif s).
\end{equation}
\end{corollary}
\begin{proof}
Recall that $f_k$ and $f$ are bounded between $[0,1]$ and therefore by the Portmanteau theorem we find
\begin{equation}
    \expec[ f_k(L_{\halpha}^{\sss(k)}([\varepsilon,t]))] 
    \to 
    \expec[ f(L_{\halpha}([\varepsilon,t]))].
\end{equation}
Recalling the definition of $f_k$ and $f$ and by Lemma \ref{lem:volume_no_small_clusters_alpha>1} and \eqref{eq:volume_growth_alpha>2_eq10}, bounded convergence implies
\begin{equation}
\begin{aligned}
\expec[\e^{-\frac{u}{k^{\gamma}} M_{ \lceil kt\rceil }} ] &=
    \expec\Big[\e^{-\frac{u}{k^{\gamma}}\tilde M^{(\delta,\varepsilon)}_{ \lceil kt\rceil }+\o{1/\varepsilon,1/\delta}{k}} \Big](1+\o{k}{})
    \\&\xrightarrow{k\to\infty} \expec\Big[ \e^{-\int_\varepsilon^t\Xi^{(L_{\halpha},\delta)}(s) \dif s+\o{1/\varepsilon,1/\delta}{}}
   \Big]
    \\&\xrightarrow{\delta,\varepsilon\to 0}
      \expec\Big[
    \e^{-u \int_0^t\int_{0}^\infty x
    \Pi_{\lambda_\alpha}(\dif x,\dif s)}
    \Big] ,
    \end{aligned}
\end{equation}
where the last convergence follows form \eqref{eq:volume_growth_alpha>2_eq10} and the definition of $\Xi^{(L_{\halpha},\delta)}$ from \eqref{eq;def_bigXi}
\end{proof}

\paragraphi{Part 5 - Extension to $J_1$-convergence}
\label{par:6.1-part5}
To extend the result of the previous part to $J_1$ convergence, convergence of the finite-dimensional distribution is needed and shown as follows:

\begin{lemma}[Extension to the finite-dimensional distribution]
\label{lem:finite_dim_assumption}
Let $0<\varepsilon<t_1<\cdots<t_n$ and $s_1,\ldots,s_n$ be arbitrary positive values. 
Then convergence of the finite-dimensional distribution follows for $\alpha>2$ and $\alpha\in(1,2)$.
\end{lemma}
\begin{proof}
Based on the analysis performed in part 2-4, it also naturally follows that, for $j\in\{1,\ldots,n\}$,
\begin{equation}
\begin{aligned}
\label{eq:finite_dim_assumption}
   & \expec\bigg[
    \e^{-u_j \sum_{\ell=\lceil kt_{j-1}\rceil}^{\lceil kt_{j}\rceil} H_\ell/k^{\gamma}}\ \Big|\
    L_{\halpha}^{\sss(k)}([t_{j-1},t_{j}])
    \bigg]\\&\qquad \xrightarrow{d} 
    \expec\bigg[
    \e^{-u_j\int_{t_{j-1}}^{t_{j}} x\Pi(dx,ds)}
    \ \Big|\
    L_{\halpha}([t_{j-1},t_{j}])
    \bigg].
    \end{aligned}
\end{equation}
With this, we can define the Laplace transform of the finite-dimensional distribution and show it converges. We start with
 \begin{equation}
 \label{eq:finite_dim_eq1}
 \begin{aligned}
    & \expec[\e^{- (u_1M_{\lceil kt_1\rceil }/k^{\gamma} +\cdots+
     u_nM_{\lceil kt_n\rceil }/k^{\gamma})}
     \mid
    L_{\halpha}^{\sss(k)}([\varepsilon,t_n])]\\&=
     \expec\bigg[\exp\bigg\{
     - (u_1+\cdots+u_n)\sum_{\ell=0}^{\lceil kt_1\rceil } H_\ell/k^\gamma 
     - \cdots\\& \qquad
     -u_n\sum_{\ell=\lceil kt_{n-1}\rceil +1}^{\lceil kt_n\rceil } H_\ell/k^\gamma
     \bigg\} \Big|~
    L_{\halpha}^{\sss(k)}([\varepsilon,t_n]) \bigg].
    \end{aligned}
 \end{equation}
As these are the volumes of disjoints forests, they are independent conditionally on $(W_k)_{k\geq 1}$ and therefore \eqref{eq:finite_dim_eq1} equals
\begin{equation}
    \prod_{j=1}^n 
    \expec\bigg[
    \exp\bigg\{ 
    \bigg(-\sum_{\ell=j}^n u_\ell\bigg)\sum_{\ell=\lceil kt_{j-1} \rceil +1}^{\lceil kt_{j}\rceil }
    H_\ell/k^\gamma
    \bigg\}
    \Big|~
    L_{\halpha}^{\sss(k)}([\varepsilon,t_n])
    \bigg] ,
\end{equation}
 under the convention that $kt_0+1=0$. Under the condition of \eqref{eq:finite_dim_assumption}, this converges in distribution to 
 \begin{equation}
      \prod_{j=1}^n 
    \expec_k\bigg[
    \exp\bigg\{ 
    \bigg(-\sum_{\ell=j}^n u_\ell\bigg)\int_{t_{j-1}}^{t_{j}}\int_0^\infty
   x\Pi(dy,ds)
    \bigg\} 
    \Big|~ (L_{\halpha}([\varepsilon,t_n])
    \bigg] ,
 \end{equation}
which can be rearranged to the finite-dimensional distribution of the limiting stochastic process, which shows the claim.  
\end{proof}

Based on \cite[Theorem 13.1]{Billingsley1999}, $J_1$-convergence follows if the finite-dimensional distribution converges, and the sequence is tight. The former follows from Lemma~\ref{lem:finite_dim_assumption} and the latter is proved in Appendix \aprefa{E.2}{App:tightness_volume>1}.
\end{proof}

\subsection{\mtitle{$k$}-cut IPC scaling for  \mtitle{$\alpha\in(0,1)$}}
\label{sec:volume_scaling_alpha<1}
We conclude the section with the volume scaling for the $k$-cut IPC for $\alpha\in(0,1)$ and thereby conclude the proof of Theorem~\ref{thm-main}. 
Contrary to $\alpha>2$ and $\alpha\in(1,2)$, this process converges in the $J_1$-topology to a degenerate process. Instead we prove convergence of the infinite-dimensional distribution in the product topology, i.e.,
\begin{equation}
    \Big(
    M_{k-\ell} W_k^{\alpha/(1-\alpha)}
    \Big)_{\ell\geq 0} 
    \xrightarrow{d}
    (Z_\ell)_{\ell\geq 0},
\end{equation}
where the discrete process $(Z_\ell)_{\ell\geq0}$ is defined as
\begin{equation}
\label{eq:def:Z_n_alpha_in_(0,1)}
Z_\ell = \sum_{i=\ell}^\infty 
\prod_{j=1}^{i} P_j^{\frac{\alpha}{1-\alpha}}\Psi^{(i)},
\end{equation}
and $(\Psi^{(i)})_{i\geq1}$ is an
 i.i.d. sequence  that has density $y\mapsto h_\alpha(y)$ as defined in \eqref{eq:def_h_alpha(0,1)} and i.i.d. sequence $(P_i)_{i\geq1}$ as defined in \eqref{eq:def_P}. 
 We continue with the proof of Theorem~\ref{thm-main} for $\alpha\in(0,1)$. This is followed by a short discussion of the scaling of $(M_{\lceil kt\rceil})_{t\geq0}$ and its finite-dimensional distribution.
 
\begin{proof}[Proof of Theorem~\ref{thm:Cs_scaling} for $\alpha\in(0,1)$]
We ignore again the contribution of the $k$ backbone vertices in $M_k$ and focus on the attached forests. The proof is then structured based on the following parts:
\begin{enumerate}
 \item \hyperref[par:6.2-part0]{\textbf{Convergence of the conditional one-dimensional Laplace transform. }}
    We start by fixing $\ell\in\mathbb{N}$, and show convergence of  $M_{k-\ell}W_k^{\alpha/(1-\alpha)}$ via the Laplace transform conditioned on $(W_i)_{i= 0}^k=(w_i)^k_{i= 0}$ in Lemma~\ref{lem:conv_one-dim_volume_alphain(0,1)}. 
     \item \hyperref[par:6.2-part1]{\textbf{Convergence of the unconditional one-dimensional distribution.} }
    We extend this to convergence in distribution of the non-conditional Laplace transform, i.e. driven by the random process $(W_k)_{k\geq 0}$ in Lemma~\ref{lem:conv_uncond_LT_alphain(0,1)}. 
    \item \hyperref[par:6.2-part2]{\textbf{Convergence of the finite-dimensional distributions.}}
    These results so far are for a fixed $t$ and in the last part
    we extend the result to the finite-dimensional distribution. This concludes the proof.
\end{enumerate}

\paragraphi{Convergence of the conditional one-dimensional Laplace transform} 
 \label{par:6.2-part0}
 We start with showing convergence of the one-dimensional distribution, where we condition on the $(W_k)_{k\geq 0}$ process:
 
\begin{lemma}[Convergence of the one-dimensional distribution]
\label{lem:conv_one-dim_volume_alphain(0,1)}
Fix $\alpha\in(0,1)$.
Fix some $\ell\in\mathbb{N}$ and consider a sequence $(w_i)_{i\geq 0}$ such that $w_{k}/w_{k-i}\to a_i\in(0,1)$ for $k\to\infty$ and all $i$. 
Then 
\begin{equation}  
\begin{aligned}
   & \expec\Big[ \e^{-sW_k^{\alpha/(1-\alpha)}
    M_{k-\ell}}
    \mid (W_i)^k_{i= 0}=(w_i)^k_{i= 0}\Big]
   \\&\qquad  \to
    \expec\Big[\e^{
    -s\sum_{i=\ell}^\infty a_i \Psi^{(i)}
    }\mid   (W_i)^k_{i= 0}=(w_i)^k_{i= 0}
    \Big]
    .
    \end{aligned}
\end{equation}
\end{lemma}
\begin{proof}
For convenience we abbreviate $\expec_{[k]}[\ \cdot\ ] = \expec[\ \cdot \mid  (W_i)^k_{i= 0}=(w_i)^k_{i= 0}]$.
 As $W_k$ scales exponentially, we restrict ourselves to only the forests at finite distance of $k$. Note that
     \begin{align}
      \expec_{[k]}\Big[\e^{-sM_{k-\ell}W_k^{\alpha/(1-\alpha)}} \Big]
      &=
      \expec_{[k]}\bigg[\exp\bigg\{-s\sum_{i=0}^{k-\ell}H_{i}W_k^{\alpha/(1-\alpha)}\bigg\} \bigg]\nonumber\\&
      =
      \expec_{[k]}\bigg[\exp\bigg\{-s\sum_{i=\ell}^{k}H_{k-i}W_k^{\alpha/(1-\alpha)}\bigg\} \bigg]
      \\&
      =\prod_{i=\ell}^k 
      \expec_{[k]}\bigg[\exp\bigg\{-s\frac{W_k^{\alpha/(1-\alpha)}}{W_{k-i}^{\alpha/(1-\alpha)}} H_{k-i}W_{k-i}^{\alpha/(1-\alpha)}\bigg\} \bigg].\nonumber
     \end{align}
Let $\phi_{k}(\cdot)$ denote the Laplace transform of $H_{k-i}W_{k-i}^{\alpha/(1-\alpha)}$ and $\phi(\cdot)$ the Laplace transform of $\Psi^{(i)}$, then by Proposition~\ref{prop-size-trees-backbone-(0,1)} and the continuity theorem \cite[Theorem 15 Chapter 14.7]{Fristedt1997}
$\phi_{k}(s) \rightarrow \phi(s)$ uniformly on bounded intervals. 
As $w_{k}/w_{k-i}$ is bounded in $k$, we also find $ \phi_{n}(sw_{k}/w_{k-i})\rightarrow \phi(s a_{i})$, i.e.
\begin{equation}
 \expec_{[k]}\bigg[ \exp\bigg\{-sH_{k-i}w_{k-i}^{\alpha/(1-\alpha)}\frac{w_k^{\alpha/(1-\alpha)}}{w_{k-i}^{\alpha/(1-\alpha)}}\bigg\}
    \bigg]\rightarrow
    \expec\Big[\exp\{-s\Psi^{(i)}a_{i}\}
    \Big].
\end{equation} 
Taking the product over $i$ and letting $k \rightarrow \infty$ implies the claim.
\end{proof}

\paragraphi{Convergence of the unconditional one-dimensional distribution}\label{par:6.2-part1}
Next, we extend this to convergence of the unconditional Laplace transform, which follows in a straightforward manner.
\begin{lemma}[Convergence of the unconditional Laplace transform]
\label{lem:conv_uncond_LT_alphain(0,1)}
Fix $\alpha\in(0,1)$ and $\ell\in\mathbb{N}$, then 
\begin{equation}
    \expec\Big[ \e^{-W_k^{\alpha/(1-\alpha)}
    M_{k-\ell}}\Big]
    \to
    \expec\Big[\e^{
    -s\sum_{i=\ell}^\infty \prod_{j=1}^i P_{j}^{\alpha/(1-\alpha)}\Psi^{(i)}
    }
    \Big].
\end{equation}
\end{lemma}
\begin{proof}
We define the functions
\begin{equation}
    f_k( (a_i)^k_{i= 1}) = 
    \expec\Big[ \e^{-sW_k^{\alpha/(1-\alpha)}
    M_{k-\ell}}
    \mid (W_k/W_{k-i})^k_{i= 0}=(a_i)^k_{i=0}\Big]
\end{equation}
and
\begin{equation}
    f((a_i)^k_{i= 1})
    =\expec\Big[
    \e^{-s\sum_{i=\ell}^\infty \Psi^{(i)} a_i}
    \mid (W_k/W_{k-i})^k_{i= 0}=(a_i)^k_{i= 0}\Big].
\end{equation}
Then, based on the previous part, we have convergence of $f_k$ to $f$,
restricted to $a_i\in(0,1)$ for all $i<k$. By Lemma \ref{lem:prod_Wk} it follows that $W_k/W_{k-i}$ has a limit in $k$ jointly for all $i<k$  that is between $(0,1)$. Therefore, by the generalised continuous function theorem \cite[Theorem 3.4.4]{Whitt2002}, we find, the limit being
\begin{equation}
    f_k(W_k/W_{k-i})
    \xrightarrow{d}
    f \Big(\prod^{i}_{j=1} P_{j}^{\alpha/(1-\alpha)}\Big).
\end{equation}
Moreover, as $f_k(\cdot)$ and $f(\cdot)$ are bounded functionals, the Portmanteau theorem implies convergence of the respective expectations.
\end{proof}

\paragraphi{Convergence of the finite-dimensional distributions}
\label{par:6.2-part2}
As $(Z_\ell)_{\ell\geq 0}$ is a discrete-time process, it remains to show convergence of the finite-dimensional distribution. This directly proves Theorem \ref{thm-main} for $\alpha\in(0,1)$. The proof shows similarities to Lemma~\ref{lem:finite_dim_assumption} but is catered to the result of Lemma~\ref{lem:conv_one-dim_volume_alphain(0,1)}.

We prove the result via the Laplace transform.
 Consider $s_1,\ldots,s_n>0$ then we consider convergence of 
 \begin{equation}
 \begin{aligned}
 \label{eq:fin_dim_distr_alpha(0,1)_1}
     &\expec_{[k]}\Big[\exp\Big\{-s_1M_{k-\ell_1}W_k^{\alpha/(1-\alpha)}-\cdots
     -s_nM_{k-\ell_n}W_k^{\alpha/(1-\alpha)}
     \Big\}\Big]
     \\& =
     \expec_{[k]}\bigg[\exp\bigg\{
     -(s_1+\cdots+s_n) \sum_{i=0}^{k-\ell_n}
     H_iW_k^{\alpha/(1-\alpha)}-\cdots\\&\qquad\qquad
     -s_1 \sum_{i=k-\ell_2+1}^{k-\ell_{1}}H_iW_k^{\alpha/(1-\alpha)}
     \bigg\}\bigg] .
\end{aligned}
 \end{equation}
 Conditionally on $(W_k)_{k\geq 0}$, we note that by independence of the forest sizes we can write \eqref{eq:fin_dim_distr_alpha(0,1)_1} to
\begin{equation}
    \prod_{u=1}^{n}
    \expec_{[k]}\bigg[ 
    \exp\bigg\{
    -\bigg(\sum_{i=1}^u s_i\bigg) 
    \sum_{i=k-\ell_{u+1}+1}^{k-\ell_{u}}
    H_i W_k^{\alpha/(1-\alpha)}
    \bigg\} 
    \bigg],
\end{equation}
where $\ell_{n+1}+1=k$ by convenience.
By Lemma~\ref{lem:conv_one-dim_volume_alphain(0,1)}, we know that these Laplace transforms converge to  
\begin{equation}
    \prod_{u=1}^{n}
    \expec_{[k]}\bigg[ 
    \exp\bigg\{
    -\bigg(\sum_{i=1}^u s_i\bigg) 
    \sum_{i=\ell_{u}}^{\ell_{u+1}-1}
    \prod_{j=1}^i
    P_j^{\alpha/(1-\alpha)} \Psi^{(i)}
    \bigg\} 
    \bigg] .
\end{equation}
As these are still independent random variables, they can be taken in one expectation. Rearranging the terms results in 
\begin{equation}
\begin{aligned}
   & \expec_{[k]}\bigg[\exp\bigg\{
    -
    \sum_{u=1}^n
    \bigg( \sum_{i=1}^u s_i\bigg)
     \sum^{\ell_{u+1}-1}_{i=\ell_u} 
    \prod_{j=1}^i
    P_j^{\alpha/(1-\alpha)} \Psi^{(i)}-
    ...
    \bigg\}\bigg] 
   \\& = \expec_{[k]}\bigg[\exp\bigg\{-\sum_{u=1}^ns_u Z_{\ell_u}
    \bigg\}\bigg]. 
    \end{aligned}
\end{equation}
According to the same argumentation as in the previous paragraph, we can take expectations on both sides and conclude the convergence of the finite dimensional distribution. Therefore, we find, for all $n$,
\begin{equation}
    (M_{k-\ell} W_k^{\alpha/(1-\alpha)})_{\ell=0}^n
    \xrightarrow{d}
    (Z_\ell)_{\ell=0}^n.
\end{equation}
This concludes the proof by taking $n\to\infty$.
\end{proof}
We conclude with a reflection on the type of convergence we chose for $\alpha\in(0,1)$. Due to the exponential decay of the $(W_k)_{k\geq0}$ process, a similar scaling of $(M_{\lceil kt\rceil })_{t\geq 0}$ as for $\alpha>2$ or $\alpha\in(1,2)$, would result in a limit independent of $t$. Indeed, fix $t>0$ then by the result of Lemma~\ref{lem:conv_one-dim_volume_alphain(0,1)} for $\ell=0$:
\begin{equation}
\label{eq:volumescaling_limit_alpha<1}
   M_{\lceil kt\rceil }W_{\lceil kt\rceil }^{\frac{\alpha}{1-\alpha}} \xrightarrow{d} 
   Z_0.
\end{equation}
Moreover, for a set $t_0<t_1<\cdots<t_n$, the finite-dimensional distribution of $(M_{\lceil kt\rceil }W_{\lceil kt\rceil}^{\frac{\alpha}{1-\alpha}})_{t>0}$ converges to an i.i.d.\ sequence, which follows in a straightforward manner.

\begin{appendix}
    \section{List of symbols and notation}

\begin{longtable}{ll}
    IPC & Invasion percolation cluster\\
    IIC & Incipient infinite cluster\\
    BP (tree) & Branching process (tree)\\
    $\alpha$-LEP & $\alpha$-enhanced lower-envelope process, see $L_{\halpha}(t)$\\
    EPDP& Exponential Poisson decay process, see $\L_\alpha(t)$\\
    w.h.p. & With high probability \\
    i.i.d. & Independent and identically distributed\\
    &\\
        $\clock$ & Root of the tree\\ 
       $ \mathcal{T}$ & BP tree with uniform weights on the edges\\
        $\mathcal{T}(p)$ & BP tree restricted to only weights smaller than $p$\\
        $T^{(i)}(\clock)$ & Invasion percolation cluster, started in the root, in $i$ steps\\
        $T(\clock)$ & IPC \\
        $\T_k(\clock)$ & $ T(\clock) \cap v_k(\clock)$\\
        $\M_k(\clock)$ & $k$-cut IPC\\
        $M_k$ & $|\M_k(\clock)|$\\
         $\tilde{T}_{i,k}$ & $i$-th conditionally finite tree attached to $v_k$\\
        $H_k$ & (conditionally finite) Volume of the total forest attached to $v_k$: $\sum_{i=1}^{\hat{D}_{v_k}}\tilde{T}_{i,k}$ \\
        $T^{\sss(k)}$ & Total progeny of a tree with distribution $X^{\sss(k)}$ \\
        &\\
        $p_c$ & Critical percolation coefficient\\
        $(W_k)_{k\geq 0}$ & Future maximum weight after $k$-th backbone vertex\\
        $(\beta_i)_{i\geq 1}$ & Weight between the $(i-1)$-st and $i$-th backbone vertex\\
        $B_k(\clock)$ & Set of vertices at most graphs distance $k$ away from the root\\
        &\\
         $X$& Degree distribution in $\mathcal{T}$, which is assumed to be a power-law  \\
         $X^\star$ & Size-biased distribution of $X$\\
         $X^{\sss(k)}$ & $\Binom(X,W_k)$\\
         $\tilde X^{\sss(k)}$ & $\tilde X$, conditioned on extinction of a tree with offspring distribution $ X^{\sss(k)}$ \\
         $Q_n$ & $\sum_{i=1}^n X_i$, for i.i.d. $X_i\sim X$\\
         $(v_k)_{k\geq 0}$ & Sequence of backbone vertices, where $v_0=\clock$\\
        $(D_{v_k})_{k\geq 0}$ & Degree sequence of backbone vertices\\
        $\hat{D}_{v_k}$ & $\Binom(D_{v_k},w_k)$\\
        &\\
        $\theta(p)$ &Survival probability of $\mathcal{T}(p)$\\
        $\hat{\theta}(p)$& 
        $p\theta(p)$,\ i.e. the
        survival probability of the BP that is attached to the backbone  \\
        $\eta(p)$ &
        $1-\theta(p)$,\ i.e. the
        extinction probability of $\mathcal{T}(p)$\\
        &\\
        $\alpha$ & Power-law exponent of the offspring distribution\\
        $\halpha$ & $\min\{\alpha,2\}$\\
        $\gamma $ & $2$ for $\alpha>2$ and $\alpha/(\alpha-1)$ for $\alpha\in(1,2)$. (not defined for $\alpha\in(0,1)$) \\
        &\\
        $L_{\eta,\halpha}(t)$ & $\alpha$-Enhanced generalised lower envelope process\\
        $L_{\halpha}(t)$ & $\alpha$-Enhanced lower envelope process, $L_{0,\halpha}(t)= L_{\halpha}(t)$\\       
        $\L_{\eta,\alpha}(t)$ & Exponential Poisson decay process\\
        &\\
        $\Gamma(\alpha)$ & The gamma function, $\alpha>-1$\\
        $\Gamma(\alpha,x)$ &The (upper) incomplete gamma function, $\alpha >-1$ and $x>0$\\
        &\\
        $\o{k}{a,x}$ & An error term that converges to 0 for $k\to\infty$ uniformly in $a\in[\varepsilon,1/\varepsilon]$ and $x>\varepsilon$\\
        $[n]$ & $\{1,\ldots,n\}$
        \\&
        \\      
$\prob_k(\cdot)$ & $\prob(\cdot\mid W_k=w_k)$\\
$\expec_k[\cdot]$ & $\expec[\cdot\mid W_k=w_k]$\\
&\\
$\Binom(n,p)$&Binomial distribution with $\expec[\Binom(n,p)]=np$\\
$\GammaD(\alpha,\lambda)$& Gamma distribution with $\expec[\GammaD[k,\lambda]]=\alpha/\lambda$\\
$\Exp(\lambda)$&Exponential distribution with $
\expec[\Exp(\lambda)]=1/\lambda $\\
$\Pois(\lambda)$&Poisson distribution with $\expec[\Pois(\lambda)]=\lambda$\\
$\Unif[a,b]$&Uniform distribution with $\expec[\Unif[a,b]]=(a+b)/2$\\
&\\
$\xrightarrow{d}$ & Convergence in distribution\\
$\xrightarrow{\prob} $ & Convergence in probability\\
$\xrightarrow{\mathcal{D}} $ & Convergence in the (Skorohod) $J_1$-topology
\end{longtable}

\section{Key details on the derivations }
\subsection{Details on the differential equation}
\label{app:difeq}
We simplify \eqref{eq:Differential_equation_g} with the use of Laplace transforms.
We first define
\begin{equation}
\label{eq:Transform_differential_eq}
     \hat{F}_{g^{(\varepsilon)}}(x,s)=\int_0^\infty \e^{-st} g^{(\varepsilon)}(x,t) \dif t
    =\int_0^\infty \e^{-st} \int_x^\infty f_{L^{(\varepsilon)}_{\eta,\halpha}(t)}(y) \dif y \dif t.
\end{equation}
Secondly, for notational convenience we define $g^{(\varepsilon)}_x(x,t)= \partial/(\partial x)g^{(\varepsilon)}(x,t)$ and $g^{(\varepsilon)}_t(x,t)$ in a similar fashion.
We observe that by integration by parts,
\begin{equation}
    \begin{aligned}
    \label{eq:simplify_distribution_alphaLEP}
        \hat{F}_{g_t^{(\varepsilon)}}(x,s)=\int_0^\infty \e^{-st} g^{(\varepsilon)}_t(x,t)\dif t
       & =[\e^{-st} g^{(\varepsilon)}(x,t))]_{t=0}^{t=\infty}+s \int_0^\infty \e^{-st} g^{(\varepsilon)}(x,t) \dif t
        \\&=-g^{(\varepsilon)}(x,0) +s \hat{F}_{g^{(\varepsilon)}}(x,s),
    \end{aligned}
\end{equation}
which exists as $g^{(\varepsilon)}(x,t)$ is a distribution for all $t\geq 0$. 
By taking the Laplace transform on both sides of \eqref{eq:Differential_equation_g} gives
\begin{equation}
\label{eq:simplify_distribution_alphaLEP_P2}
   c \int_0^\infty \e^{-st} g_t^{(\varepsilon)}(x,t)\dif t
   =
   \int_0^\infty\int_x^\infty \e^{-st}
        \frac{x^{(1+\eta)/(\hat{\alpha}-1)}}{y^{(1+\eta)/(\hat{\alpha}-1)}-1}
        g_y^{(\varepsilon)}(y,t)
        \dif y\dif t.
\end{equation}
We notice that the left-hand side exists by \eqref{eq:simplify_distribution_alphaLEP} and therefore the right-hand side must exist as well. Furthermore, we notice that the expression in the integral on the right-hand side does not change its sign, meaning we can switch the integrals to simplify \eqref{eq:simplify_distribution_alphaLEP_P2}
to
\begin{equation}
\label{eq:simplify_distribution_alphaLEP_P3}
  c \hat{F}_{g_t^{(\varepsilon)}}(x,s)=
  \int_x^\infty \frac{x^{(1+\eta)/(\hat{\alpha}-1)}}{y^{(1+\eta)/(\hat{\alpha}-1)}-1} \hat F_{g_y^{(\varepsilon)}}(x,s) 
        \dif y.
\end{equation}
Differentiation of both sides with respect to $x$ and recalling  \eqref{eq:simplify_distribution_alphaLEP} and
\eqref{eq:simplify_distribution_alphaLEP_P3} results in
\begin{equation}
\begin{aligned}
  &cx(-g^{(\varepsilon)}_x(x,0) +s\hat{F}_{g_x^{(\varepsilon)}}(x,s) ) \\&\quad -c\frac{1+\eta}{\halpha-1}(-g^{(\varepsilon)}(x,0)+ s\hat{F}_{g^{(\varepsilon)}}(x,s))+x^2 \hat{F}_{g_x^{(\varepsilon)}}(x,s)=0.
  \end{aligned}
\end{equation}
This can be written in a more conventional way as
\begin{equation}
    \hat F_{g^{(\varepsilon)}_x}(x,s) - \frac{sc}{(\halpha-1)(x^2+scx)}
    \hat F_{g^{(\varepsilon)}}(x,s)+ c\frac{ -xg^{(\varepsilon)}_x(x,0)+\frac{1+\eta}{\halpha-1} g^{(\varepsilon)}(x,0) }{x^2+scx}=0,
\end{equation}
which is a first-order differential equation with respect to $x$ for the Laplace transform in $t$ and can be solved by the method of integrating factor \cite[Section 2.1]{Boyce2009}. This results in
\begin{equation}
\label{eq:diff_equation_finalise_1}
\begin{aligned}
   & \hat{F}_{g^{(\varepsilon)}}(x,s)=c
     \Big(\frac{x}{x+sc}\Big)^{\frac{1+\eta}{\halpha-1}}
    \int 
    \Big(
    \frac{x+sc}{x}
    \Big)^{\frac{1+\eta}{\halpha-1}}
    \frac{
    xg^{(\varepsilon)}_x(x,0)-\frac{1+\eta}{\halpha-1} g^{(\varepsilon)}(x,0) }{x^2+scx}
    \dif x\\&\quad=
   - c
     \Big(\frac{x}{x+sc}\Big)^{\frac{1+\eta}{\halpha-1}}
    \int_x^\infty 
    \Big(
    \frac{y+sc}{y}
    \Big)^{\frac{1+\eta}{\halpha-1}}
    \frac{
    yg^{(\varepsilon)}_x(y,0)-\frac{1+\eta}{\halpha-1} g^{(\varepsilon)}(y,0) }{y^2+scy}
    \dif y \\&\quad\quad +C(s)\Big(\frac{x}{x+sc}\Big)^{\frac{1+\eta}{\halpha-1}},
    \end{aligned}
\end{equation}
where $C(s)$ is some constant that can depend on $s$ and where the second equality follows as $\hat F_{g^{(\varepsilon)}}(x,s)\to 0$ for $x\to\infty$, based on the definition in \eqref{eq:Transform_differential_eq}.
Recall that $g^{(\varepsilon)}(x,0)$ is the tail distribution of $c$ times a gamma random variable with parameters $( (1+\eta)/(\halpha-1),\varepsilon)$.

Consider next the expression in \eqref{eq:diff_equation_finalise_1}, that we slightly simplify by
\begin{equation}
\begin{aligned}
\label{eq:diff_equation_finalise_C1}
   \hat{F}_{g^{(\varepsilon)}}&(x,s)=-
     \Big(\frac{x}{x+sc}\Big)^{\frac{1+\eta}{\halpha-1}}     
     \int_x^\infty 
    \Big(
    \frac{y+sc}{y}
    \Big)^{\frac{1+\eta}{\halpha-1}}
    \frac{
    cg^{(\varepsilon)}_x(y,0) }{y+sc}
    \dif y
    \\&\quad +\Big(\frac{x}{x+sc}\Big)^{\frac{1+\eta}{\halpha-1}}
    \underbrace{
    \int_x^\infty 
     \Big(\frac{y+sc}{y}
    \Big)^{\frac{1+\eta}{\halpha-1}}
    \frac{c\frac{1+\eta}{\halpha-1} g^{(\varepsilon)}(y,0)}{y^2+scy}
    \dif y}_{(\spadesuit)}+C(s)\Big(\frac{x}{x+sc}\Big)^{\frac{1+\eta}{\halpha-1}},
    \end{aligned}
\end{equation}
where we recall that $g^{(\varepsilon)}(x,\varepsilon)$ is the tail distribution of a gamma random variable with shape $(\eta+1)/(\alpha-1)$ and rate $\varepsilon$.
We first consider the part that is marked with $(\spadesuit)$ in \eqref{eq:diff_equation_finalise_C1}. We use partial integration to simplify this integral.
In order to do so, we make use of the following integral equality  that can be obtained by a tactical substitution of  $ z=(y+sc)/y$, so that $\dif z = -sc/y^2\dif y$: 
\begin{equation}
\begin{aligned}
\label{eq:diff_equation_finalise_1.4}
    \int_x^\infty \Big( \frac{y+sc}{y}
   & \Big)^{\frac{1+\eta}{\halpha-1}}\frac{1}{y^2+scy}
    \dif y =
    \frac{1}{sc}\int_x^\infty \Big( \frac{y+sc}{y}
    \Big)^{\frac{1+\eta}{\halpha-1}-1}\frac{sc}{y^2}
    \dif y
    \\&=-\frac{1}{sc}
    \int_{y=x}^{y=\infty}
    z^{\frac{1+\eta}{\halpha-1}-1}
    \dif z
    =\Big[
    -
    \frac{\halpha-1}{sc(1+\eta)}
    \Big( \frac{y+sc}{y}
    \Big)^{\frac{1+\eta}{\halpha-1}}
    \Big]_{y=x}^{y=\infty}
    \\&=
    \frac{\halpha-1}{sc(1+\eta)}
    \Big( \frac{x+sc}{x}
    \Big)^{\frac{1+\eta}{\halpha-1}}
    -\frac{\halpha-1}{sc(1+\eta)}
    .
    \end{aligned}
\end{equation}
With the result of \eqref{eq:diff_equation_finalise_1.4} in mind, we find by partial integration
\begin{equation}
\begin{aligned}
\label{eq:diff_equation_finalise_2}
   \int_x^\infty 
    \Big(
    \frac{y+sc}{y}
    \Big)^{\frac{1+\eta}{\halpha-1}}
    \frac{
    c\frac{1+\eta}{\halpha-1} g^{(\varepsilon)}(y,0)}{y^2+scy}
    \dif y&=
    \frac{c}{sc}
    \Big(
    \frac{x+sc}{x}
    \Big)^{\frac{1+\eta}{\halpha-1}}g^{(\varepsilon)}(x,0)-\frac{cg^{(\varepsilon)}(\infty,0)}{sc}
    \\&\quad+
   \int_x^\infty \frac{c}{sc}\Big(
    \frac{y+sc}{y}
    \Big)^{\frac{1+\eta}{\halpha-1}}g^{(\varepsilon)}_x(y,0) \dif y.
    \end{aligned}
\end{equation}
Recalling that $g^{(\varepsilon)}(x,0)$ is a distribution in $x$, we find that $g^{(\varepsilon)}(\infty,0)=0$. 
Substituting the result of \eqref{eq:diff_equation_finalise_2} in \eqref{eq:diff_equation_finalise_C1} gives 
\begin{equation}
\begin{aligned}
    \hat{F}_{g^{(\varepsilon)}}(x,s)&=
     \Big(\frac{x}{x+sc}\Big)^{\frac{1+\eta}{\halpha-1}}
     \Big( 
    \int_x^\infty \Big(
    \frac{y+sc}{y}
    \Big)^{\frac{1+\eta}{\halpha-1}}g^{(\varepsilon)}_x(y,0)
    \Big(-\frac{c}{y+sc}+\frac{1}{s} \Big)\dif y\\&\qquad+
     \frac{1}{s}
   \Big(
    \frac{x+sc}{x}
    \Big)^{\frac{1+\eta}{\halpha-1}}g^{(\varepsilon)}(x,0)
     \Big)
     +C(s) \Big(\frac{x}{x+sc}\Big)^{\frac{1+\eta}{\halpha-1}}
     \\&=     
      \Big(\frac{x}{x+sc}\Big)^{\frac{1+\eta}{\halpha-1}} \frac{1}{s}\int_x^\infty \Big(
    \frac{y+sc}{y}
    \Big)^{\frac{1+\eta}{\alpha-1}-1}g^{(\varepsilon)}_x(y,0)
   \dif y
   \\&\qquad+\frac{1}{s}
      g^{(\varepsilon)}(x,0) +C(s) \Big(\frac{x}{x+sc}\Big)^{\frac{1+\eta}{\halpha-1}} .
         \end{aligned}
\end{equation}
We note that $g_x^{(\varepsilon)}(x,0)$ equals minus the density of a gamma distribution with shape $(1+\eta)/(\halpha-1)$ and rate $\varepsilon/c$. Substituting this density gives
\begin{equation}
\label{eq:diff_equation_finalise_C3}
\begin{aligned}
    &\frac{1}{s}g^{(\varepsilon)}(x,0) 
    - \Big(\frac{x}{x+sc}\Big)^{\frac{1+\eta}{\halpha-1}} \frac{1}{s}\int_x^\infty \Big(
    \frac{y+sc}{y}
    \Big)^{\frac{1+\eta}{\halpha-1}-1}
    \\&\quad \times \frac{(\varepsilon/c)^{\frac{1+\eta}{\halpha-1}}y^{\frac{1+\eta}{\halpha-1}-1}}{\Gamma((1+\eta)/(\halpha-1))} \e^{-\varepsilon y/c}
   \dif y +C(s) \Big(\frac{x}{x+sc}\Big)^{\frac{1+\eta}{\halpha-1}}\\& =
   \frac{1}{s}g^{(\varepsilon)}(x,0) 
    - \frac{\e^{s\varepsilon}}{s}\int_x^\infty \Big( \frac{\varepsilon}{c}\cdot
    \frac{y+sc}{x+sc}
    \Big)^{\frac{1+\eta}{\halpha-1}-1}\\&\quad \times
    \frac{x^{\frac{1+\eta}{\halpha-1}}}{\Gamma((1+\eta)/(\halpha-1))} \e^{-\varepsilon(y+sc)/c}\frac{\varepsilon \dif y}{c(x+sc)} \quad+C(s) \Big(\frac{x}{x+sc}\Big)^{\frac{1+\eta}{\halpha-1}}.
    \end{aligned}
\end{equation}
In the final step we rewrite the integral in \eqref{eq:diff_equation_finalise_C3}.
This is done with the integral substitution 
\begin{equation}
    \frac{\varepsilon}{c}\cdot
    \frac{y+sc}{x+sc}=u+\frac{\varepsilon}{c}\quad 
    \text{ so that, }
    \quad 
    \frac{\varepsilon \dif y}{c(x+sc)}=\dif u.
\end{equation}
By this substitution we find
\begin{equation}
\begin{aligned}
&  \frac{1}{s}g^{(\varepsilon)}(x,0) 
    - \frac{\e^{s\varepsilon}}{s}\int_0^\infty ( u+\varepsilon/c)
    )^{\frac{1+\eta}{\halpha-1}-1}
    \frac{x^{\frac{1+\eta}{\halpha-1}}\e^{-(u+\varepsilon/c)(x+sc)}}{\Gamma((1+\eta)/(\halpha-1))} \frac{\varepsilon \dif y}{c(x+sc)}\\& =
    \frac{1}{s}\prob(c\GammaD(1/(\halpha-1),\varepsilon)>x)
  -\int_0^\infty
    (u+\varepsilon/c)^{\frac{1+\eta}{\halpha-1}-1}\frac{ x^{\frac{1+\eta}{\halpha-1}}\e^{-x(u+\varepsilon/c)}}{\Gamma((1+\eta)/(\halpha-1))} \frac{\e^{-scu}}{s}\dif u.
\end{aligned}
\end{equation}
One can therefore write \eqref{eq:diff_equation_finalise_1} to
\begin{equation}
\begin{aligned}
    \label{eq:diff_equation_finalise_1.1}
   \hat F_{g^{(\varepsilon)}}(x,s)  &=
    \frac{1}{s}\prob(\GammaD( (1+\eta)/(\halpha-1),x)>\varepsilon/c)
  -\int_0^\infty \frac{x^{\frac{1+\eta}{\halpha-1}}}{\Gamma((1+\eta)/(\halpha-1))}  \\&\quad \times (u+\varepsilon/c)^{(1+\eta)/(\halpha-1)-1}\e^{-x(u+\varepsilon/c)}\frac{\e^{-sc u}}{s}\dif u +C(s) \Big( \frac{x}{x+sc}\Big)^{\frac{1+\eta}{\halpha-1}}.
    \end{aligned}
\end{equation}
This defines the transforms up to a constant that may dependent on $s$. In order to derive the constant, we use the fact that $x \mapsto g^{(\varepsilon)}(x,t)$ is a tail distribution with respective density $f_{L^{(\varepsilon)}_{\eta,\halpha}(t)}(x)$ to obtain
\begin{equation}
    \int_0^\infty 
    \e^{-s t}
    g^{(\varepsilon)}(x,t) \dif t 
    =
   \hat F_{g^{(\varepsilon)}}(x,s).
   \end{equation}
Thus, by differentiating with respect to $x$ on both sides,   
\begin{equation}
   -\int_0^\infty 
    \e^{-s t}
  f_{L^{(\varepsilon)}_{\eta,\halpha}(t)}(x)\dif t 
    =
   \frac{\partial}{\partial x} \hat F_{g^{(\varepsilon)}}(x,s).
   \end{equation}
Integrating over $x$ and interchanging the integrals gives   
   \begin{equation}
  - \int_0^\infty 
   \int_0^\infty 
    \e^{-s t}
    f_{L^{(\varepsilon)}_{\eta,\halpha}(t)}(x) \dif x\dif t 
     =
      \int_0^\infty 
       \frac{\partial}{\partial x} \hat F_{g^{(\varepsilon)}}(x,s) 
       \dif x.
\end{equation}
For the left-hand side that we note that $x\mapsto f_{L^{(\varepsilon)}_{\eta,\halpha}(t)}(x)$ is a density and must integrate to 1 and on the right-hand side we use the fundamental theorem of calculus. This results in the equality $  -1/s= C(s) -1/s$ and we conclude that $C(s)$ in \eqref{eq:diff_equation_finalise_1.1} equals $0$. 

Finally, it remains to show that the tail distribution of a gamma random variable with
shape $(1+\eta)/(\halpha-1)$ and rate $(t+\varepsilon)/c$, where we recall $c=c(\eta)$ from \eqref{eq:ceta}, has the same Laplace transform with respect to $t$ as \eqref{eq:diff_equation_finalise_1.1}. We use the following scaling property of the gamma distribution, that states, for $t>0$,
\begin{equation}
\label{eq:gamma_scaling}
\GammaD((1+\eta)/(\halpha-1),t) \overset{d}{=}(1/t)\GammaD((1+\eta)/(\halpha-1),1).
\end{equation}
This property then implies
\begin{equation}
\begin{aligned}
    \prob(c\GammaD(& (1+\eta)/(\halpha-1),t+\varepsilon)
    >x)
    \\&=\prob( c/(t+\varepsilon)\GammaD((1+\eta)/(\halpha-1),1)>x)\\&
    =\prob(\GammaD((1+\eta)/(\halpha-1),x)>(t+\varepsilon)/c).
    \end{aligned}
\end{equation}
Moreover, for the Laplace transform of the gamma tail distribution, this means
\begin{equation}
\begin{aligned}
    &\int_0^\infty \prob( c\GammaD( (1+\eta)/(\halpha-1), t+\varepsilon)>x) \e^{-st} \dif t\\&\quad =
    \int_0^\infty
    \prob( \GammaD( (1+\eta)/(\halpha-1), x)>(t+\varepsilon)/c)\e^{-st} \dif t
    \\&\quad =
    c\int_0^\infty
    \prob( \GammaD( (1+\eta)/(\halpha-1), x)> z+\varepsilon/c)\e^{-scz} \dif z.
    \end{aligned}
\end{equation}
By partial integration and the established Laplace transform of the gamma distribution, we can simplify this as
\begin{equation}
\begin{aligned}
   &-c\Big[ \frac{\e^{-scz}}{sc}
    \prob( \GammaD( (1+\eta)/(\halpha-1)), x)>z+\varepsilon/c)
    \Big]_{z=0}^{z=\infty}
    \\&\quad-c\int_0^\infty 
    \frac{x^{(1+\eta)/(\halpha-1)}(z+\varepsilon/c)^{(1+\eta)/(\halpha-1)-1}}{\Gamma((1+\eta)/(\halpha-1)) } \e^{-x(z+\varepsilon/c)}
    \frac{\e^{-scz}}{sc}\dif z
    \\&=
    \frac{1}{s}
    \prob( \GammaD( (1+\eta)/(\halpha-1)), x)>\varepsilon/c)\\&\quad
    -\int_0^\infty 
    \frac{x^{(1+\eta)/(\halpha-1)}(z+\varepsilon/c)^{(1+\eta)/(\halpha-1)-1}}{\Gamma((1+\eta)/(\halpha-1)) } \e^{-x(z+\varepsilon/c)}
    \frac{\e^{-scz}}{s}\dif z,
\end{aligned}
\end{equation}
which equals the right-hand side of  \eqref{eq:diff_equation_finalise_1.1} with $C(s)=0$.
Finally, uniqueness of Laplace transforms gives the desired result.

\subsection{Details on Tightness}

\subsubsection{Tightness of \mtitle{$(kV^*_{\lceil kt\rceil})_{t\geq0}$}}
\label{app:3-1}
To argue tightness, we use the Arzel\`a-Ascoli theorem, as  described in
	\cite[Theorem 16.8]{Billingsley1999} for the space $ D[\varepsilon,R]$ for any $0<\varepsilon<R<\infty$. This result is defined as follows:
	\begin{criterion}[Tightness condition in \protect{$D[\varepsilon,R]$}   ]
	\label{lem:Tightness_Crit_Arz_As}
	    Let $(X_k(t))_{t\geq 0}$ be a sequence of c\`adl\`ag stochastic processes and $R>\varepsilon$ be given. Then $(X(t))_{t\geq 0}$ is {\em tight} if and only if the following conditions are met:
	    	\begin{description}
		\item[Tightness of the supremum.] For every $\eta>0$, there exists an $a>0$ such that $$\underset{k\to\infty}{\emph{limsup}}
		\prob(\sup_{t\in [\vep,R]} |X_k(t)| >a)<
		\eta.
		$$
		\item[Modulus of continuity.] For every $r>0$ and $\eta>0$, there exists a $\delta>0$ such that
		$$\underset{k\to \infty}{\emph{limsup}}
		\prob(\bar{w}_{X_k}(\delta)>r)<\eta,$$
	where, for a process $x=(x(t))_{t\in [\vep,R]}$, $\bar{w}_x(\delta)$ is the c\`adl\`ag modulus of $x$ given by
		\begin{equation}
		\inf\limits_{(t_i)_{i\geq1}\in P_\delta}\
		\max\limits_{i\in[1,j]}\
		\sup\limits_{s,t\in[t_i,t_{i+1}]}
		|x(t)-x(s)|,
		\end{equation}
	and $P_\delta$ is the set of all partitions of the interval $[\varepsilon,R]$, where the increment lengths are at most $\delta$. 
	\end{description}
		\end{criterion}

    In this appendix, we show that the conditions of Criterion \ref{lem:Tightness_Crit_Arz_As} are met for $(k\tilde V_{\lceil kt \rceil}(\eta))_{t\geq 0}$ for $\alpha>2$, $\alpha\in(1,2)$ and for $( (\tilde \V_{\lceil kt\rceil} (\eta) )^{1/k})_{t\geq0}$ for $\alpha\in (0,1)$. 
    
    \begin{lemma}[Tightness of $(k\tilde V_{\lceil kt \rceil}(\eta))_{t\geq 0}$ for $\alpha>2$ and $\alpha\in(1,2)$]
    \label{lem:Appendix-Tightness-V_alpha>1}
    Recall the definition of 
	$\tilde V_{k}(\eta)$ in \eqref{eq:def_V_tilde}, for $\alpha>2$
	and $\alpha\in(1,2)$. Then the sequence $(k\tilde V_{\lceil kt \rceil}(\eta))_{t\geq 0}$ is tight.
	\end{lemma}
	\begin{proof}
    We note that $(k\tilde V_{\lceil kt \rceil}(\eta))_{t\geq 0}$
    monotone decreasing in $t$. Therefore, by pointwise convergence, the \textbf{Tightness of the supremum} condition follows. For \textbf{Modulus of continuity}, the proof is inspired by \cite[Corollary 6.3, step 2]{Michelen2019}, where the authors prove tightness for $\alpha>2$ and we next extend this to $\alpha\in(1,2)$. \\

	Consider an arbitrary $t_1, t_2$, such that $t_1-t_2=\delta$. Then we note that $k\tilde V_{\lceil kt_1\rceil}(\eta)-k\tilde V_{\lceil kt_2\rceil}(\eta)>0$ only if a jump has occurred within $k\delta$ time steps.
	Based on the result of Lemma~\ref{lem-lim-reduced-env-process}, we find that for some $C$, there exists a $k_0$ large enough such that for $k>k_0$, $\tilde V_{kt}(\eta) \leq C/k$ w.h.p.
	uniformly for $\eta\in(-1,1)$.
	Therefore jumps occur with at most probability $C/k$. We find
	\begin{equation}
	\begin{aligned}
	    \prob( \text{jump between $t_1$ and $t_2$ at depth $k$} ) &\leq 
	    \prob( \Binom( k\delta, C/k ) \geq 1)(1+o(1))\\
	    &= (1- (1-C/k)^{\delta k})(1+o(1)) \sim 1-\e^{-C\delta},
	\end{aligned}
	\end{equation}
	which can be made arbitrarily small by letting $\delta\to0$.
	\end{proof}
	We continue by showing tightness for $\alpha\in(0,1)$:
	\begin{lemma}[Tightness of $((\tilde\V_{\lceil kt \rceil}(\eta))^{1/k})_{t\geq 0}$ for $\alpha\in(0,1)$]
	\label{lem:Appendix-Tightness-V_alpha<1}
	Recall the definition of $\tilde\V_{k}(\eta)$, as given in \eqref{eq:def_V_tilde_alpha<1}, for $\alpha\in(0,1)$. Then the process $( (\tilde\V_{\lceil kt\rceil }(\eta))^{1/k})_{t>0}$ is tight.
	\end{lemma}
	\begin{proof}
	We again use the conditions of Criterion \ref{lem:Tightness_Crit_Arz_As} to show tightness.
	However, in this regime, the jump probability  is independent of the location. Therefore, it is expected that sequential jumps are close. However, we show that under the proposed scaling, these jumps are sufficiently small. 
	The process is again monotone decreasing in $t$. Therefore, by pointwise convergence, the \textbf{Tightness of the supremum} condition follows.
	For the \textbf{Modulus of continuity} condition, we
	let some small $s < 1 $ be given. Then we show that 
	\begin{equation}
	    \limsup_{k\to\infty} \prob( \bar w_{x}(\delta) > s) = o_\delta(1).
	\end{equation}
	We first upper bound $w_x(\delta)$, i.e.\ the supremum over increments of length $\delta$, and focus on one specific interval, say $[kt_0, k(t_0+\delta)]$ for $t_0<t$ and allowing for a jump at each of the $k\delta$ instances. Define $(U_i)_{i=1}^{k\delta}$ i.i.d. standard uniform random variables, then by definition of $\tilde \V_k(\eta)$ in \eqref{eq:def_V_tilde_alpha<1},
	\begin{equation}
 \begin{aligned}
	\label{eq:tightness_V_alpha<1_eq1}
	&\limsup_{k\to\infty} \prob( w_{x}(\delta) > s)\\&\qquad \leq 
	    \limsup_{k\to\infty} \prob\Big(
	    (\tilde\V_{kt_0}(\eta))^{1/k}-
	    (\tilde\V_{kt_0}(\eta) U_1^{\frac{1-\alpha}{\alpha(1+\eta)}} \cdots U_{k\delta}^{\frac{1-\alpha}{\alpha(1+\eta)}}  )^{1/k} >s\Big).
     \end{aligned}
	\end{equation}
	Reworking the terms and using the fact that $(\tilde\V_{kt_0}(\eta))^{1/k}$ converges to a deterministic limit, for $k$ large enough, there exists some $C$ such that, uniformly in $\eta\in[ -\varepsilon,\varepsilon]$, \eqref{eq:tightness_V_alpha<1_eq1} can be bounded by
	\begin{equation}
	\begin{aligned}
    \label{eq:tightness_V_alpha<1_eq2}
    &\limsup_{k\to\infty} \prob\bigg( C \bigg(1- \prod_{i=1}^{k\delta} U_i^{ \frac{1}{k} \frac{1-\alpha}{\alpha(1+\eta)}} \bigg) >s\bigg) \\&\qquad =
    \limsup_{k\to\infty} \prob\bigg(
    \frac{1-\alpha}{\alpha(1+\eta)}
    \frac{1}{k} \sum_{i=1}^{k\delta} (-\log(U_i))
     > -\log(1-s/\max\{1,C\})\bigg)\\
     &\qquad= \1\Big\{\delta \frac{1-\alpha}{\alpha(1+\eta)} >
     -\log(1-s/\max\{1,C\})\Big\},
	    	\end{aligned}
	\end{equation}
	where the last equality  is due to the law of the large numbers. Note that for $\delta$ small, this indicator becomes 0, which concludes the proof.
\end{proof}

\subsubsection{Tightness of \mtitle{$(M_{\lceil kt\rceil})_{t\geq 0}$}}
    	\label{App:tightness_volume>1}
Tightness is shown again via the Arzel\`a-Ascoli theorem, where its conditions are presented in Criterion \ref{lem:Tightness_Crit_Arz_As}.   	
    	
\begin{lemma}[Tightness of the  $(k^{-\gamma}M_{\lceil kt \rceil})_{t\geq 0}$ process for $\alpha>2$ and $\alpha\in (1,2)$]
\label{app:lem:TightnessofM}
 For $\alpha>2$ or $\alpha\in(1,2)$, the sequence $(k^{-\gamma}M_{\lceil kt \rceil})_{t\geq 0}$ is tight.
\end{lemma}
\begin{proof}
    We show tightness based on the conditions of Criterion \ref{lem:Tightness_Crit_Arz_As}. 
     The \textbf{Tightness of the supremum} condition follows from $k^{-\gamma}M_{\lceil kt\rceil}$ being monotone in $t$. As the process also converges pointwise to the Cox process, this condition follows directly.
     For the
    \textbf{Modulus of continuity} condition, consider  $\varepsilon<t_1<t_2<t$ such that $|t_1-t_2|<\delta$. We show that for any  $\varepsilon>0$, there exists a $\delta>0$ so that 
    \begin{equation}
        \prob_{[kt_1,kt_2]}\bigg(k^{-\gamma} \sum_{i=kt_1}^{kt_2} H_i >\varepsilon\bigg) <\eta, 
    \end{equation}
    which is exactly the difference between $k^{-\gamma}M_{kt_1}$ and $k^{-\gamma}M_{kt_2}$ and where we use the notation from \eqref{eq:def_E[kepsilon,kt]}.
    We show that this term is small by showing that is unlikely to find a single forest of size $O(k^\gamma)$ between backbone vertices  $kt_1$ and $kt_2$. Denote $\sum_{i=kt_1}^{kt_2} H_i =M_{[kt_1,kt_2]}$,
    and define the event 
    \begin{equation}
        E_{[kt_1,kt_2]} = \{ \exists j: H_j\geq \varepsilon k^{\gamma}, j\in[kt_1,kt_2])\}
    \end{equation}
    then, by conditioning on the existence of a large forest,
    \begin{equation}
    \begin{aligned}
       &  \prob_{[kt_1,kt_2]}(M_{[kt_1,kt_2]} >\varepsilon k^{\gamma})
        \\&\qquad = \prob_{[kt_1,kt_2]}(M_{[kt_1,kt_2]} >\varepsilon k^{\gamma}\mid E_{[kt_1,kt_2]}) \prob_{[kt_1,kt_2]}(  E_{[kt_1,kt_2]})\\&\qquad+
         \prob_{[kt_1,kt_2]}(M_{[kt_1,kt_2]} >\varepsilon k^{\gamma}\mid E_{[kt_1,kt_2]}^c) \prob_{[kt_1,kt_2]}(E_{[kt_1,kt_2]}^c)
         \\&\qquad\leq 
          \prob_{[kt_1,kt_2]}(E_{[kt_1,kt_2]}) + o(k),
    \end{aligned}
    \end{equation}
    where $\prob( E_{[kt_1,kt_2]}^c)=\prob(M_{[kt_1,kt_2]} >\varepsilon k^{\gamma}\mid \forall j: H_j < \varepsilon k^\gamma, j\in[kt_1,kt_2])$ converges to 0 based on \eqref{eq:small_volume_small_clusters1}
    and \eqref{eq:small_volume_small_clusters2}.
    By design, there are $\delta k$ instances where we can find a forest of size $O(k^\gamma)$. Therefore, we need to derive the probability that a specific forest is of that size. Consider $\ell<k$, then 
    \begin{equation}
    \begin{aligned}
        \prob_\ell(H_\ell > \varepsilon k^\gamma)
          &= 
           \prob_\ell(H_\ell > (\varepsilon (k/\ell)^\gamma)
           \ell^\gamma)=
           \frac{1}{k}
           (k/\ell)\int_
           {\varepsilon (k/\ell)^\gamma }
           ^\infty 
           h(y, a_\ell) 
           )\dif y.
    \end{aligned}
    \end{equation}
Note that in this case $\ell\in [kt_1,kt_2]$ and therefore
\begin{equation}
    \prob_\ell(H_\ell > \varepsilon k^\gamma)
    \leq 
    \frac{1}{k} 
    t_1^{-1} \int_{\varepsilon t_2^{-\gamma}}^\infty 
    h(y, t_1 L_{\halpha}^{\sss(k)}(t_1))\dif y
    =C^\dagger/k.
\end{equation}
The last equality follows as $h(y,a)$ converges exponentially fast to 0 for $y\to\infty$ for $\alpha>2$ and $\alpha\in(1,2)$. 
 Conditioned on $W_k$, these are independent events, so that
    \begin{equation}
        \prob_{[kt_1,kt_2]}( \exists j: H_j=O(k^2), j\in[kt_1,kt_2])
        \leq \prob( \Binom(\delta k,\ C^\dagger / k) = 0)
        =1-\Big( 1- \frac{C^{\dagger}}{k}\Big)^{\delta k},
    \end{equation}
    which can be made arbitrarily small for $\delta$ small uniformly in $k>1$.
\end{proof}

\subsection{\mtitle{$\lambda$} is a finite measure}
\label{app:sec_lambda_finite}
Consider $\lambda$ as given in \eqref{eq:def_lambda_PPP}, then we show $\lambda$ is indeed a finite measure. For this, it is sufficient to show that the limiting Cox process as given in \eqref{eq:Poisson_measure_goal} is finite. We do this via the Laplace transform $\phi(s)$ for the Cox process as given in \eqref{eq:volume_growth_alpha>2_eq10} and show that $\phi(s)\to 1$ if $s\to 0$, i.e. we show
\begin{equation}
    \int_0^t\int_0^\infty (1-\e^{-sy})
    \frac{h_\alpha(yx^{-\gamma},\varepsilon)}{x^{\gamma+1}}\dif y\dif x\to 0, \qquad s\to0,
\end{equation}
as $xL_{\halpha}(x)/(\halpha-1)$ is bounded with high probability.
As $h_\alpha$ differs per regime, we discuss $\alpha>2$ and $\alpha\in(1,2)$ separately 
\paragraphi{$\alpha>2$} Based on $h_\alpha$ in \eqref{eq:def_h_alpha>2}, we consider 
\begin{equation}
\begin{aligned}
        &\int_0^t\int_0^\infty (1-\e^{-sy})
    \frac{h_\alpha(yx^{-\gamma},1/\varepsilon)}{x^{\gamma+1}}\dif y\dif x \\&\qquad\leq  
    C
    \int_0^t     
\int_0^\infty (1-\e^{- sy})
y^{-3/2}
     \e^{-cyx^{-2}}\dif y \dif x 
\\&\qquad \leq 
  C
    \int_0^t     
\int_0^\infty s
y^{-1/2}
     \e^{-cyx^{-2}}\dif y \dif x,
    \end{aligned}
\end{equation}
where we used $(1-\e^{-x})\leq x$. Next, we perform a substitution of $u^2 = c_2yx^{-2}$ to find
\begin{equation}
\begin{aligned}
     C
    \int_0^t     
\int_0^\infty s
y^{-1/2}
     \e^{-cyx^{-2}}\dif y \dif x
     &=
      2Cs
    \int_0^t     
\int_0^\infty 
c^{-1/2}x\e^{-u^2}\dif u\dif x
\\&
=
 Csc^{-1/2}\sqrt{\pi}
 \int_0^t  x\dif x   
=o_s(1)
\end{aligned}
\end{equation}

\paragraphi{$\alpha\in(1,2)$} Based on $h_\alpha$ in \eqref{eq:def_h_alpha(1,2)} we consider
\begin{equation}
    \begin{aligned}
    &\alpha c_x p_c^{\frac{2-\alpha}{\alpha-1}}
    \int_0^t     
\int_0^\infty (1-\e^{- sy})
y^{\frac{1-2\alpha}{\alpha}}
       \int^{\infty}_0 
      \psi_\alpha\bigg(
     -\frac{z}{p_c }-\frac{1/\varepsilon}{ xy^{\frac{1-\alpha}{\alpha}} }\bigg) z^{1-\alpha} \dif z
     \dif y\dif x.
     \\& \leq 
     t\int_0^\infty\int_0^\infty (1-\e^{- sy})
y^{\frac{1-2\alpha}{\alpha}}
\psi_\alpha\bigg(
     -\frac{z}{ p_c }-\frac{1/\varepsilon}{ ty^{\frac{1-\alpha}{\alpha}} }\bigg) z^{1-\alpha} \dif z
     \dif y\\&
     \leq 
     t\int_0^\infty\int_0^\infty s
y^{\frac{1-\alpha}{\alpha}}
\psi_\alpha\bigg(
     -\frac{z}{p_c }-\frac{1/\varepsilon}{ ty^{\frac{1-\alpha}{\alpha}} }\bigg) z^{1-\alpha} \dif z
     \dif y
     =o_s(1),
    \end{aligned}
\end{equation}
as the density of a stable law $\psi_\alpha(x)$ of a totally skewed stable distribution admits stretched exponential tails for $x\to-\infty$. Moreover, as $(1-\alpha)/\alpha\in(-1,0)$ and $1-\alpha\in (-1,0)$, the integral exists for $z,y$ around 0.

\end{appendix}

\begin{acks}[Acknowledgments]
The authors thank the anonymous reviewer for their in-depth feedback and comments, which have significantly improved our paper.
\end{acks}

\begin{funding}
This work is supported by the Netherlands Organisation for Scientific Research (NWO) through Gravitation-grant NETWORKS-024.002.003. 
\end{funding}


\begin{supplement}
\stitle{Auxiliary and supportive details on the derivations}
\sdescription{We provide additional proofs on some claims in the paper. These proofs include the following.
We first extend the result of Lemma \ref{lem:deriv_theta} to the scaling of $\theta'(p)$ around $p=p_c$ and  $\theta^{-1}(x)$ around $x=0$. This result plays a role in the proof of Theorem \ref{thm-subthmWk}.
Then we prove the bounds for the jump sizes of the $(W_k)_{k\geq0}$ process as presented in \eqref{eq:bounds_jumpsize_alpha(1,2)} and \eqref{eq:couplingbound_alpha<1_2}, for  $\alpha\in(1,2)$ and $\alpha\in(0,1)$, respectively. This result is used in the proof of Theorem \ref{thm-subthmWk}.
We then show a partial summing identity to relate sums involving mass functions to tail distribution. This proves useful for many proofs in the paper, especially in Section \ref{sec:DBK}.
We continue by proving some auxiliary lemmas that extend convergence result to be uniform in $a$, which is paramount in Section \ref{sec:TP_scale}.
We conclude with the detailed proof of Lemma~\ref{lem:local_lims_upper_bound_mass_alpha(1,2)}.
}
\end{supplement}



\bibliographystyle{imsart-number} 
\bibliography{Zbib.bib}       

\end{document}